\documentclass{article}

\title{Large Deviations and Exit-times for reflected McKean-Vlasov equations with self-stabilizing terms and superlinear drifts}

\author{
\normalsize Daniel Adams\textit{$^{a}$}\footnote{D.A was supported by The Maxwell Institute Graduate School in Analysis and its Applications, a Centre for Doctoral Training funded by the UK Engineering and Physical Sciences Research Council (grant EP/L016508/01), the Scottish Funding Council, Heriot-Watt University and the University of Edinburgh. } \\
        \small  d.t.s.adams@sms.ed.ac.uk 
\and
\normalsize Gon\c calo dos Reis\textit{$^{b,c,}$}\footnote{G.d.R. acknowledges support from the \emph{Funda{\c c}$\tilde{\text{a}}$o para a Ci$\hat{e}$ncia e a Tecnologia} (Portuguese Foundation for Science and Technology) through the project UIDB/00297/2020 (Centro de Matem\'atica e Aplica\c c$\tilde{\text{o}}$es CMA/FCT/UNL).} \\
        \small  G.dosReis@ed.ac.uk
\and
\normalsize Romain Ravaille\textit{$^{d,}$}\footnote{R.R. acknowledges support from the Labex MiLyon and from ANR-19-CE40-0009} \\
        \small  r.ravaille@univ-st-etienne.fr 
\and
\normalsize William Salkeld\textit{$^{b,}$}\footnote{W.S. is grateful to the Laura Wisewell Travel fund that facilitated this project. } \\
        \small  w.j.salkeld@sms.ed.ac.uk 
\and 
\normalsize Julian Tugaut\textit{$^{d,}$}\footnote{J.T. acknowledges support from the Labex MiLyon and from ANR-19-CE40-0009} \\
          \small  julian.tugaut@univ-st-etienne.fr
}

\usepackage[utf8]{inputenc}
\usepackage[T1]{fontenc}
\usepackage[english]{babel}
\usepackage{charter}

\usepackage{verbatim}
\usepackage{mathtools}

\usepackage{graphicx}
\usepackage{bbm}
\usepackage{framed}
\usepackage[normalem]{ulem}
\usepackage{amsthm}
\usepackage{amsmath}
\usepackage{amssymb}
\usepackage{amsfonts}
\usepackage{dsfont}
\usepackage{enumerate}
\usepackage{xcolor}
\usepackage{stmaryrd} 
\usepackage[top=1 in,bottom=1in, left=1 in, right=1 in]{geometry}
\usepackage{csquotes}
\usepackage[toc,page]{appendix}

\usepackage{hyperref}

\usepackage{soul}

\theoremstyle{definition}
\newtheorem{theorem}{Theorem}[section]
\newtheorem{lemma}[theorem]{Lemma}

\newtheorem{corollary}[theorem]{Corollary}
\newtheorem{defn}[theorem]{Definition}
\newtheorem{example}[theorem]{Example}
\newtheorem{rem}[theorem]{Remark}
\newtheorem{prop}[theorem]{Proposition}

\newtheorem{assumption}[theorem]{Assumption}

\numberwithin{equation}{section}
\numberwithin{figure}{section}
 \usepackage[nodayofweek]{datetime}

\usepackage{pifont}
\usepackage{
            ulem		
} 

\newcommand{\bB}{\mathbb{B}}

\newcommand{\bE}{\mathbb{E}}

\newcommand{\bP}{\mathbb{P}}

\newcommand{\bR}{\mathbb{R}}

\newcommand{\bN}{\mathbb{N}}

\newcommand{\bW}{\mathbb{W}}
\newcommand{\bZ}{\mathbb{Z}}

\newcommand{\cB}{\mathcal{B}}

\newcommand{\cD}{\mathcal{D}}

\newcommand{\cF}{\mathcal{F}}
\newcommand{\cG}{\mathcal{G}}
\newcommand{\cH}{\mathcal{H}}

\newcommand{\cK}{\mathcal{K}}

\newcommand{\cN}{\mathcal{N}}

\newcommand{\cP}{\mathcal{P}}

\newcommand{\cS}{\mathcal{S}}
\newcommand{\cT}{\mathcal{T}}

\newcommand{\fD}{\mathfrak{D}}

\newcommand{\1}{\mathbbm{1}}

\newcommand{\<}{\langle}
\renewcommand{\>}{\rangle}
\newcommand{\PP}{\mathbb{P}}
\newcommand{\EE}{\mathbb{E}}

\newcommand{\n}{\textbf{n}}

\date{%
    \footnotesize 
    $^{a}$~Maxwell Institute for Mathematical Sciences School of Mathematics University of Edinburgh Edinburgh UK EH9 3FD
    \\
    $^{b}$~School of Mathematics, University of Edinburgh, The King's Buildings, Edinburgh, UK
    \\
    $^{c}$~Centro de Matem\'atica e Aplica\c c$\tilde{\text{o}}$es (CMA), FCT, UNL, Portugal
    \\
    $^{d}$~Universit\'e Jean Monnet, Institut Camille Jordan, 23 Rue du Docteur Paul Michelon, 42023 Saint-\'Etienne, France
    \\[2ex]
    \longdate \today \ (\currenttime)
    \vspace{-0.8cm}
}

\usepackage[nobysame
           ,alphabetic
           ,numeric
           ,initials
           ]{amsrefs}

\begin{document}
\selectlanguage{english}

\maketitle
\begin{abstract}
We study reflected McKean-Vlasov diffusions over a convex, non-bounded domain with self-stabilizing coefficients that do not satisfy the classical Wasserstein Lipschitz condition. We establish existence and uniqueness results for this class and address the propagation of chaos. Our results are of wider interest: without the McKean-Vlasov component they extend reflected SDE theory, and without the reflective term they extend the McKean-Vlasov theory.

We prove a Freidlin-Wentzell type Large Deviations Principle and an Eyring-Kramer's law for the exit-time from subdomains contained in the interior of the reflecting domain. Our characterization of the rate function for the exit-time distribution is explicit.
\end{abstract}

\textbf{Keywords:} reflected McKean-Vlasov equations, Self-stabilizing diffusions, Super-linear growth, Freidlin-Wentzell Large Deviations Principle, Eyring-Kramer Law

\textbf{MSC 2010 subject classifications:} Primary 60F10; secondary 60G07

\section{Introduction}

In this article we study $\bR^d$-valued \textit{Stochastic Differential Equation}s (SDE)  whose dynamics are confined to a subset $\cD \subset \bR^d$, namely, the solution $X_t$ is repelled away from the boundary $\partial \cD$ by a reflection mechanism defined in terms of the outward normal and a local time at the boundary. These \textit{reflected SDEs}, enable one to model an impenetrable frontier at which the process is ``constrained'' and have advanced as a rich field within the applied probability theory. They are used to model physical transport processes \cite{costantini1991diffusion}, molecular dynamics \cite{saisho1994model}, biological systems \cites{dangerfield2012modeling,niu2016modelling} and appear in mathematical finance \cite{HanHuLee2016} and stochastic control  \cites{kruk2000optimal,ramasubramanian2006insurance}. Lastly, this reflection problem, the so-called \textit{Skorokhod problem} \cites{Skorokhod1961stochastic,Skorokhod1962stochastic}, has also proven particularly useful in analysing a variety of queuing and communication networks. The literature on the latter is vast, see \cites{ward2003diffusion,ramanan2003fluid} or \cites{chen2013fundamentals}.  

In this work, we focus on the general class of  \textit{reflected McKean-Vlasov equations}
\begin{equation}
\label{eq:MVE}
\begin{split}
X_t^{i} =& X_0 +\int_0^t b(s, X_s^{i}, \mu_s)ds + \int_0^t f\ast \mu_s(X_s^{i}) ds + \int_0^t \sigma(s, X_s^{i}, \mu_s) dW_s^{i} - k_t^{i}, 
\\
|k^i|_t=& \int_{0}^{t} \1_{ \partial\cD }(X_s^i) d|k^i|_s, 
\qquad 
k_t^i= \int_{0}^{t} \1_{ \partial\cD}(X_s^i)\n (X_s^i) d|k^i|_s,
\qquad
\mu_t(dx) = \bP\big[ X_t^i \in dx\big],
\end{split}
\end{equation}
where $\n$ is a vector field on the boundary of the domain $\cD$ in an outward normal direction, $W$ is a Brownian motion and $k$ is a bounded variation process with variation $|k|$ acting as a local time that constrains the process to the domain $\cD$. Thus, the instant the path attains the boundary $\partial \cD$ of the domain, $k$ increases creating a contribution that ensures the path remains inside the domain. $\mu$ is the law of the solution process $X$ and the coefficients $b$ and $f$ are locally Lipschitz over the domain $\cD$. We denote by $f\ast\mu(\cdot)$ the convolution of a function $f$ with the measure $\mu$. 

The law of the above diffusion solves the nonlinear Fokker-Planck equation with a Neumann boundary condition (see also \cite{wang2021distribution}), formally
\begin{equation}
\label{eq:FokkerPlanckEquation}
\begin{split}
& \partial_t \mu_t(x) = \nabla \cdot \Big( \tfrac{1}{2} \nabla^T \cdot (\sigma \cdot \sigma^T)(t, x, \mu_t) \mu_t(x) - b(s, x, \mu_t)\mu_t(x) - f \ast \mu_t(x) \mu_t(x) \Big)
\\
& \Big\langle \n(x),  \tfrac{1}{2} \nabla^T \cdot (\sigma\cdot \sigma^T) (t, x, \mu_t) \mu_t(x) - b(t, x, \mu_t) \mu_t(x) - f \ast \mu_t(x)\mu_t(x) \Big\rangle=0 \quad \forall x\in \partial \cD.
\end{split}
\end{equation}

It is widely known that McKean-Vlasov equations arise as the mean field limit of a system of interacting particles, the so-called \textit{Propagation of Chaos} (PoC): for $N\in \bN$ and $i\in\{1, ..., N\}$, the system of equations
\begin{equation}
\label{eq:ParticleSystem}
\begin{split}
X_t^{i, N} =& X_0 + \int_0^t b( s, X_s^{i, N}, \mu_s^N ) ds + \int_0^t f\ast \mu_s^N (X_s^{i, N}) ds + \int_0^t \sigma( s, X_s^{i, N}, \mu_s^N ) dW_s^{i, N} - k_t^{i, N}, 
\\
|k^{i, N}|_t=& \int_{0}^{t} \1_{ \partial\cD }(X_s^{i, N}) d|k^{i, N}|_s,
\qquad 
k_t^{i, N}= \int_{0}^{t} \1_{ \partial\cD}(X_s^{i, N}) \n (X_s^{i, N}) d|k^{i, N}|_s, 
\qquad
\mu_t^N = \tfrac{1}{N} \sum_{j=1}^N \delta_{X_t^{j, N}},
\end{split}
\end{equation}
has a dynamics that converges as $N\to \infty$ to that of Equation \eqref{eq:MVE},  

The problem of confining a stochastic process to a domain was first posed by Skorokhod in \cite{Skorokhod1961stochastic}. The seminal works \cite{tanaka2002stochastic}, \cite{lions1984stochastic} and \cite{saisho1987stochastic} prove that such solutions exist and are unique in the multi-dimensional case for different classes of domain.  \cite{tanaka2002stochastic} works with processes on a convex domain while \cite{saisho1987stochastic} studies domains that satisfy a ``Uniform Exterior Sphere'' and ``Uniform Interior Cone'' condition but imposes more restrictive assumptions on the equation's coefficients. \cite{sznitman1984nonlinear} was the first to prove wellposedness of reflected McKean-Vlasov equations in smooth bounded domains. The above works impose strong restrictions on the coefficients, usually requiring that they are Lipschitz and bounded. We prove the existence and uniqueness for a broader class of McKean-Vlasov reflected SDE in general convex domains, crucially not requiring global Lipschitz continuity, nor bounded coefficients, nor a bounded domain. We allow for superlinear growth components in both space and in the convolution component (the measure component).  
Very recently, \cite{wang2021distribution} contributes new wellposedness results under singular coefficients and establishes exponential ergodicity under a variety of conditions.

In this work we focus on reflections according to an outward normal of the solution's path as $X_t\in\partial\cD$, but other types of reflections exist. \textit{Oblique reflected SDEs} are reflected SDEs where the vector field $\n$ is not normal but oblique to the boundary. Wellposedness is studied in \cites{lions1984stochastic,anderson1976small} and in \cites{costantini1992skorohod,dupuis1993sdes} for non-smooth domains. \textit{Elastic reflections} appears in \cite{Spiliopoulos2007ReflectedAndLangevin}. 
A recently introduced form of reflections motivated by financial applications,  see \cite{briand2018bsdes}, is the  \textit{reflection in mean} where the reflection happens at the level of the distribution and is generally weaker than the classical pathwise constraint. A typical mean reflection constraint asks for the expected value (of a given function of the solution) to be non-negative, e.g.~$\bE[h(X_t)]\geq 0$.  See \cite{briand2016particles} for a particle system approximation of mean reflected SDE and its numerics. The particle system approximations are similar to the classical McKean-Vlasov setting. Lastly, a Large Deviation Principle for mean reflected SDE is achieved in \cite{li2018large} while the exit-time problem, in the likes of our study in Section \ref{sec:ExitTimes} below, is open. 

\subsubsection*{Large Deviations and Exit-times}

The second part of this work focuses in obtaining a \textit{Large Deviations Principle} and the characterisation of the exit-time from a subdomain $\fD\subsetneq \cD$ for the small noise limit for the reflected McKean-Vlasov equation
\begin{equation}
\label{eq:MVELimiting}
\begin{split}
X_t^\varepsilon 
&= X_0 + \int_0^t b(s, X_s^\varepsilon, \mu_s^\varepsilon )ds 
       + \int_0^t f \ast \mu_s^\varepsilon (X_s^\varepsilon ) ds 
       + \sqrt{\varepsilon}\int_0^t \sigma(s, X_s^\varepsilon, \mu_s^\varepsilon) dW_s - k_t^\varepsilon,
\\
|k^{\varepsilon}|_t
&= \int_{0}^{t} \1_{ \partial\cD }(X_s^{\varepsilon}) d|k^{\varepsilon}|_s, 
\qquad 
k_t^{\varepsilon}= \int_{0}^{t} \1_{ \partial\cD}(X_s^{\varepsilon}) \n (X_s^{\varepsilon}) d|k^{\varepsilon}|_s,
\qquad 
\mu_t^\varepsilon(dx) = \bP\big[ X_t^\varepsilon \in dx\big].
\end{split}
\end{equation}

The asymptotic theory of Large Deviations Principles (LDP) \cite{DZ} quantifies the rate of convergence for the probability of rare events. First developed by Schilder in \cite{schilder1966some}, an LDP is equivalent to convergence in probability with the addition that the rate of convergence is a specific speed controlled by the rate function. Consider a drift term $b$ that has some basin of attraction and assume the noise in our system is small. Under such conditions, it is common for the system to exhibit a meta-stable behaviour. Loosely speaking, this terminology refers to when a particle is forced towards a basin of attraction and spends long periods of time there before moving to the next basin of attraction. The particle only leaves after receiving a large "kick" from its noise which in the small noise limit, i.e., as the noise vanishes, is an increasingly rare event. This property of the dynamics poses a difficulty for numerical simulations since the numerical scheme takes an impractical amount of time to observe any deviations from the basin.  LDPs help by quantifying the probability of this rare event. 

A Freidlin–Wentzell LDP provides an estimate for the probability that the sample path of an It\^o diffusion will stray far from the mean path when the size of the driving Brownian motion is small with respect to a pathspace norm. Freidlin-Wentzell LDPs for reflected SDEs have been explored in a number of works. For bounded and Lipschitz coefficients, \cite{dupuis1987large} provides the LDP in general convex domains. For smooth domains, \cite{anderson1976small} obtains the LDP under the assumption of bounded and Lipschitz coefficients. Additional references on LDPs for reflected processes can be found in \cite{priouret1982remarques}.

Close to our work is \cite{liu2020large} where large and moderate deviations for non-reflected McKean-Vlasov equations with jumps is addressed via the Dupuis-Ellis weak convergence framework \cite{dupuis2011weak}. Their comprehensive wellposedness results \cite{liu2020large}*{Proposition 5.3} are established under a uniformly Lipschitz measure assumption on the coefficients (their assumption A1 and A2) while here we allow for fully super-linear growth in both measure and space components. 

LDPs are a suitable language for studying the rare event of exiting from a basin of attraction. For classical reflected SDEs the exit-time from a subdomain $\fD\subsetneq \cD$  is a trivial problem as one exits the subdomain $\fD$ before hitting the boundary of $\cD$, and hence,  the exit-time result for $\fD$ is recovered from standard SDE counterpart. This is a priori \textit{not} the case for reflected McKean-Vlasov equations where the reflection term affects the law and paths to ensure it remains on the domain and is thus different from the law of the non-reflected McKean-Vlasov. 

In the small noise limit the exit-problem for non-reflected SDEs is well documented. A great introduction to the subject can be found in \cite{DZ}*{Section 5.7}; for an in-depth study with slowly-varying time-dependent coefficients see \cite{Herrmann2013StochasticR}*{Section 4}; the excellent work \cite{HIP} characterises the exit-time of a McKean-Vlasov equation after obtaining a large deviation principle; see \cite{tugaut2016simple} for a simpler proof relying only on classical Freidlin-Wentzell estimates; and \cite{T2011f}, where the same results are obtained by transference from the particle system to the McKean-Vlasov system via propagation of chaos and Freidlin-Wentzell estimates.

\subsubsection*{Our motivation and contributions}

Our \textit{contributions} are threefold: (i) existence and uniqueness results for McKean-Vlasov SDEs constrained to a convex domain $\cD\subseteq \bR^d$ with coefficients that have superlinear growth in space and are non-Lipschitz in measure; (ii) a large deviations principle for this class of processes; and, (iii) the explicit characterisation of the first exit-time of the solution process from a subdomain $\fD\subsetneq \cD$.

For (i), unlike previous works on reflected SDEs, we do not rely on the domain as a way of ensuring the coefficients are bounded or Lipschitz. We work with drift terms that satisfy a one-sided Lipschitz condition over the (possibly unbounded) domain and are locally Lipschitz. Further, we do not restrict ourselves to measure dependencies that are Lipschitz on the domain, but additionally work with a drift term that satisfies a self-stabilizing assumption that ensures any particle is attracted towards the mean of the distribution/particle system. Critically, in a convex domain this will always be away from the boundary. 

From a technical point of view, the non-Lipschitz measure component, $f$ in \eqref{eq:MVE}, destroys the standard contraction argument. Nonetheless, we are able to establish an intermediate fixed point argument which decouples $f$, leaving $b$ to be dealt with. The main workaround result is Lemma \ref{lemma:Gamma-FirstContraction} in combination with a specific moment estimate mechanism. The closest result to ours is that of \cite{HIP}. There, specific structural assumptions are required: drift of specific polynomial form, $\sigma$ is constant, no-time dependencies, deterministic coefficients and, critically, $b$ and $f$ need to be combined into a mean-field interaction term of order $1$. We lift all these constraints. 

To the best of our knowledge, the scope of our well-posedness results for McKean-Vlasov equations, and separately for reflected SDEs, are not found in the literature. Thus, our contributions extend known results for McKean-Vlasov equations and reflected SDEs.

For (ii), our study of the LDPs is based on techniques which directly address the presence of the law in the coefficients and avoid the associated particle system. Methodologically, our approach relies on the classical mechanism of exponentially good approximations but employing judiciously chosen auxiliary processes and less standard tricks to obtain the main results. As in \cite{dos2019freidlin}, it turns out that the correct LDP rate function for McKean-Vlasov equations can be found through certain ODE equations (skeletons) where the McKean-Vlasov's noise and distributions are replaced by smooth functions and the degenerate distribution corresponding to the ODE's solution respectively. 

For (iii), the LDP results are the intermediate step necessary to study the exit-time of $X^\varepsilon$ from an open subdomain $\fD\subsetneq \cD$. Motivated by numerical applications, as in  \cites{di2017jump,di2019sharp}, we provide the \textit{explicit form of the rate function} for the exit-time distribution (the exit-cost $\Delta$ in Theorem \ref{thm:ExitTime}). 

Intuitively, the solution to \eqref{eq:MVELimiting} depends on its own law, hence one expects its exit-time from a subdomain to differ from the exit-time of its non-reflected analogue. Similarly, the exit-time of one of the particles in the system \eqref{eq:ParticleSystem} will be altered by the presence of the reflection since this particle will interact with other particles which have already been reflected. However, we will show that, in the small noise limit the exit-time of our McKean-Vlasov reflected SDE is unaltered and we are able to establish a familiar Eyring-Kramer's type law. 

The \textit{motivation} of our work stems from numerical considerations around the simulation of McKean-Vlasov equations (reflected or not) where the measure component is non-Lipschitz, in finite and infinite time horizon, and non-constant diffusion coefficients.  For instance, reflected McKean-Vlasov equations appear in \cite{LeiteWilliams2019} and \cite{Anderson2019} as models for bio-chemistry and our framework allows us to study the Granular media equation (see \eqref{eq:FokkerPlanckEquation})
$$
\partial_t \mu_t(x) = \tfrac{1}{2} \nabla^2 \mu_t(x) + \nabla\cdot \Big( \nabla B(x)\mu_t(x) + \nabla F \ast \mu_t(x) \mu_t(x) \Big),
$$
where $B$ is the constraining potential and $F$ is the interactive potential. This models the velocity distribution in the hydrodynamic limit of a collection of inelastic particles. In the case where the potentials $B$ and $F$ are convex, it is well known that the solution rapidly converges (as $t\to \infty$) towards an invariant distribution \cite{BGG1}. Our work opens a clear pathway to analyse the behaviour of \eqref{eq:MVE} and \eqref{eq:ParticleSystem} as $t\to \infty$.

An important and fully unanswered question left open by this work relates to effective numerical methods for this class of McKean-Vlasov equationss (even in the non-reflected case). On one hand the penalisation methodology of \cite{Slominski2013-rSDE-penalization} seem feasible, where the reflection on the bounded domain enforces boundedness of the solution process and the compact support of its law (a trick exploited in \cite{bouchard2017numerical}). On the other hand, explicit step Euler-type discretizations \cite{dos2018simulation} for super-linear drifts have been shown to work but only for drifts that are Lipschitz in the measure components. 

\medskip
\textit{This work is organised as follows.} Section \ref{sec:Preliminaries} introduces notation, setting and objects of interest. In Section \ref{section exist.unique} we address the wellposedness of the reflected McKean-Vlasov equations, of the associated reflected interacting particle system and present a Propagation of Chaos result. Sections \ref{sec:LDPs} and \ref{sec:ExitTimes} cover the Freidlin-Wentzell Large deviations and exit-time results respectively.

\section{Preliminaries}
\label{sec:Preliminaries}

We denote by $\bN=\{1,2,\cdots\}$ the set of natural numbers; $\bZ$ and $\bR$ denote the set of integers and real numbers respectively, with the real positive half-line set as $\bR^+=[0,\infty)$. For $t\in\bR$, we denote its floor as $\lfloor t \rfloor$ (the largest integer less than or equal to $t$). For any $x,y\in\bR^d$, $\langle x,y\rangle$ stands for the usual Euclidean inner product and $\|x\|=\langle x,x\rangle^{1/2}$ the usual Euclidean distance. Let $A$ be a $d\times d'$ matrix, we denote the transpose of $A$ by $A'$  and let $\| A \|$ be the Hilbert-Schmidt norm. Define the derivative of a function $f:\bR\to \bR^d$ as $\dot{f}$.

For sequences $(f_n)_{n\in \bN}$ and $(g_n)_{n\in\bN}$, we use the symbols $\lesssim ,\gtrsim $ in the following way: 
\begin{align*}
f_n \lesssim g_n \ \ \iff  \ \ \limsup_{n\to \infty} \frac{f_n}{g_n}\leq C,~\text{for some}~C>0,
\end{align*}
and
\begin{align*}
f_n \gtrsim g_n \ \ \iff  \ \ \liminf_{n\to \infty} \frac{f_n}{g_n}\geq C,~\text{for some}~C>0. 
\end{align*}
For a set $\cD \subset \bR^d$, we denote its interior (largest open subset) by $\cD^\circ$, its closure (smallest closed cover) by $\overline{\cD}$ and the boundary by $\partial \cD = \overline{\cD} \backslash \cD^{\circ}$. For $x\in \bR^d$,$r\geq 0$, denote $B_r(x)\subset \bR^d$ as the open ball of radius $r$ centred at $x$.

Let $f:\bR^d \to \bR$ be a differentiable function. Then we denote by $\nabla f$ the gradient operator and $\nabla^2 f$ to be the Hessian operator. Let $C([0,T]; \bR^d)$ be the space of continuous function $f:[0,T] \to \bR^d$ endowed with the supremum norm $\|\cdot\|_{\infty,[0,T]}$. For $x\in \mathbb{R}^d$ let $C_{x}([0,T]; \bR^d)$ be the subspace of $C([0,T]; \bR^d)$ of functions $f:[0,T] \to \bR^d$ with $f(0)=x$.

Let $\tilde{\Omega}=C_0([0,T]; \bR^{d'})$ be the canonical $d'$-dimensional Wiener space and let $W$ be the Wiener process with law $\tilde{\bP}$. Let $(\tilde{\cF_t})_{t\in[0,T]}$ be the standard augmentation of the filtration generated by the Brownian motion. Then we have the probability space $(\tilde{\Omega}, \tilde{\cF}, (\tilde{\cF_t})_{t\in[0,T]}, \tilde{\bP})$. Additionally, let $([0,1], \cB([0,1]), \overline{\bP})$ be a probability space with the Lebesgue measure $\overline{\bP}$. Our probability space is structured as follows:
\begin{enumerate}
\item The sample space will be $\Omega =[0,1] \times \tilde{\Omega}$,
\item The $\sigma$-algebra over this space will be $\cF = \sigma( \cB([0,1]) \times \tilde{\cF})$ with filtration $\cF_t = \sigma( \cB([0,1]) \times \tilde{\cF_t})$, 
\item The probability measure will be the product measure $\bP = \overline{\bP} \times \tilde{\bP}$. 
\end{enumerate}

For $p\geq 1$, let $L^p(\Omega, \cF, \bP; \cD)$ be the space of random variables over the probability space $(\Omega, \cF, \bP)$ with state space $\cD$ and finite $p$ moments.  For $p\geq 1$, let $\cS^p([0,T];\bR^d)$ be the space of  $(\tilde{\cF_t})_{t\in[0,T]}$-adapted processes $X:\Omega\times [0,T]\to \cD$ satisfying $\bE[ \|X\|^p_{\infty, [0,T]} ]^{1/p} < \infty$ where $\|X\|_{\infty,[0,T]}:=\sup_{s\in[0,T]} \| X_s \|$.

Let $\cH_1^0$ be the Cameron Martin Hilbert space for Brownian motion: the space of all absolutely continuous paths on the interval $[0, T]$ which start at $0$ and have a derivative almost everywhere which is $L^{2}([0, T]; \bR^{d'})$ integrable
$$
\cH_1^0:=\big\{h:[0, T] \to \bR^{d'},\ h(0)=0,\  h(\cdot)=\int_{0}^{\cdot} \dot{h}(s) d s,\ \dot{h} \in L^{2}([0, T]; \bR^{d'} )\big\}.
$$

Let $\cD$ (possibly unbounded) be a subset of $\bR^d$ and $\cB_\cD$ be the Borel $\sigma$-algebra over $\cD$. Let $\cP_r(\cD)$ be the set of all Borel probability measures which have finite $r^{th}$ moment. 
\begin{defn}
\label{defn:Wasserstein}
Let $r\geq 1$. Let $(\cD, d)$ be a metric space with Borel $\sigma$-algebra $\cB_\cD$. Let $\mu, \nu \in \cP_r(\cD)$. We define the Wasserstein $r$-distance $\bW_\cD^{(r)}: \cP_r(\cD) \times \cP_r(\cD) \to \bR^+$ to be

$$
\bW_\cD^{(r)} (\mu, \nu) = \Big( \inf_{\pi \in \Pi_r(\mu,\nu)} \int_{\cD \times \cD} d(x, y)^r \pi(dx, dy) \Big)^{\frac{1}{r}},
$$

where $\Pi_r(\mu,\nu)\subset \cP_r(\cD \times \cD)$ is the space of joint distributions over $\cD \times\cD$ with marginals $\mu$ and $\nu$. 
\end{defn}

\subsubsection*{Domain, outward normal vectors and properties}

The processes that we consider in this paper are confined to a domain $\cD$. 
\begin{defn}
Let $\cD$ be a subset of $\bR^d$ that has non-zero Lebesgue measure interior. For $x\in \partial \cD$, define
\begin{align*}
\cN_{x, r}:=& \{ \n\in \bR^d: \|\n\|=1, B_r(x+r\n)\cap \cD^{\circ} = \emptyset\}
\quad\textrm{and}\quad
\cN_x:= \cup_{r>0} \cN_{x, r}.
\end{align*}
We call the set $\cN_x$ the outward normal vectors. 
\end{defn}
For general domains, the set $\cN_x$ can be empty, for example if the boundary contains a concave corner. Furthermore if the boundary is not smooth at $x$ then it may be the case that $| \cN_{x,r}| = \infty$.  

\begin{defn}
Let $\cD\subset \bR^d$ with non-zero Lebesgue measure interior. We say that $\cD$ has a \emph{Uniform Exterior Sphere} if $\exists r_0>0$ such that $\forall x\in \partial \cD$, $\cN_{x, r_0} \neq \emptyset$. 
\end{defn}

The existence of a uniform exterior sphere ensures there is at least one outward normal vector at every point on the boundary. When this is not the case, there is no canonical choice for the reflective vector field. The following property of convex domains will be used extensively throughout this paper.

\begin{lemma}\label{lem:NormalToDomain}
Let $\cD\subset \bR^d$ be a convex domain with interior that has non-zero Lebesgue measure. Then $\cD$ has a Uniform Exterior Sphere, and for any $x\in \partial \cD$ and  $\n(x)\in \cN_{x}$ it holds that 
\begin{equation}\label{equation uniform exterior sphere property}
\langle \n(x), y-x\rangle \leq 0,~\forall y\in \cD.
\end{equation}

\begin{proof}
First we prove that $\cD$ has a Uniform Exterior Sphere. Let $r>0$ be fixed and let $x\in \partial \cD$. If $\cD$ is a convex subspace of $\bR^d$, then there exists a semi-plane $(\cS)$ which contains $\cD$. Thus we have a hyperplane $\cH_x$ that contains $x$ and $\cD^{\circ}\cap \cH_x=\emptyset$. Then, $\exists \n$ such that $\forall y\in \cH_x$ we have $\< y, \n\>=0$. Without loss of generality, $\n$ can be chosen to be an exiting vector from $\cD$. Consider the open ball $B_r (x + r\n)$. This is an open set contained in the complement of the closed semi-plane ($\cS^c$). Thus $B_r(x + r\n) \cap \cD^{\circ} = \emptyset$. Hence $\cN_{x, r}\neq \emptyset$. Now we show \eqref{equation uniform exterior sphere property}, For $x\in \partial \cD$, we have just shown that a vector $\n(x) \in \cN_{x}$ exists. Further, $\exists r>0$ such that $\n\in \cN_{x, r}$ and denote $z = x+r\n(x)$. Then
$$
\inf_{y\in \cD} \| z-y\| = \| z-x\|. 
$$
If this is not the case the ball of radius $r$ centred at $y$ would intersect with the $\cD^{\circ}$ and hence
\begin{align*}
\| (x-z) + (y-x)\| \geq& \| z-x\|
\quad \Rightarrow \quad \langle x-z, y-x\rangle \geq 0,
\end{align*}
rearranging this yields that \eqref{equation uniform exterior sphere property}.
\end{proof}
\end{lemma}
 
Motivated by this Lemma, we will make the following assumption throughout this paper.
\begin{assumption}
\label{assumption:domain}
Let $\cD\subset \bR^d$ be a closed, convex set with non-zero Lebesgue measure interior.
\end{assumption}
For example, if $d=2$ a possible choice is $\cD=[0,\infty)^2$ or $\cD=[0,a]\times (-\infty,\infty)$ for some $a>0$, stressing the fact that we allow for unbounded domains with non-smooth boundaries. 

At this point it is worth mentioning that if the domain is non-convex, it may not satisfy such helpful conditions. For example both  \cite{saisho1987stochastic} and  \cite{lions1984stochastic} assume the uniform exterior sphere condition and cannot access Lemma \ref{lem:NormalToDomain}, whereas \cite{tanaka2002stochastic} relies on Lemma \ref{lem:NormalToDomain}. 

\subsubsection*{Reflective boundaries and the Skorokhod problem}
We are now in the position to formulate the Skorokhod problem which was first stated and studied in \cites{Skorokhod1961stochastic, Skorokhod1962stochastic}. 

A path $\gamma:[0,T] \to \bR^d$ is said to be c\`adl\`ag if it is right continuous and has left limits. 

\begin{defn}
\label{dfn:Skorokhodproblem}
Let $\gamma:[0,T] \to \bR^d$ be a c\`adl\`ag path and let $\cD$ be a subset of $\bR^d$. Suppose additionally that $\gamma_0\in \cD$. For each $x\in \partial \cD$, suppose that $\cN_x\neq \emptyset$. Let $\n:\partial \cD \to \bR^d$ such that $\n(x)\in \cN_{x}$. The triple $(\gamma, \cD, \n)$ denotes the \emph{Skorokhod problem}.  

We say that the pair $(\eta, k)$ is a solution to the Skorokhod problem $(\gamma, \cD, \n)$ if $\eta:[0,T] \to \overline{\cD}$ is a c\`adl\`ag path, $k:[0,T] \to \bR^d$ is a bounded variation path and
\begin{equation}
\label{eq:Skorokhodproblem}
 \eta_t=\gamma_t - k_t , \quad k_t = \int_0^t \n(\eta_s)\1_{\partial \cD}(\eta_s) d|k|_s, \quad |k|_t = \int_0^t \1_{\partial \cD} (\eta_s) d|k|_s,
\end{equation}
where $\n(x)\in \cN_{x}$ when $x \in \partial \cD$ and $\n(x)=0$ otherwise. 
\end{defn}

This problem was first studied in the deterministic setting in \cite{Chaleyat1980Reflexion} and in the stochastic setting in \cite{tanaka2002stochastic}. For general domains, one may be unable to show uniqueness, or even existence of a solution to the Skorokhod problem. We emphasise that this will not be an issue that we explore in this paper.
\begin{theorem}[\cite{tanaka2002stochastic}*{Theorem 3.1}]
\label{thm:SkorokhodProblem}
Let $\cD$ satisfy Assumption \ref{assumption:domain}. Let $(\Omega, \cF, (\cF_t)_{t\in [0,T]}, \bP)$ be a filtered probability space. Let $\gamma=(\gamma_t)_{t\in[0,T]}$ be an $\cF_t$-adapted $\bR^d$-valued semimartingale with $\gamma_0\in \cD$. 

Then there exists a unique solution to the Skorokhod problem $(\gamma, \cD, \n)$ $\bP$-a.s. 
\end{theorem}

\section{Existence, uniqueness and propagation of chaos}
\label{section exist.unique}

In this section, we prove that under appropriate assumptions there exists a unique solution to the Stochastic Differential Equations \eqref{eq:MVE}. In the subsequent step, we address the \textit{Propagation of Chaos} result regarding convergence of the solution of the particle system \eqref{eq:ParticleSystem} to the solution of the McKean-Vlasov \eqref{eq:MVE}.  

In Section \ref{subsec:ExistUniq-RSDEs} we prove \textit{existence and uniqueness for a broad class of classical reflected SDEs} where the coefficients are assumed random, time-dependent and satisfying a superlinear growth condition. Crucially, we do not restrict ourselves to a bounded domain. In Section \ref{subsec:ExistUniq-RMVEs}  we prove \textit{existence and uniqueness for reflected McKean-Vlasov SDEs} satisfying a $\bW^{(2)}$-Lipschitz condition in the measure component. This is \textit{generalised in Theorem \ref{thm:ExistUnique-LocLip-SSMVE} to coefficients that are locally Lipschitz in measure}, although in this final step we necessarily restrict to deterministic coefficients; the proof of the result is provided in Section \ref{subsec:ExistUniq-SSMVEs}. 

Lastly, in Section \ref{sec:PoC}, we prove that the limit of a single equation within the system of interacting equations \eqref{eq:ParticleSystem} converges to the dynamics of Equation \eqref{eq:MVE}, i.e.~\textit{Propagation of Chaos (PoC)}. 

\subsection{Existence and uniqueness for reflected SDEs}
\label{subsec:ExistUniq-RSDEs}

Let $t\geq 0$. We commence by studying classical reflected SDEs of the form 
\begin{equation}
\label{eq:reflectedSDE}
\begin{split}
X_t =& \theta + \int_0^t b(s, X_s) ds + \int_0^t \sigma(s, X_s) dW_s - k_t,
\\
|k|_t=&\int_0^t \1_{\partial \cD}(X_s) d|k|_s, \qquad k_t = \int_0^t \1_{\partial \cD}(X_s) \n(X_s) d|k|_s.
\end{split}
\end{equation}
This first result is a generalisation of Tanaka's classical results in \cite{tanaka2002stochastic} extended to the case where the drift and diffusion terms are random and time dependent, and the drift term satisfies a one-sided Lipschitz condition. 

\begin{theorem}
\label{thm:ExistUnique-LocLip-Ref}
Let $\cD$ satisfy Assumption \ref{assumption:domain}. Let $p\geq 2$. Let $W$ be a $d'$ dimensional Brownian motion. Let $\theta:\Omega \to \cD$, $b:[0,T] \times \Omega \times \cD \to \bR^d$ and $\sigma:[0,T] \times \Omega \times \cD \to \bR^{d\times d'}$ be progressively measurable maps. Suppose that 
\begin{itemize}
\item $\theta \in L^p( \cF_0, \bP; \cD)$.
\item $\exists x_0\in \cD$ such that $b$ and $\sigma$ satisfy the integrability conditions
$$
\bE\Big[ \Big( \int_0^T \| b(s, x_0) \| ds \Big)^p \Big]
\vee
\bE\Big[ \Big( \int_0^T \| \sigma(s, x_0) \|^2 ds \Big)^{p/2} \Big] < \infty. 
$$
\item $\exists L>0$ such that for almost all $(s, \omega)\in [0,T] \times \Omega$ and $\forall x, y\in \cD$, 
$$
\big\langle b(s, x) - b(s, y), x-y \big\rangle
\leq  L \| x-y \|^2
\quad\textrm{and}\quad  
\| \sigma(s, x) - \sigma(s, y) \| \leq L \| x-y \|, 
$$
\item $\forall n\in \bN$, $\exists L_n>$ such that $\forall x,y\in \cD_n = \cD \cap \overline{B_n(x_0)}$, 
$$
\| b(s, x) - b(s, y) \| \leq L_n \| x-y \| 
\quad \textrm{for almost all $(s, \omega) \in [0,T] \times \Omega$. }
$$
\end{itemize}
Then there exists a unique solution to the reflected Stochastic Differential Equation \eqref{eq:reflectedSDE} in $\cS^p([0,T])$ and
$$
\bE\Big[ \| X - x_0\|_{\infty, [0,T]}^p \Big] \lesssim \bE\Big[ \| \theta - x_0 \|^p \Big] + \bE\Big[ \Big( \int_0^T \| b(s, x_0) \| ds \Big)^p \Big] + \bE\Big[ \Big( \int_0^T \|  \sigma(s, x_0) \|^2 ds \Big)^{p/2} \Big]. 
$$
\end{theorem}

The proof is given in Appendix \ref{appendixB}.

\subsection{Existence and uniqueness for McKean-Vlasov equations}
\label{subsec:ExistUniq-RMVEs}

Next, for $t \geq 0$, we study reflected McKean-Vlasov equations, i.e.~stochastic processes of the form
\begin{equation}
\label{eq:reflectedMVE}
\begin{split}
X_t =& \theta + \int_0^t b(s, X_s, \mu_s) ds + \int_0^t \sigma(s, X_s, \mu_s) dW_s - k_t, \quad \bP\big[ X_t \in dx\big] = \mu_t(dx), 
\\
|k|_t=& \int_0^t \1_{\partial D}(X_s) d|k|_s, \qquad k_t = \int_0^t \1_{\partial \cD}(X_s) \n(X_s) d|k|_s. 
\end{split}
\end{equation}

\begin{theorem}
\label{thm:ExistUnique-LocLip-MVE}
Let $\cD$ satisfy Assumption \ref{assumption:domain}. Let $p\geq 2$. Let $W$ be a $d'$ dimensional Brownian motion. Let $\theta:\Omega \to \cD$, $b:[0,T] \times \Omega \times \cD \times \cP_2(\cD) \to \bR^d$ and $\sigma:[0,T] \times \Omega \times \cD \times \cP_2(\cD) \to \bR^{d\times d'}$ be progressively measurable maps. Assume that 
\begin{itemize}
\item $\theta \in L^p( \cF_0, \bP; \cD)$ and $\theta \sim \mu_\theta$. 
\item $\exists x_0\in \cD$ such that $b$ and $\sigma$ satisfy the integrability conditions
$$
\bE\Big[ \Big( \int_0^T \| b(s, x_0, \delta_{x_0}) \| ds \Big)^p \Big]
\vee
\bE\Big[ \Big( \int_0^T \| \sigma(s, x_0, \delta_{x_0}) \|^2 ds \Big)^{p/2} \Big] < \infty. 
$$
\item $\exists L>0$ such that for almost all $(s, \omega)\in [0,T] \times \Omega$, $\forall \mu, \nu \in \cP_2(\cD)$ and $\forall x, y\in \bR^d$, 
\begin{align*}
\Big\langle b(s, x, \mu) - b(s, y, \mu), x-y \Big\rangle \leq L\| x-y \|^2, 
\quad
\| \sigma(s, x, \mu) - \sigma(s, y, \mu) \| \leq L \| x-y \|, 
\\
\| b(s, x, \mu) - b(s, x, \nu) \| \leq L \bW^{(2)}_\cD (\mu, \nu), 
\quad
\| \sigma(s, x, \mu) - \sigma(s, x, \nu) \| \leq L \bW^{(2)}_\cD (\mu, \nu). 
\end{align*}
\item $\forall n\in \bN$, $\exists L_n>$ such that $\forall x,y\in \cD \cap \overline{B_n(x_0)}$, 
$$
\| b(s, x, \mu) - b(s, y, \mu) \| \leq L_n \| x-y \| 
\quad \textrm{for almost all $(s, \omega) \in [0,T] \times \Omega$. }
$$

\end{itemize}
Then there exists a unique solution to the reflected McKean-Vlasov equation \eqref{eq:reflectedMVE} in $\cS^p([0,T])$ and
$$
\bE\Big[ \| X - x_0\|_{\infty, [0,T]}^p \Big] \lesssim \bE\Big[ \| \theta - x_0 \|^p\Big] + \bE\Big[ \Big( \int_0^T \| b(s, x_0, \delta_{x_0}) \| ds \Big)^p \Big] + \bE\Big[ \Big( \int_0^T \| \sigma(s, x_0, \delta_{x_0}) \|^2 ds \Big)^{p/2} \Big]. 
$$
\end{theorem}

\begin{proof}
Throughout this proof, we distinguish between measures $\nu \in \cP_2 \big( C([0,T]; \cD) \big)$ and their pushforward measure with respect to path evaluation $\nu_t \in \cP_2(\cD)$. 

Then for $\nu^1, \nu^2 \in \cP_2\big( C([0,T]; \cD)\big)$, we have
\begin{align}
\label{eq:thm:ExistUnique-LocLip-MVE2.1}
\sup_{t\in [0,T]} \bW_{\cD}^{(2)} \Big(\nu_t^1, \nu_t^2 \Big) \leq \bW_{C([0,T]; \cD)}^{(2)} \Big( \nu^1, \nu^2 \Big). 
\end{align}
For $\nu \in \cP_2( C([0,T]; \cD))$, we define the reflected Stochastic Differential Equation
\begin{equation}
\label{eq:thm:ExistUnique-LocLip-MVE1.1}
\begin{split}
X_t^{(\nu)} =&\theta + \int_0^t b(s, X_s^{(\nu)}, \nu_s) ds + \int_0^t \sigma(s, X_s^{(\nu)}, \nu_s) dW_s - k_t^{(\nu)}, 
\\
|k^{(\nu)}|_t=& \int_0^t \1_{\partial D}(X_s^{(\nu)}) d|k^{(\nu)}|_s, 
\quad
k^{(\nu)}_t = \int_0^t \1_{\partial \cD}(X_s^{(\nu)}) \n(X_s^{(\nu)}) d|k^{(\nu)}|_s. 
\end{split}
\end{equation}

Let $x_0\in \cD$. For $\mu_0 \in \cP_2(\cD)$, let $\mu_0' \in \cP_2\big( C([0,T]; \cD)\big)$ be the law of the constant path with initial distribution $\mu_0$. Using the Lipschitz condition for the measure dependency of $b$ and $\sigma$, we have
\begin{align*}
\bE\Big[ \Big( \int_0^T \| b(s, x_0, \nu_s) \| ds \Big)^p \Big] 
\leq& \bE\Big[ \Big( \int_0^T \| b(s, x_0, \mu_0) \| ds + L\int_0^T \bW_\cD^{(2)} (\nu_s, \mu_0) ds \Big)^p \Big]
\\
\leq& 2^{p-1} \bE\Big[ \Big( \int_0^T \| b(s, x_0, \mu_0) \| ds \Big)^p \Big] + 2^{p-1} L^p T^p \bW_{C([0,T]; \cD)}^{(2)} (\nu, \mu_0')^p, 
\\
\bE\Big[ \Big( \int_0^T \| \sigma(s, x_0, \nu_s) \|^2 ds \Big)^{p/2} \Big] 
\leq& \bE\Big[ \Big( 2\int_0^T \| \sigma(s, x_0, \mu_0) \|^2 ds + 2L^2 \int_0^T \bW_\cD^{(2)} (\nu_s, \mu_0) ds \Big)^{p/2} \Big]
\\
\leq& 2^{p-1} \bE\Big[ \Big( \int_0^T \| \sigma(s, x_0, \mu_0) \|^2 ds \Big)^{p/2} \Big] + 2^{p-1} L^{p} T^{p/2}\bW_{C([0,T]; \cD)}^{(2)} (\nu, \mu_0')^p.
\end{align*}
Therefore, by Theorem \ref{thm:ExistUnique-LocLip-Ref}, we have existence and uniqueness of Equation \eqref{eq:thm:ExistUnique-LocLip-MVE1.1}. Consider the operator $\Xi: \cP_2\big( C([0,T]; \bR^d)\big) \to \cP_2\big( C([0,T]; \bR^d)\big)$ defined by
$$
\Xi[\nu] := \mu^{(\nu)},
$$
where $\mu^{(\nu)}$ is the law of the solution to Equation \eqref{eq:thm:ExistUnique-LocLip-MVE1.1}. Now, for any two measures $\nu^1, \nu^2 \in \cP_2\big( C([0,T]; \cD)\big)$, 
\begin{align*}
\Big\| X^{(\nu^1)}_t - X^{(\nu^2)}_t \Big\|^2 \leq
& 2\int_0^t \Big\langle X_s^{(\nu^1)} - X_s^{(\nu^2)}, b(s, X_s^{(\nu^1)}, \nu_s^1) - b(s, X_s^{(\nu^2)}, \nu_s^2) \Big\rangle ds
\\
& + 2\int_0^t \Big\langle X_s^{(\nu^1)} - X_s^{(\nu^1)}, \Big( \sigma(s, X_s^{(\nu^1)}, \nu_s^1) - \sigma(s, X_s^{(\nu^2)}, \nu_s^2) \Big) dW_s \Big\rangle
\\
& + \int_0^t \Big\| \sigma(s, X_s^{(\nu^1)}, \nu_s^1) - \sigma(s, X_s^{(\nu^2)}, \nu_s^2) \Big\|^2 ds
-2\int_0^t \Big\langle X_s^{(\nu^1)} - X_s^{(\nu^2)}, dk_s^{(\nu^1)} - dk_s^{(\nu^2)} \Big\rangle.
\end{align*}

The reflective term in the above expression is negative due to the convexity of the domain and Lemma \ref{lem:NormalToDomain}. Therefore, taking a supremum over time, expectations, and using Burkholder-Davis-Gundy inequality, we get
\begin{align*}
\bE\Big[& \| X^{(\nu^1)} - X^{(\nu^2)}\|_{\infty, [0,T]}^2 \Big] 
\\
\leq& 2L\int_0^T \bE\Big[ \|X^{(\nu^1)} - X^{(\nu^2)}\|_{\infty, [0,t]}^2 \Big] dt
+ 2L \bE\Big[ \|X^{(\nu^1)} - X^{(\nu^2)}\|_{\infty, [0,T]} \cdot \int_0^T \sup_{s\in[0,t]} \bW^{(2)}_{\cD}(\nu_s^1, \nu_s^2) dt \Big]
\\
&+ 4C_1L \bE\Big[ \| X^{(\nu^1)} - X^{(\nu^2)}\|_{\infty, [0,T]} \Big( \int_0^T  \sup_{s\in[0,t]} \bW_{\cD}^{(2)} (\nu_s^1, \nu_s^2)^2 dt \Big)^{1/2} \Big]
\\
&+ 4C_1L \bE\Big[ \| X^{(\nu^1)} - X^{(\nu^2)}\|_{\infty, [0,T]} \Big( \int_0^T \| X^{(\nu^1)} - X^{(\nu^2)} \|_{\infty, [0,t]}^2 dt\Big)^{1/2} \Big]
\\
&+ 2L^2\int_0^T \bE\Big[ \| X^{(\nu^1)} - X^{(\nu^2)}\|_{\infty, [0, t]}^2 dt
+ 2L^2 \int_0^T \sup_{s\in [0,t]} \bW_{\cD}^{(2)} (\nu_s^1, \nu_s^2)^2 dt.
\end{align*}

Careful application of Young's Inequality, Gr\"onwall's inequality and Equation \eqref{eq:thm:ExistUnique-LocLip-MVE2.1} yields that there exists a constant $K>0$ such that
$$
\bW_{C([0,T]; \cD)}^{(2)} \Big( \Xi[ \nu^1], \Xi[\nu^2] \Big)^2 \leq \bE\Big[ \| X^{(\nu^1)} - X^{(\nu^1)}\|_{\infty, [0,T]}^2 \Big] 
\leq
K \int_0^T \bW_{C([0,t]; \cD)}^{(2)} \Big(\nu^1, \nu^2\Big)^2 dt.
$$
Iteratively applying the operator $\Xi$ $n$ times gives
\begin{align*}
\bW_{C([0,T]; \cD)}^{(2)} \Big( \Xi^n[ \nu^1], \Xi^n[\nu^2] \Big)^2 
\leq& K^n \int_0^T \int_0^{t_1} ... \int_0^{t_{n-1}} \bW_{C([0,t_n]; \cD)}^{(2)} \Big(\nu^1, \nu^2\Big)^2 dt_n ... dt_2 dt_1
\\
\leq& \frac{K^n}{n!} \bW_{C([0,T]; \cD)}^{(2)}\Big( \nu^1, \nu^2 \Big)^2. \end{align*}
Choosing $n\in \bN$ such that $\tfrac{K^n}{n!}<1$, we obtain that the operator $\Xi^n$ is a contraction operator, so a unique fixed point on the metric space $\cP_2\big( C([0,T]; \cD)\big)$ paired with the Wasserstein metric must exist. 

This unique fixed point is the law of the McKean-Vlasov equation \eqref{eq:reflectedMVE}. 
\end{proof}
\begin{rem}
It is worth remarking that the framework of coefficients that satisfy a Lipschitz condition in their measure dependency with respect to the Wasserstein distance is broad, but in this manuscript we are predominantly interested in coefficients where the measure dependency is not Lipschitz.
\end{rem}

\subsubsection*{Main result: existence and uniqueness for McKean-Vlasov equations under reflection}

We next study McKean-Vlasov equations with the addition of a self-stabilizing drift term that does not satisfy a Lipschitz condition with respect to the Wasserstein distance. For example, in Equation \eqref{eq:MVE}, we have $f\ast \mu_t(x): = \int_{\cD} f(x - y) \mu_t(dy)$, the convolution of the vector field $f$ with the measure $\mu_t$. Consider
\begin{equation}
\label{eq:reflectedSSMVE}
\begin{split}
X_t =& \theta + \int_0^t b(s, X_s, \mu_s) ds + \int_0^t \sigma(s, X_s, \mu_s) dW_s + \int_0^t f \ast \mu_s(X_s) ds - k_t, 
\\
|k|_t=& \int_0^t \1_{\partial D}(X_s) d|k|_s, \qquad k_t = \int_0^t \1_{\partial \cD}(X_s) \n(X_s) d|k|_s, \qquad \bP\Big[ X_t \in dx\Big] = \mu_t(dx). 
\end{split}
\end{equation}
We show existence of a solution to the above reflected McKean-Vlasov equation under the following assumption. 
\begin{assumption}
\label{ass:ExistUnique-LocLip-SSMVE}
Let $r>1$ and $p>2r$. Let $\theta:\Omega \to \cD$, $b:[0,T] \times \cD \times \cP_2(\cD) \to \bR^d$, $f:\bR^d \to \bR^d$ and $\sigma:[0,T] \times \cD \times \cP_2(\cD) \to \bR^{d \times d'}$. Assume that
\begin{itemize}
\item $\theta \in L^p( \cF_0, \bP; \cD)$ and $\theta \sim \mu_\theta$, 
\item $\exists x_0\in \cD$ such that $b$ and $\sigma$ satisfy the integrability conditions
$$
\int_0^T \| b(s, x_0, \delta_{x_0}) \| ds 
\vee
\int_0^T \| \sigma(s, x_0, \delta_{x_0}) \|^2 ds < \infty. 
$$
\item $\exists L>0$ such that for almost all $s\in [0,T]$, $\forall \mu, \nu \in \cP_2(\cD)$ and $\forall x, y\in \cD$, 
\begin{align*}
\Big\langle b(s, x, \mu) - b(s, y, \mu), x-y \Big\rangle \leq L\| x-y \|^2, 
\quad
\| \sigma(s, x, \mu) - \sigma(s, y, \mu) \| \leq L \| x-y \|, 
\\
\| b(s, x, \mu) - b(s, x, \nu) \| \leq L \bW^{(2)}_\cD (\mu, \nu), 
\quad
\| \sigma(s, x, \mu) - \sigma(s, x, \nu) \| \leq L \bW^{(2)}_\cD (\mu, \nu),  
\end{align*}
\item $f(0)=0$, $f(x) = -f(-x)$ and $\exists L>0$ such that $\forall x, y \in \bR^d$, 
$
\big\langle f(x) - f(y), x-y\big\rangle \leq L \| x-y \|^2$ , 
\item $\forall n\in \bN$, $\exists L_n>$ such that $\forall x,y\in \cD \cap \overline{B_n(x_0)}$, 
$$
\| b(s, x, \mu) - b(s, y, \mu) \| \leq L_n \| x-y \| 
\quad \textrm{for almost all $(s, \omega) \in [0,T] \times \Omega$,  }
$$
\item $\exists L>0$ such that $\forall x,y\in \bR^d$ 
$$
\| f(x) - f(y) \| \leq C \| x-y \| \big( 1 + \| x \|^{r-1} + \| y \|^{r-1} \big), 
\quad 
\| f(x) \| \leq C \big( 1 + \| x \|^r\big).
$$
\end{itemize}
\end{assumption}

\begin{theorem}
\label{thm:ExistUnique-LocLip-SSMVE}
Let $\cD\subseteq \bR^d$ (not necessarily bounded) satisfy Assumption \ref{assumption:domain}. Let $r>1$ and $p>2r$. Let $W$ be a $d'$ dimensional Brownian motion. Let $\theta$, $b$, $\sigma$ and $f$ satisfy Assumption \ref{ass:ExistUnique-LocLip-SSMVE}. 

Then there exists a unique solution to the reflected McKean-Vlasov equation \eqref{eq:reflectedSSMVE} in $\cS^p([0,T])$ (explicit $\cS^p$-norm bounds are given below in \eqref{eq momoment bounds for existence uniqueness}).
\end{theorem}
The proof of this theorem is the content of the next section.

\begin{rem}
A nuanced detail of the following proof is the calculation of moments and potentially singular and non-integrable drifts. In  \cite{imkeller2019Differentiability}, the authors studied processes where the drift term could have polynomial growth that was greater than the  moments of the final solution. The conclusion was that time integrals of these drift terms ``smooth out'' the non-integrability. 

In this paper, we only require a one-sided Lipschitz condition for the spatial variable. However, we were unable to remove the polynomial growth condition for the self-stabilizing term $f$. This is because one needs integrability of the convolution of the law of the solution with the vector field $f$ before the self-stabilisation acts to push deviating paths back towards the mean of the distribution. 
\end{rem}

\subsection{Proof of Theorem \ref{thm:ExistUnique-LocLip-SSMVE}}
\label{subsec:ExistUniq-SSMVEs}

This proof is inspired by \cite{BRTV}. Unlike the proof of Theorem \ref{thm:ExistUnique-LocLip-MVE} which constructs a contraction operator on the space of measures, we construct a fixed point on a space of functions. Each function will give rise to a McKean-Vlasov process by substituting it into the equation as a drift term. Then, the law of this McKean-Vlasov equation is convolved with the vector field $f$ to obtain a new function. This trick allows us to bypass the non-Lipschitz property of the functional $g(x, \mu):= f \ast \mu(x)$  while still exploiting the one-sided Lipschitz condition in the spatial variable.

Our contributions in this section include developing this method to allow for diffusion terms that are not constant. This is novel, even before the addition of a domain of constraint. The non-constant diffusion complicates the computation of moment estimates which are key to this method. Of particular interest is Proposition \ref{prop:SSMVE-Moments}, which diverges from previous literature.  

\begin{defn}
\label{defn:WeirdNorm-Space}
Let $r>1$. Let $x_0 \in \cD$ and $L>0$ be as in Assumption \ref{ass:ExistUnique-LocLip-SSMVE}. For $g:[0,T]\times \cD \to \bR^d$, let
$$
\| g \|_{[0,T],r}:= \sup_{t \in [0,T]} \left( \sup_{x\in \cD}  \frac{ \|g(t,x)\|}{1 + \|x-x_0\|^{r}}  \right).
$$
Let $\Lambda_{[0,T], r}$ be the space of all functions $g:[0,T] \times \cD \to \bR^d$ such that $\| g\|_{[0,T], r}<\infty$ and
$$
\langle g(t,x) - g(t,y), x-y\rangle \leq L \|x-y\|^2\qquad \forall xy,\in \cD,~t\in[0,T]. 
$$
\end{defn}

The space $\Lambda_{[0,T], r}$ is a Banach space. For $g\in \Lambda_{[0,T], r}$, consider the reflected McKean-Vlasov equation
\begin{equation}
\label{eq:reflected-SSMVE-b}
\begin{split}
X_t^{(g)} =& \theta + \int_0^t b(s, X_s^{(g)}, \mu_s^{(g)}) ds + \int_0^t \sigma(s, X_s^{(g)}, \mu_s^{(g)}) dW_s  + \int_0^t g(s, X_s^{(g)}) ds - k_t^{(g)},
\\
|k^{(g)}|_t=& \int_0^t \1_{\partial D}(X_s^{(g)}) d|k^{(g)}|_s, \quad 
k_t^{(g)} = \int_0^t \1_{\partial \cD}(X_s^{(g)}) \n(X_s^{(g)}) d|k^{(g)}|_s, \quad 
\bP\Big[ X_t^{(g)} \in dx\Big] = \mu_t^{(g)}(dx). 
\end{split}
\end{equation}

By Theorem \ref{thm:ExistUnique-LocLip-MVE}, we know that there exists a unique solution to this McKean-Vlasov equation for every choice of $g\in \Lambda_{[0,T], r}$ and every $r\geq 1$. Further, we have the moment estimate that for $\varepsilon>0$ and $T_0\in [0,T-\varepsilon]$, 

\begin{align}
\nonumber
&\sup_{t\in [T_0, T_0+\varepsilon]} \bE\Big[ \| X^{(g)}_t - x_0 \|^p \Big] 
\\
\nonumber
&\leq
\Bigg( 4\bE\Big[ \| X_{T_0}^{(g)} - x_0 \|^p \Big]
+ \big( 4 (p-1)\big)^{p-1} \bigg( \Big( \int_{T_0}^{T_0+\varepsilon} \| b(r, x_0, \delta_{x_0} ) \| dr\Big)^p + \Big( \int_{T_0}^{T_0+\varepsilon} \| g(r, x_0) \| dr \Big)^p \bigg)
\\
\label{eq:reflected-SSMVE-b:moment}
&
\quad 
+ 2(p-1)^{p/2} \cdot (p-2)^{(p-2)/2} \cdot 4^{p/2} \Big( \int_{T_0}^{T_0+\varepsilon}  \| \sigma(r, x_0, \delta_{x_0} ) \|^2 dr \Big)^{\tfrac{p}{2}} \Bigg) 
\cdot \exp\Big( \big( 4pL + 2p(p-1)L^2  \big) \varepsilon \Big). 
\end{align}
Our challenge will be to find a $g$ such that $g(t, x) = f \ast \mu_t^{(g)}(x)$. 
\begin{defn}
\label{dfn:Gamma-ContractionOp}
Let $b$, $\sigma$ and $f$ satisfy Assumption \ref{ass:ExistUnique-LocLip-SSMVE}. Let $g\in \Lambda_{[0,T], r}$. Let $X^{(g)}$ be the unique solution to the McKean-Vlasov equation \eqref{eq:reflected-SSMVE-b} with law $\mu^{(g)}$.  Let $\Gamma: \Lambda_{[0,T], r} \to C([0,T] \times \cD; \bR^d)$ be defined by
$$
\Gamma[g](t, x):= f\ast \mu_t^{(g)}(x) = \bE\big[ f(x - X_t^{(g)}) \big]. 
$$
\end{defn}

Our goal is to demonstrate that the operator $\Gamma$ has a fixed point $g'$. Then the McKean-Vlasov equation $X^{(g')}$ that solves \eqref{eq:reflected-SSMVE-b} will be the solution to the McKean-Vlasov equation \eqref{eq:reflectedSSMVE}.
\begin{lemma}
\label{lem:Gamma-WellDefined}
Let $\Gamma$ be the operator defined in Definition \ref{dfn:Gamma-ContractionOp}. Then $\forall T_0\in[0,T]$ and $\forall \varepsilon>0$ such that $T_0+\varepsilon< T$, $\Gamma $ maps $\Lambda_{[T_0,T_0+\varepsilon], r}$ to $\Lambda_{[T_0,T_0+\varepsilon], r}$.
\end{lemma}
\begin{proof}
Fix $T_0\in[0,T]$ and $\varepsilon>0$ appropriately. Let $g\in \Lambda_{[T_0,T_0 +\varepsilon], r}$. Then $\forall x, y\in \bR^d$ and $\forall t\in[T_0,T_0 + \varepsilon]$, 
\begin{align*}
\Big\langle x-y, \Gamma[g](t, x) - \Gamma[g](t, y)\Big \rangle 
=
\int_{\cD} \Big\langle x-y, f(x-u) - f(y-u) \Big\rangle d\mu_t^{(g)}(u)
\leq
L \| x-y \|^2.  
\end{align*}
Secondly, 
\begin{align*}
\bE\Big[ f(X_t^{(g)} - x) \Big] \leq
& 2C + \big(C+2^r\big) \Big( \| x - x_0 \|^r + \bE\Big[ \| X_t^{(g)} \|^r \Big] \Big)
\\
\leq& \Big(2C + 2^{r+1}\Big) \Big( 1 + \| x-x_0 \|^r\Big) \Big( 1 + \bE\Big[ \| X_t^{(g)} - x_0 \|^r\Big] \Big). 
\end{align*}
By Assumption \ref{ass:ExistUnique-LocLip-SSMVE}, we know the process $X^{(g)}$ has finite moments of order $p>2r$. Thus 
\begin{equation}
\label{eq:lem:Gamma-WellDefined-1}
\Big\| \Gamma[g] \Big\|_{[T_0,T_0+\varepsilon], r} \leq \Big( 2C + 2^{r+1}\Big)\cdot \Big( 1 + \sup_{t\in[T_0,T_0+\varepsilon]} \bE\Big[ \| X^{(g)}_t -x_0 \|^r \Big] \Big). 
\end{equation}
Combining these with Equation \eqref{eq:reflected-SSMVE-b:moment} and using that
\begin{align*}
\Big( \int_{T_0}^{T_0+\varepsilon} \| g(s, x_0) \| ds \Big)^p \leq \varepsilon^p \| g\|_{[T_0, T_0+\varepsilon], r}^p, 
\end{align*}
we obtain that 
\begin{align}
\nonumber
\Big\| \Gamma[g] \Big\|_{[T_0, T_0+\varepsilon], r}
\leq& \Big(2C+2^{r+1}\Big)\Big( 1+ \sup_{t\in[0,T_0]} \bE\Big[ \| X_t^{(g)} - x_0 \|^r \Big] \Big)
\\
\nonumber
&+ \Bigg( \big( 4 (p-1)\big)^{p-1} \bigg( \Big( \int_{T_0}^{T_0+\varepsilon} \| b(s, x_0, \delta_{x_0} ) \| ds\Big)^p + \Big( \int_{T_0}^{T_0+\varepsilon} \| g(s, x_0) \| ds \Big)^p \bigg)
\\
\nonumber
&\quad + 2(p-1)^{p/2} \cdot (p-2)^{(p-2)/2} \cdot 4^{p/2} \Big( \int_{T_0}^{T_0+\varepsilon}  \| \sigma(s, x_0, \delta_{x_0} ) \|^2 ds \Big)^{\tfrac{p}{2}} \Bigg) 
\\
\label{eq:lem:Gamma-WellDefined-2}
&\quad \cdot \exp\Big( \big( 4pL + 2p(p-1)L^2  \big) \varepsilon \Big). 
\end{align}

Taking $T_0=0$ and $\varepsilon = T$, we get $\Big\| \Gamma[g]\Big\|_{[0,T], r}<\infty$ for any $g\in \Lambda_{[0,T], r}$. 

\end{proof}

\begin{lemma}
\label{lemma:Gamma-FirstContraction}
Let $T_0\in [0,T]$ and let $\varepsilon>0$ such that $T_0+\varepsilon < T$. Let $\Gamma$ be the operator given in Definition \ref{dfn:Gamma-ContractionOp}. Then there exists a constant $K$ such that $\forall g_1, g_2 \in \Lambda_{[T_0,T_0+\varepsilon], r}$ with  $g_1(t) = g_2(t)$ $\forall t\in[0,T_0]$ we have 
$$
\Big\| \Gamma[g_1] - \Gamma[g_2] \Big\|_{[T_0,T_0+\varepsilon], r} \leq \| g_1 - g_2 \|_{[T_0, T_0 + \varepsilon], r} K \sqrt{\varepsilon} e^{K \varepsilon}. 
$$
\end{lemma}

\begin{proof}
Let $g_1,g_2:[0,T] \times \cD \to \bR^d$ such that $g_1(t) = g_2(t)$ for $t\in[0,T_0]$. Let $X^{(g_1)}$ and $X^{(g_2)}$ be solutions to Equation \eqref{eq:reflected-SSMVE-b}. Firstly, for $t\in [T_0, T_0 + \varepsilon]$ we have, applying It\^o's formula, 
%
%
\begin{align*}
\| X_t^{(g_1)} &- X_t^{(g_2)} \|^2 
\\
=
&
2\int_{T_0}^t \Big\langle X_s^{(g_1)} - X_s^{(g_2)}, b(s, X_s^{(g_1)}, \mu_s^{(g_1)}) - b(s, X_s^{(g_2)}, \mu_s^{(g_2)}) \Big\rangle ds
\\
&+ 2\int_{T_0}^t \Big\langle X_s^{(g_1)} - X_s^{(g_2)}, g_1(X_s^{(g_1)}) - g_1(X_s^{(g_2)}) \Big\rangle ds
+ 2\int_{T_0}^t \Big\langle X_s^{(g_1)} - X_s^{(g_2)}, g_1(X_s^{(g_2)}) - g_2(X_s^{(g_2)}) \Big\rangle ds 
\\ 
&+2\int_{T_0}^t \Big\langle X_s^{(g_1)} - X_s^{(g_2)}, \Big(\sigma(s, X_s^{(g_1)}, \mu_s^{(g_1)}) - \sigma(s, X_s^{(g_2)}, \mu_s^{(g_2)}) \Big)dW_s \Big\rangle
\\
&+\int_{T_0}^t \Big\| \sigma(s, X_s^{(g_1)}, \mu_s^{(g_1)}) - \sigma(s, X_s^{(g_2)}, \mu_s^{(g_2)}) \Big\|^2 ds 
- 2\int_{T_0}^t \Big\langle X_s^{(g_1)} - X_s^{(g_2)}, dk_s^{(g_1)} - dk_s^{(g_2)}\Big\rangle.
\end{align*}
Taking expectations, a supremum over time and applying Lemma \ref{lem:NormalToDomain}, we get
\begin{align*}
\sup_{t\in[T_0,T_0+\varepsilon ]} \bE\Big[ \| X_t^{(g_1)} - X_t^{(g_2)} \|^2 \Big]\leq& (6L + 4L^2)\int_{T_0}^{T_0+\varepsilon} \sup_{s\in[T_0,T_0+t]} \bE\Big[ \| X_s^{(g_1)} - X_s^{(g_2)} \|^2 \Big] dt
\\
&+2\int_{T_0}^{T_0+\varepsilon} \bE\Big[ \| X_t^{(g_1)} - X_t^{(g_2)} \| \cdot \| g_1 - g_2\|_{[T_0,T_0+t], r} \Big( 1 + \| X_t^{(g_2)} - x_0 \|^{r} \Big) \Big] dt. 
\end{align*}

An application of Gr\"onwall's Inequality yields
\begin{align}
\nonumber
\sup_{t\in[T_0,T_0+\varepsilon]} &\bE\Big[ \| X_t^{(g_1)} - X_t^{(g_2)} \|^2 \Big]
\\
\label{eq:prop:Gamma-FirstContraction1.1}
&\leq 8\| g_1 - g_2\|_{[T_0,T_0+\varepsilon], r}^2 \cdot \varepsilon \cdot e^{(8L^2+12L)\varepsilon} \cdot \bigg( 1 + \sup_{t\in [T_0, T_0+\varepsilon]} \bE\Big[ \| X^{(g_2)}_t - x_0 \|^{2r} \Big] \bigg).
\end{align}
Let $x\in \cD$. Using the polynomial growth assumption of $f$, we have that
\begin{align}
\nonumber
\bE\Big[ &f(x - X_t^{(g_1)}) - f(x - X_t^{(g_2)}) \Big]
\\
\nonumber
\leq& (C+2^r) 
\bE\Big[ \| X_t^{(g_1)} - X_t^{(g_2)} \| 
\cdot 
\big( 1 + \| x - x_0 \|^r\big) 
\cdot 
\big( 1 + \| X_t^{(g_1)} - x_0 \|^{r} + \| X_t^{(g_2)} - x_0 \|^{r} \big) \Big]
\\
\label{eq:prop:Gamma-FirstContraction2.1}
\leq& (C+2^r) \cdot \Big( 1 + \|x - x_0 \|^r\Big) \bE\Big[ \| X_t^{(g_1)} - X_t^{(g_2)} \|^2\Big]^{\tfrac{1}{2}} \cdot \bE\Big[ \Big( 1 + \| X_t^{(g_1)} - x_0 \|^{r} + \| X_t^{(g_2)} - x_0 \|^{r} \Big)^2 \Big]^{\tfrac{1}{2}}.
\end{align}
By Assumption \ref{ass:ExistUnique-LocLip-SSMVE} and \eqref{eq:reflected-SSMVE-b:moment} we have that
$$
\sup_{t\in [0,T]} \bE\Big[ \| X_t^{(g_1)} - x_0 \|^{2r} \Big], \quad
\sup_{t\in [0,T]} \bE\Big[ \| X_t^{(g_2)} - x_0 \|^{2r} \Big] <\infty. 
$$
Further, these bounds are uniform and depend only on $b$ and $\sigma$. 

Substituting Equation \eqref{eq:prop:Gamma-FirstContraction1.1} into Equation \eqref{eq:prop:Gamma-FirstContraction2.1}, we get 
\begin{align}
\nonumber
\Big\|& \Gamma[g_1] - \Gamma[g_2] \Big\|_{[T_0,T_0+\varepsilon], r} = \sup_{t\in [T_0,T_0+\varepsilon]} \sup_{x\in \cD} \frac{\bE\Big[ f(x - X_t^{(g_1)}) - f(x - X_t^{(g_2)}) \Big] }{1 + | x - x_0|^r}
\\
\label{eq:prop:Gamma-FirstContraction}
&\leq (C+2^r) 3\sqrt{8} \| g_1 - g_2\|_{[T_0,T_0+\varepsilon], r} \sqrt{\varepsilon} e^{(4L^2+6L)\varepsilon} \Bigg( 1+ \sup_{t\in [T_0, T_0+\varepsilon]} \bE\Big[ \| X_t^{(g_1)} \|^{2r} + \| X_t^{(g_2)} \|^{2r} \Big] \Bigg) . 
\end{align}

\end{proof}

Next, our goal is to establish a subset on which this operator is a contraction operator. 

\begin{defn}
Let $K>0$. For $T>0$ and $r>1$, we define
$$
\Lambda_{[0,T], r, K}:= \Big\{ g\in \Lambda_{[0,T], r}: \|g\|_{[0,T], r}\leq K\Big\}. 
$$
\end{defn}

Our goal is to choose $T$ and $K$ so that $\Gamma$ is a contraction operator when restricted to $\Lambda_{[0,T], r, K}$. 

\begin{prop}
\label{prop:SSMVE-ExistenceLocalSolution} 
Let $\Gamma:\Lambda_{[0,T], r} \to \Lambda_{[0,T], r}$ be as defined in Definition \ref{dfn:Gamma-ContractionOp}. Then $\exists K_1, \varepsilon>0$ such that, 
$$
\Gamma\Big[ \Lambda_{[0, \varepsilon], r, K_1} \Big] \subset \Lambda_{[0,\varepsilon], r, K_1},
\qquad\textrm{and}\qquad
\forall g_1, g_2 \in \Lambda_{[0,\varepsilon], r, K_1}\quad 
\Big\| \Gamma[g_1] - \Gamma[ g_2] \Big\|_{[0,\varepsilon], r} \leq \frac{1}{2} \Big\| g_1 - g_2 \Big\|_{[0, \varepsilon], r}. 
$$
As such, there exists a unique solution to Equation \eqref{eq:reflectedSSMVE} on the interval $[0,\varepsilon]$. 
\end{prop}

\begin{proof}

Let $\varepsilon>0$. Let $g\in \Lambda_{[0,\varepsilon], r, K_1}$. Taking Equation \eqref{eq:lem:Gamma-WellDefined-2} with $T_0=0$ provides

\begin{align*}
&\Big\| \Gamma[g] \Big\|_{[0, \varepsilon], r}
\\
&\leq \Big(2C+2^{r+1}\Big)\Big( 1+ \bE\Big[ |\theta - x_0|^r \Big] \Big)
+ \Bigg( \big( 4 (p-1)\big)^{p-1} \bigg( \Big( \int_{0}^{\varepsilon} | b(s, x_0, \delta_{x_0} )| ds\Big)^p + \Big( \varepsilon K_1 \Big)^p \bigg)
\\
&\quad + 2(p-1)^{p/2} \cdot (p-2)^{(p-2)/2} \cdot 4^{p/2} \Big( \int_{0}^{\varepsilon}  |\sigma(s, x_0, \delta_{x_0} ) |^2 ds \Big)^{\tfrac{p}{2}} \Bigg) 
\cdot \exp\Big( \big( 4pL + 2p(p-1)L^2  \big) \varepsilon \Big). 
\end{align*}
Choose $K_1= 2(2C + 2^{r+1}) \Big( 1+ \bE\Big[ \| \theta - x_0 \|^p \Big]\Big)$. We have the limit

$$
\lim_{\varepsilon \to 0} \Big( \int_{0}^{\varepsilon} \| b(s, x_0, \delta_{x_0} ) \| ds\Big)^p + \Big( \int_{0}^{\varepsilon} \| \sigma(s, x_0, \delta_{x_0} ) \|^2 ds \Big)^{\tfrac{p}{2}} = 0. 
$$
Then we can choose $\varepsilon'>0$ such that $
\big\| \Gamma[g] \big\|_{[0, \varepsilon'], r}< K_1. 
$

Secondly, using Equation \eqref{eq:prop:Gamma-FirstContraction} we choose $\varepsilon''>0$ such that
$$
\Big\| \Gamma[g_1] - \Gamma[g_2] \Big\|_{[0,\varepsilon''], r} < \frac{\| g_1 - g_2\|_{[0,\varepsilon''],r}}{2}. 
$$
We emphasise that the choice of $\varepsilon = \min\{ \varepsilon', \varepsilon''\}$ is dependent on the choice of $K_1$.

Define $d:\Lambda_{[0,\varepsilon], r} \times \Lambda_{[0,\varepsilon], r} \to \bR^+$ to be the metric $d(g_1, g_2) = \| g_1 - g_2\|_{[0,\varepsilon], r}$. The metric space $(\Lambda_{[0,\varepsilon], r, K_1}, d)$ is non-empty, complete and $\Gamma:\Lambda_{[0,\varepsilon], r, K_1} \to \Lambda_{[0,\varepsilon], r, K_1}$ is a contraction operator. Therefore, $\exists g'\in \Lambda_{[0,\varepsilon], r, K_1}$ such that $\Gamma[g'] = g'$. Thus $\forall t\in [0,\varepsilon]$, 
$$
g'\Big(t, X_t^{(g')} \Big) =  f\ast \mu_t^{(g')} (X_t^{(g')}). 
$$
Substituting this into \eqref{eq:reflected-SSMVE-b}, we obtain \eqref{eq:reflectedSSMVE}. Thus a solution to \eqref{eq:reflectedSSMVE} exists in $\cS^p([0,\varepsilon])$. 
\end{proof}
Our challenge now is to find a solution over the whole interval $[0,T]$. 

\begin{prop} 
\label{prop:SSMVE-Moments}
Let $\cD$ satisfy Assumption \ref{assumption:domain}. Let $r>1$ and $p>2r$. Let $W$ be a $d'$ dimensional Brownian motion. Let $b$, $\sigma$ and $f$ satisfy Assumption \ref{ass:ExistUnique-LocLip-SSMVE}. Suppose that a solution $X$ to the McKean-Vlasov equation \eqref{eq:reflectedSSMVE} exists in $\cS^p([0,T_0])$ for some $0<T_0<T$. Then there exists a constant $K_2=K_2(p, T)$ such that
$$
\bigg(\sup_{t\in[0,T_0]} \bE\Big[ \| X_t - x_0 \|^p \Big]\bigg) \vee  \bigg(\bE\Big[ \| X - x_0\|_{\infty, [0,T_0]}^p \Big]\bigg)< K_2. 
$$
\end{prop}
The challenge of this proof is that the symmetry trick for establishing second moments (see Equation \eqref{eq:prop:SSMVE-Moments1.0}) does not hold for higher moments. However, if we try to bypass this using the methods of \cite{HIP}, the non-constant diffusion terms yields integrals that blow up. Arguing by induction on $m$, we fix this by considering 
$$
\sup_{t\in[0,T]} \bE\Big[ \| X_t - x_0\|^{2m} \Big] + \bE\Big[ \| X_t - \tilde{X}_t\|^{2m} \Big], 
$$
and demonstrating via a Gr\"onwall argument that this is finite, even though a similar argument would not work for either of these terms on their own. 
\begin{proof}
Suppose that $t\in[0,T_0]$. Let $(X_t, k_t)$, $(\tilde{X_t}, \tilde{k_t})$ and $(\overline{X_t}, \overline{k_t})$ be independent, identically distributed solutions of Equation \eqref{eq:reflectedSSMVE}. 

Consider the two processes
\begin{align*}
\| X_t - x_0 \|^2 = \| \theta &- x_0 \|^2 
+ 2\int_0^t \Big\langle X_s - x_0, b(s, X_s, \mu_s) \Big\rangle ds
+ 2\int_0^t \Big\langle X_s - x_0, \sigma(s, X_s, \mu_s) dW_s \Big\rangle 
\\
&+ \int_0^t \Big\| \sigma(s, X_s, \mu_s) \Big\|^2 ds
+ 2\int_0^t \Big\langle X_s - x_0, \overline{\bE}\Big[ f(X_s - \overline{X_s}) \Big] \Big\rangle ds
- 2\int_0^t \Big\langle X_s - x_0, dk_s \Big\rangle, 
\\
\| X_t - \tilde{X_t} \|^2 
= \| \theta &- \tilde{\theta} \|^2 + 2\int_0^t \Big\langle X_s - \tilde{X_s}, b(s, X_s, \mu_s) - b(s, \tilde{X_s}, \mu_s) \Big\rangle ds
\\
&+ 2\int_0^t \Big\langle X_s - \tilde{X_s}, \sigma(s, X_s, \mu_s)dW_s - \sigma(s, \tilde{X_s}, \mu_s)d\tilde{W}_s \Big\rangle
\\
&+ \int_0^t \Big\| \sigma(s, X_s, \mu_s) \Big\|^2 + \Big\| \sigma(s, \tilde{X_s}, \mu_s) \Big\|^2 ds
\\
&+ 2\int_0^t \Big\langle X_s - \tilde{X_s}, \overline{\bE}\Big[ f(X_s - \overline{X_s}) - f(\tilde{X_s} - \overline{X_s}) \Big] \Big\rangle ds 
- 2\int_0^t \Big\langle X_s - \tilde{X_s}, dk_s - d\tilde{k_s} \Big\rangle. 
\end{align*}

We remark that since $f$ is symmetric we have the identity

\begin{equation}
\label{eq:prop:SSMVE-Moments1.0}
\bE\Big[ \Big\langle X_s - x_0, \overline{\bE}\Big[ f( X_s - \overline{X_s}) \Big] \Big\rangle \Big] 
\leq L\cdot \bE\Big[ \overline{\bE}\Big[ \| X_s - \overline{X_s}\|^2 \Big]\Big].
\end{equation}

Taking expectations of both processes (and no longer distinguishing between the integral operators $\bE$ and $\tilde{\bE}$) and adding them together, we get
\begin{align*}
\bE\Big[ \| X_t - x_0\|^2 + \| X_t - \tilde{X_t}\|^2 \Big] \leq& \bE\Big[ \|\theta - x_0\|^2\Big] + \bE\Big[ \|\theta - \tilde{\theta}\|^2 \Big]
\\
&+ (4L+12L^2) \int_0^t \bE\Big[ \|X_s - x_0\|^2 \Big] ds + 2\int_0^t \bE\Big[ \|X_s - x_0\|\Big] \cdot \|b(s, x_0, \delta_{x_0})\| ds
\\
&
+ 6\int_0^t \| \sigma(s, x_0, \delta_{x_0})\|^2 ds
+ 6L \int_0^t \bE\Big[ \|X_s - \tilde{X_s}\|^2 \Big] ds.
\end{align*}
Taking a supremum over $t\in[0,T_0]$, then applying Young's inequality followed by Gr\"onwall's inequality, we obtain
\begin{align*}
\sup_{t\in[0,T_0]} \bE\Big[ \| X_t - x_0 \|^2 + \| X_t - \tilde{X_t} \|^2 \Big] \leq&2\Bigg( \bE\Big[ \| \theta - x_0 \|^2 \Big] + \bE\Big[ \| \theta - \tilde{\theta} \|^2 \Big] 
\\
&+ \Big( \int_0^T \Big\| b(s, x_0, \delta_{x_0}) \Big\| ds\Big)^2 + \int_0^T \Big\| \sigma(s, x_0, \delta_{x_0})\Big\|^2 ds \Bigg) e^{(4L + 12L^2) T}.
\end{align*}
We proceed via induction. Let
$$
Y_t = X_t - \bE[ X_t]
$$
be the centred process. Then
\begin{equation}
\label{eq:prop:SSMVE-Moments1.1}
\bE\Big[ \| X_t - x_0 \|^{2m} \Big] \leq 2^{2m-1} \Big( \bE\Big[ \| X_t - x_0 \|^2\Big]^{m} + \bE\Big[ \| Y_t \|^{2m}\Big] \Big) . 
\end{equation}
Let $\xi$ and $\tilde{\xi}$ be independent copies of a scalar random variable with mean $0$. Then by the Binomial Theorem, we have that for $m\in \bN$, 
\begin{align*}
\bE\Big[ ( \xi - \tilde{\xi})^{2m} \Big] 
=
& \sum_{k=0}^{2m} (-1)^k \binom{2m}{k} 
\bE\Big[ \xi^k\Big] \bE\Big[ \xi^{2m - k}\Big], 
\end{align*}
and therefore from \cite{HIP}*{Proposition 2.12}
\begin{equation}\label{eq:prop:SSMVE-Moments1.2}
2\bE\Big[ \| Y_t \|^{2m} \Big] \leq c(m,d) \Big( \bE\Big[ \| X_t - \tilde{X}_t\|^{2m}\Big] +   \Big( 1 +  \bE\Big[ \| Y_t\|^{2m-2} \Big] \Big)^2 \Big),
\end{equation}
for a constant $c(m,d)$ depending only on $m$ and $d$. In what follows we write $c(m,d,L)$ for a constant possibly changing on each line, but dependent only on $m,d$ and Lipshitz constant $L$. We combine Equations \eqref{eq:prop:SSMVE-Moments1.1} and Equation \eqref{eq:prop:SSMVE-Moments1.2} to get
\begin{align}
\nonumber
 \bE\Big[ \| X_t &- x_0 \|^{2m} \Big] + \bE\Big[ \| X_t - \tilde{X_t} \|^{2m} \Big]
\\
&
\leq c(m,d,L) \Big(  \bE\Big[ \| X_t - x_0 \|^2\Big]^m + \Big( 1+ \bE\Big[ \| Y_t \|^{2m-2}\Big] \Big)^2 \Big) 
\label{eq:prop:SSMVE-Moments2.1}
+ c(m,d,L)\bE\Big[ \| X_t - \tilde{X_t} \|^{2m} \Big]. 
\end{align}
We use It\^o's formula to get that
\begin{align*}
\| X_t - \tilde{X}_t\|^{2m} =& \| \theta - \tilde{\theta}\|^{2m} + 2m \int_0^t \| X_s - \tilde{X}_s\|^{2m-2} \Big\langle X_s-\tilde{X}_s , b(s, X_s, \mu_s) - b(s, \tilde{X}_s, \mu_s) \Big\rangle ds
\\
&+ 2m \int_0^t \| X_s - \tilde{X}_s\|^{2m-2} \Big\langle X_s - \tilde{X}_s, \overline{\bE}\Big[ f( X_s - \overline{X}_s) - f(\tilde{X_s} - \overline{X}_s) \Big] \Big\rangle ds
\\
&+ 2m \int_0^t \| X_s - \tilde{X}_s\|^{2m-2} \Big\langle X_s - \tilde{X}_s, \sigma(s, X_s, \mu_s) dW_s - \sigma(s, \tilde{X}_s, \mu_s) d\tilde{W}_s \Big\rangle
\\
+m(2m-1)& \int_0^t \| X_s - \tilde{X}_s\|^{2m-2} \Big( \| \sigma(s, X_s, \mu_s) \|^2 + \| \sigma(s, \tilde{X}_s, \mu_s) \|^2 \Big) ds
- 2m\int_0^t \Big\langle X_s - \tilde{X}_s, dk_s - d\tilde{k}_s \Big\rangle,
\end{align*}

Now for any $K>0$, 
\begin{align*}
K& \sup_{t\in[0,T]} \bE\Bigg[ \int_0^t \| X_s - \tilde{X}_s \|^{2m-2} \Big( \| \sigma(s, X_s, \mu_s) \|^2 + \| \sigma(s, \tilde{X}_s, \mu_s) \|^2 \Big) ds \Bigg]
\\
\leq& 12L^2 K \int_0^T \bE\Big[ \| X_s - \tilde{X}_s \|^{2m} \Big] ds + \tfrac{12L^2K}{m} \int_0^T \bE\Big[ \| X_s - x_0\|^{2m} \Big] ds
\\
&+\sup_{t\in[0,T]} \frac{ \bE\Big[ \| X_t - \tilde{X}_t\|^{2m} \Big] }{2} + \big[ 2(m-1)\big]^{m-1} \cdot \Big[ \tfrac{6K}{m} \Big]^{m} \cdot \Big( \int_0^T \Big| \sigma(s, x_0, \delta_{x_0}) \Big|^2 ds \Big)^m.
\end{align*}

Applying this with Equation \eqref{eq:prop:SSMVE-Moments2.1} yields
\begin{align*}
&\sup_{t\in[0,T]} \bE\Big[ \| X_t - x_0 \|^{2m} \Big] + \bE\Big[ \| X_t - \tilde{X_t} \|^{2m} \Big]
\\
&\leq 
c(m,d,L)\Bigg(  \bE\Big[ \| X_t - x_0 \|^2\Big]^m +  \Big( 1+ \bE\Big[ \| Y_t \|^{2m-2}\Big] \Big)^2 
+  \bE\Big[ \| \theta - \tilde{\theta}\|^{2m} \Big] +\Big( \int_0^T \| \sigma(s, x_0, \delta_{x_0}) \|^2 ds \Big)^m
\\
&+   \int_0^T \sup_{s\in[0,t]} \bE\Big[ \| X_s - \tilde{X}_s \|^{2m} \Big] +  \bE\Big[ \| X_s - x_0\|^{2m} \Big] dt \Bigg)
+   \frac{1}{2}\sup_{t\in[0,T]} \bE\Big[ \| X_t - \tilde{X}_t\|^{2m}\Big] .
\end{align*}

Combining all terms together, we get that there exist a constant $c=c(m,d,L,T)$, dependent only on $m, d, L,T$ and not $T_0$ such that
\begin{align*}
\sup_{t\in[0,T_0]} \bE\Big[ \| X_t - x_0 \|^{2m} + \| X_t - \tilde{X_t} \|^{2m} \Big] 
\leq&
c\Bigg(1+ \int_0^{T_0} \sup_{s\in[0,t]} \bE\Big[ \| X_s - x_0 \|^{2m} + \| X_s - \tilde{X_s} \|^{2m} \Big] dt \Bigg).
\end{align*} 
Thus via Gr\"onwall
$$
\sup_{t\in[0,T_0]} \bE\Big[ \| X_t - x_0 \|^{2m} + \| X_t - \tilde{X_t} \|^{2m} \Big] 
\leq c e^{c T_0}< c e^{c T}.
$$
Hence, by induction we have finite moment estimates for all $m\in \bN$ such that $2m\leq p$. In particular, this is true for $2m\geq 2r$. For sharp moment estimates, we use the methods from the proof of Theorem \ref{thm:ExistUnique-LocLip-Ref} to get
\begin{align}
\bE\Big[ \| X - x_0 \|_{\infty, [0,T_0]}^p \Big] 
\lesssim&
\bE\Big[ \| \theta - x_0 \|^p\Big] + \Big(\int_0^{T_0} \| b(s, x_0, \delta_{x_0}) \| ds \Big)^p
\nonumber
\\
&+ \Big( \int_0^{T_0} \| \sigma(s, x_0, \delta_{x_0}) \|^2 ds\Big)^{p/2} + \Big( \int_0^{T_0} \Big\| \tilde{\bE}\Big[ f(\tilde{X_s} - x_0) \Big] \Big\| ds \Big)^p
\nonumber
\\
\lesssim& \bE\Big[ \| \theta - x_0 \|^p \Big] + \Big(\int_0^{T} \| b(s, x_0, \delta_{x_0}) \| ds \Big)^p
\nonumber
\\
&+ \Big( \int_0^{T} \| \sigma(s, x_0, \delta_{x_0}) \|^2 ds\Big)^{p/2} + \Big( T C \sup_{t\in[0,T_0]} \bE\Big[ \| X_t - x_0 \|^r + 1 \Big] \Big)^p. \label{eq momoment bounds for existence uniqueness}
\end{align}
\end{proof}

Finally, we are in position to prove Theorem \ref{thm:ExistUnique-LocLip-SSMVE}. 

\begin{proof}[Proof of Theorem \ref{thm:ExistUnique-LocLip-SSMVE}.]

By Proposition \ref{prop:SSMVE-ExistenceLocalSolution}, we have that a unique solution to Equation \eqref{eq:reflectedSSMVE} exists on the interval $[0,\varepsilon]$. Let $\delta>0$ and $g\in \Lambda_{[\varepsilon, \varepsilon + \delta],r}$. Then again by \eqref{eq:lem:Gamma-WellDefined-2}

\begin{align*}
\Big\| \Gamma[g] \Big\|_{[\varepsilon, \varepsilon+\delta], r}
\leq& \Big(2C+2^{r+1}\Big)\Big( 1+ \sup_{t\in[0,\varepsilon]} \bE\Big[ \| X_t - x_0 \|^r \Big] \Big)
\\
&+ \Bigg( \big( 4 (p-1)\big)^{p-1} \bigg( \Big( \int_{\varepsilon}^{\varepsilon+\delta} \| b(s, x_0, \delta_{x_0} )\| ds\Big)^p + \Big( \delta \|g\|_{[\varepsilon, \varepsilon+\delta],r} \Big)^p \bigg)
\\
&\quad + 2(p-1)^{p/2} \cdot (p-2)^{(p-2)/2} \cdot 4^{p/2} \Big( \int_{\varepsilon}^{\varepsilon+\delta} \| \sigma(s, x_0, \delta_{x_0} ) \|^2 ds \Big)^{\tfrac{p}{2}} \Bigg) 
\\
&\quad \cdot \exp\Big( \big( 4pL + 2p(p-1)L^2  \big) \delta \Big). 
\end{align*}

By Proposition \ref{prop:SSMVE-Moments}, we know that
$$
2\Big(2C+2^{r+1}\Big)\Big( 1+ \sup_{t\in[0,\varepsilon]} \bE\Big[ \| X_t - x_0 \|^r \Big] \Big)< K_5,
$$
for some $K_5$ independent of $\varepsilon$. Then for $\| g\|_{[\varepsilon, \varepsilon+\delta],r}<K_5$, we get
\begin{align*}
\Big\| \Gamma[g] \Big\|_{[\varepsilon, \varepsilon+\delta], r}
\leq& \tfrac{K_5}{2} + \Bigg( \big( 4 (p-1)\big)^{p-1} \bigg( \Big( \int_{\varepsilon}^{\varepsilon+\delta} \| b(s, x_0, \delta_{x_0} )\| ds\Big)^p + \big( \delta K_5 \big)^p \bigg)
\\
&\quad + 2(p-1)^{p/2} \cdot (p-2)^{(p-2)/2} \cdot 4^{p/2} \Big( \int_{\varepsilon}^{\varepsilon+\delta}  \| \sigma(s, x_0, \delta_{x_0} ) \|^2 ds \Big)^{\tfrac{p}{2}} \Bigg) 
\\
&\quad \cdot \exp\Big( \big( 4pL + 2p(p-1)L^2  \big) \delta \Big). 
\end{align*}

By the uniform continuity of the mappings
\begin{align*}
\delta \mapsto& \int_{\varepsilon}^{\varepsilon+\delta} \| b(s, x_0, \delta_{x_0} )\| ds 
\quad \mbox{and}\quad
\delta \mapsto \int_{\varepsilon}^{\varepsilon+\delta} \| \sigma(s, x_0, \delta_{x_0} )\|^2 ds, 
\end{align*}
we choose $\delta'>0$ (independently of $\varepsilon$) so that $\big\|\, \Gamma[g]\, \big\|_{[\varepsilon, \varepsilon+\delta'], r}<K_5$. Next, we use Equation \eqref{eq:prop:Gamma-FirstContraction} to get

\begin{align*}
\Big\|& \Gamma[g_1] - \Gamma[g_2] \Big\|_{[\varepsilon,\varepsilon+\delta], r} 
\\
&\leq (C+2^r) 3\sqrt{8} \| g_1 - g_2\|_{[\varepsilon,\varepsilon+\delta], r} \sqrt{\delta} e^{(4L^2+6L)\delta} \Bigg( 1+ \sup_{t\in [\varepsilon, \varepsilon+\delta]} \bE\Big[ \| X_t^{(g_1)} - x_0 \|^{2r} + \|X_t^{(g_2)} - x_0 \|^{2r}\Big] \Bigg) . 
\end{align*}

Next, using Equation \eqref{eq:reflected-SSMVE-b:moment}, we get

\begin{align*}
\Big\|
\Gamma[g_1] - \Gamma[g_2] \Big\|_{[\varepsilon,\varepsilon+\delta], r} 
&\leq (C+2^r) 3\sqrt{8} \| g_1 - g_2\|_{[\varepsilon,\varepsilon+\delta], r} \sqrt{\delta} e^{(4L^2+6L)\delta} \Bigg( 1+ 8\sup_{t\in[0,\varepsilon]} \bE\Big[ | X_{t} - x_0|^{2r} \Big]
\\
&\quad +2\big( 4 (2r-1)\big)^{2r-1} \bigg( \Big( \int_{\varepsilon}^{\varepsilon+\delta} | b(s, x_0, \delta_{x_0} )| ds \Big)^{2r} + \big( \delta K_5 \big)^{2r} \bigg)
\\
&\quad +4(2r-1)^{r} \cdot (2r-2)^{r-1} \cdot 4^{r} \Big( \int_{\varepsilon}^{\varepsilon+\delta}  |\sigma(s, x_0, \delta_{x_0} ) |^2 ds \Big)^{r}
\Bigg) e^{\big(8rL + 4r(2r-1)L^2 \big) \delta}.
\end{align*}

Finally, by Proposition \ref{prop:SSMVE-Moments}, we choose $\delta''>0$ (independently of $\varepsilon$) such that
$$
\Big\| \Gamma[g_1] - \Gamma[g_2] \Big\|_{[\varepsilon,\varepsilon+\delta''], r} 
\leq
\frac{1}{2} \| g_1 - g_2\|_{[\varepsilon,\varepsilon+\delta''], r}. 
$$
Let $\delta = \min\{ \delta', \delta''\}$. 

Define $d:\Lambda_{[\varepsilon,\varepsilon + \delta], r} \times \Lambda_{[\varepsilon,\varepsilon + \delta], r} \to \bR^+$ be the metric $d(g_1, g_2) = \| g_1 - g_2\|_{[\varepsilon,\varepsilon + \delta], r}$. The metric space $(\Lambda_{[\varepsilon,\varepsilon + \delta], r, K_3}, d)$ is non-empty, complete and $\Gamma:\Lambda_{[\varepsilon,\varepsilon + \delta], r, K_3} \to \Lambda_{[\varepsilon,\varepsilon + \delta], r, K_3}$ is a contraction operator. Therefore, $\exists g'\in \Lambda_{[\varepsilon,\varepsilon + \delta], r, K_3}$ such that $\Gamma[g'] = g'$.

Thus $\forall t\in [\varepsilon,\varepsilon + \delta]$, 
$$
g'\big(t, X_t^{(g')} \big) = f\ast \mu_t^{(g')} \big(X_t^{(g')}\big).
$$
Repeating this argument and concatenating, we obtain a function $g\in \Lambda_{[0,T], r}$ such that $\forall t\in[0,T]$
$$
g\big(t, X_t^{(g)} \big) = f\ast \mu_t^{(g)} \big(X_t^{(g)}\big). 
$$
Substituting this into Equation \eqref{eq:reflected-SSMVE-b}, we obtain Equation \eqref{eq:reflectedSSMVE} over the interval $[0,T]$.
\end{proof}

\subsection{Propagation of chaos}
\label{sec:PoC}

We are interested in the ways in which the dynamics of a single equation within a system of reflected interacting equations of the form  \eqref{eq:ParticleSystem} converges to the  dynamics of the reflected McKean-Vlasov equation.  

Let $N\in \bN$ and let $i\in\{1, ..., N\}$. We now study the law of a solution to the interacting particle system
\begin{equation}
\label{eq:reflectedSSPS}
\begin{split}
X_t^{i, N} =& \theta^{i} + \int_0^t b(s, X_s^{i, N}, \mu_s^N) ds + \int_0^t \sigma(s, X_s^{i, N}, \mu_s^N) dW_s^{i, N} + \int_0^t f \ast \mu_s^N (X_s^{i, N}) ds - k_t^{i, N}, 
\\
|k^{i, N}|_t=& \int_0^t \1_{\partial D}(X_s^{i, N}) d|k^{i, N}|_s, 
\qquad k_t^{i, N} = \int_0^t \1_{\partial \cD}(X_s^{i, N}) \n(X_s^{i, N}) d|k^{i, N}|_s, 
\qquad \mu_t^N = \tfrac{1}{N} \sum_{j=1}^N \delta_{X_t^{j, N}}.
\end{split}
\end{equation}

We demonstrate Propagation of Chaos (PoC), that is for a finite time interval $[0,T]$ the trajectories of the particle system on average converge to that of the McKean-Vlasov equation. 

\begin{theorem}[Propagation of Chaos (PoC)] 
\label{thm:PoC}
Let $\cD\subset \bR^d$ satisfy Assumption \ref{assumption:domain}. \textcolor{black}{ Let $\theta^i$ be independent identically distributed copies of $\theta$, and let $\theta$, $b$, $\sigma$ and $f$ satisfy Assumption \ref{ass:ExistUnique-LocLip-SSMVE}. Let $W^{i, N}$ be a sequence of independent Brownian motions taking values on $\bR^{d'}$. Additionally, suppose that $p>\max \{ 2r, 4\}$.  Let $X^i_t$ be a sequence of strong solutions to Equation \eqref{eq:reflectedSSMVE} driven by the Brownian motion $W^{i,N}$, and with initial conditions $\theta^i$}. Let $X_t^{i, N}$ be the solution to particle system \eqref{eq:reflectedSSPS}.

Then there exists a constant $c=c(T)>0$, depending only on $T$, such that
\begin{equation}
\sup_{t\in[0,T]} \bE\Big[ \| X^{i,N}_t - X^i_t \|^2\Big] \leq c(T) \begin{cases} 
N^{-1/2},~&d<4,
\\
N^{-1/2}\log N,~&d=4,
\\
N^{\frac{-2}{d+4}},~&d>4.
\end{cases}
\end{equation}
\end{theorem}

\begin{proof}
Firstly, we assume that the noise driving the McKean-Vlasov equation \eqref{eq:reflectedSSMVE} and the noise driving the particle system \eqref{eq:reflectedSSPS} have correlation 1. Using  It\^o's formula, summing over $i$ and taking expectations, 

\color{black}
\begin{align}
\nonumber
\sum_{i=1}^N \bE\Big[ \| X_t^{i, N} - X_t^i\|^2\Big] \leq&
2L \int_0^t \sum_{i=1}^N \bE\Big[ \| X_s^{i,N} - X_s^{i} \|^2 \Big] ds 
 + 2L \int_0^t \sum_{i=1}^N \bE\Big[ \|X_s^{i,N} - X_s^{i} \| \cdot \bW_{\cD}^{(2)} (\mu^N_s, \mu_s) \Big] ds
\\
\nonumber
&+ 4L^2 \int_0^t \sum_{i=1}^N \bE\Big[ \| X_s^{i,N} - X_s^i \|^2  + \bW_{\cD}^{(2)} \Big( \mu_s^N, \mu_s \Big)^2\Big] ds
\\
\label{eq:thm:PoC1.1}
& + 2\int_0^t \sum_{i=1}^N \bE\Big[ \Big\langle X_s^{i, N} - X_s^i, \tfrac{1}{N}\sum_{j=1}^N f(X_s^{i, N} - X_s^{j, N}) - f(X_s^i - X_s^{j})  \Big\rangle \Big]  ds
\\
\label{eq:thm:PoC1.2}
& +2 \int_0^t \sum_{i=1}^N \bE\Big[ \Big\langle X_s^{i, N} - X_s^i, \tfrac{1}{N}\sum_{j=1}^N f(X_s^i - X_s^j) - f\ast \mu_s(X_s^i) \Big\rangle \Big]  ds.
\end{align}

Re-arranging the double sum and using that $f$ is odd, we can rewrite the integrand of \eqref{eq:thm:PoC1.1} as
\begin{align}
\sum_{i, j=1}^N \bE\Big[  
\Big\langle X_s^{i, N} - X_s^i,&  f(X_s^{i, N} - X_s^{j, N}) - f(X_s^i - X_s^{j})  \Big\rangle \Big]
\nonumber
\\
&
=\frac{1}{2}\sum_{i,j=1}^N \bE\Big[  \Big\langle (X_s^{i, N} - X_s^{j,N} )-(X_s^i-X_s^j),
f(X_s^{i, N} - X_s^{j, N}) - f(X_s^i - X_s^{j})  \Big\rangle \Big],
\label{v1}
\end{align}
and thus using the one-sided Lipschitz property of $f$ we can bound \eqref{v1} by $L\sum_{i=1}^N \bE\big[ \| X_s^{i, N} - X_s^i \|^2 \big] $. 

Consider the sum over $j$ in the integrand of \eqref{eq:thm:PoC1.2}. One observes that after using the Cauchy Schwartz inequality we have the product of the two terms
\begin{align}
    \nonumber
    \bE\Big[ \Big\langle X_s^{i, N} - X_s^i, \sum_{j=1}^N \big( f(X_s^i 
    & - X_s^j) - f\ast \mu_s(X_s^i) \big) \Big\rangle \Big]
    \\
    &\leq
    \bE\Big[ \| X_s^{i, N} - X_s^i\|\Big]^{1/2} \bE\Big[ \big\|\sum_{j=1}^N \big( f(X_s^i - X_s^j) - f\ast \mu_s(X_s^i) \big) \big\|^2 \Big]^{1/2}.
    \label{x3}
    \end{align}
    We next show that the second of these terms is bounded by $C \sqrt{N}$ for some fixed constant $C>0$. We have 
    \begin{align}
        \bE\Big[ \big\|\sum_{j=1}^N \big( f(X_s^i - X_s^j) - f\ast \mu_s(X_s^i) \big) \big\|^2 \Big]
        &= \sum_{j,k=1}^N\bE\Big[ \big\langle f(X_s^i - X_s^j) - f\ast \mu_s(X_s^i),f(X_s^i - X_s^k) - f\ast \mu_s(X_s^i) \big\rangle  \Big]
        \nonumber
        \\
        &=\sum_{j=1}^N \bE\Big[ \big\| f(X_s^i - X_s^j) - f\ast \mu_s(X_s^i)\big\|^2 \Big]
        \label{x1}
        \\
       & \leq C N
       \label{x2}
    \end{align}
where \eqref{x1} is due to the fact that the cross terms (i.e., $i\neq j$) are all zero since in this case $X^j$ is independent of $X^k$, and \eqref{x2} follows from the polynomial growth of $f$ and the control on the moments $\bE[\|X^i_s\|^{2r}]$. Using \eqref{x3} in conjunction with \eqref{x2}, it is clear that the integrand in \eqref{eq:thm:PoC1.2} is some constant multiple of $\sqrt{N}+\frac{1}{\sqrt{N}}\sum_{i=1}^N \bE[\|X^{i,N}_s-X^i_s\|^2]$ (from the inequality $|x|\leq 1+|x|^2$). Next, dealing with the $\bW_{\cD}^{(2)}( \mu_\cdot^N, \mu_\cdot)$ terms, set $\nu_\cdot^N = \tfrac{1}{N} \sum_{j=1}^N \delta_{X_\cdot^{j}}$. By the triangle inequality, we get
\begin{align}
\bE\Big[ \bW_{\cD}^{(2)}( \mu_s^N, \mu_s) \Big]
&
\leq
\bE\Big[ \Big( \tfrac{1}{N}\sum_{i=1}^N \| X_s^{i, N} - X_s^{i} \|^2 \Big)^{1/2}  + \bW_{\cD}^{(2)}( \nu_s^N, \mu_s) \Big].
\label{v4}
\end{align}
Assembling all the previous bounds with the estimate obtained after applying It\^o's formula, we get 
\begin{align}
\nonumber
\sum_{i=1}^N \bE\Big[ \| X_t^{i, N} - X_t^i\|^2\Big] \lesssim&
 \int_0^t \sum_{i=1}^N \bE\Big[ \| X_s^{i,N} - X_s^{i} \|^2 \Big]  ds + t\sqrt{N}
+ N \int_0^t  \bW_{\cD}^{(2)} (\mu^N_s, \mu_s) \Big] ds.
\end{align}
Noting that the particles are exchangeable, and taking the supremum over $t\in[0,T]$ we find that 
\begin{align*}
\sup_{t\in[0,T]} \bE\Big[ \| X_t^{i, N} - X_t^{i} \|^2 \Big]
\lesssim& 
 \int_0^T \sup_{t\in[0,s]} \bE\Big[ \|X_s^{i, N} - X_s^i\|^2 \Big] ds
+ T\Big(\frac{1}{\sqrt{N}}+\sup_{t\in[0,T]} \bE\Big[ \bW_{\cD}^{(2)} \Big(\nu^N_t, \mu_t \Big)^2 \Big]\Big).
\end{align*}
Applying Gr\"onwall inequality yields 
\begin{align*}
\sup_{t\in[0,T]} \bE\Big[ \| X_t^{i, N} - X_t^{i} \|^2 \Big] 
\lesssim
T\Big(\frac{1}{\sqrt{N}}+\sup_{t\in[0,T]} \bE\Big[ \bW_{\cD}^{(2)} \Big(\nu^N_t, \mu_t \Big)^2 \Big]\Big).
\end{align*}
Finally, by assumption on $p$ all processes have moments larger the 4th one, thus one can use the well known rate of convergence for an empirical distribution to the true law, see  \cite{carmona2018probabilistic}*{Theorem 5.8}, and obtain
$$
\bE\Big[ \bW_{\cD}^{(2)} \Big(\nu^N_t, \mu_t \Big)^2 \Big] \lesssim \begin{cases} 
N^{-1/2},~&d<4,
\\
N^{-1/2}\log N,~&d=4,
\\
N^{\frac{-2}{d+4}},~&d>4,
\end{cases}
$$
to conclude. Note that the latter convergence rate dominates the $T/\sqrt{N}$ element in the main error estimate.
\end{proof}
\color{black}
\subsection{An example}

A key advantage of the framework that we consider for Theorem \ref{thm:ExistUnique-LocLip-MVE} and Theorem \ref{thm:ExistUnique-LocLip-SSMVE} is that the drift term $b$ is locally Lipschitz over $\cD$. We demonstrate that the measure dependencies allowed for with the self-stabilizing term $f \ast \mu$ do not satisfy a Lipschitz condition with respect to the Wasserstein distance. 
\begin{example}
Let $\cD = \bR^+$. Let $F(x) = {-x^4}/{4}$ so that $f(x) = \nabla F(x) = -x^3$. Consider the dynamics
$$
X_t = W_t - \int_0^t \int_{\cD} (X_s - y)^3 \mu_t(dy) ds - k_t, \quad \mu_t (dx) = \bP\big[ X_t \in dx\big], \quad X_0 = 1.
$$

Without entering details and assuming $\mu, \nu \in \cP_4(\cD)$, the Lions derivative of $\mu\mapsto \Psi_x(\mu):= -\int_\cD (x-y)^3\mu(dy)$ is unbounded, meaning that the "Lipschitz" constant of $\mu\mapsto \Psi_x(\mu)$ depends on $x$ in an unbounded way since $\cD$ is unbounded.

For the reader familiarised with the theory, see \cite{carmona2018probabilistic}*{Section 5}, the Lions derivative of the functional $\Psi_x(\cdot)$ follows from Example 1 in Section 5.2.2 (p385) and is given by $\partial_\mu \psi_x(\mu)(Z)=f'(x-Z)$ for $Z\sim \mu$. Their Remark 5.27 (p384) and Remark 5.28 (p390) connect to the Lipschitz constant.
\end{example}

\section{Large Deviation Principles}
\label{sec:LDPs}

Throughout this section let $\varepsilon>0$, all results hold under the following assumptions:

\begin{assumption}
\label{assumption : Holder regularity of sigma}
Suppose that $\cD\subset \bR^d$ satisfies Assumption \ref{assumption:domain}. Suppose that $b,\sigma,$ and $f$ satisfy Assumptions \ref{ass:ExistUnique-LocLip-SSMVE}. Additionally, suppose that $\exists L>0, \exists \beta\in(0,1]$ such that  $\forall s,t\in[0,T]$, $\forall \mu \in \cP_2(\cD)$ and $\forall x \in \cD$,
\begin{equation*}
\|\sigma(t,x,\mu)-\sigma(s,x,\mu)\|\leq L\|t-s\|^{\beta}.
\end{equation*}
\end{assumption}
The regularity on $\sigma$ imposed above will allow us to make an Euler scheme approximation to the dynamics. We begin by reminding the reader of the definition of a Freidlin-Wentzell Large Deviation Principle. 
\begin{defn}
Let $E$ be a metric space. A function $I:E \to [0,\infty]$ is said to be a \emph{rate function} if it is lower semi-continuous and the level sets of $I$ are closed. A \emph{good rate function} is a rate function whose level sets are compact. 
\end{defn}

The rate function is used to encode the asymptotic rate for a convergence in probability statement that is called a Large Deviations Principle. 

\begin{defn}
\label{definition LDP}
Let $x\in \cD$. A family of probability measures $\{\mu^{\varepsilon}\}_{\varepsilon>0}$ on $C_x([0,T]; \cD)$ is said to satisfy a Large Deviations Principle with rate function $I$ if  
\begin{equation}
-\inf_{h\in G^\circ} I(h) \leq
\liminf_{\varepsilon \to 0} \varepsilon \log\mu^\varepsilon [G^\circ] \leq 
\limsup_{\varepsilon \to 0}\varepsilon \log\mu^{\varepsilon}[\overline{G}] \leq 
-\inf_{h\in \overline{G}} I(h),
\end{equation}
for all Borel subsets $G$ of the space $C_x([0,T]; \cD)$.
\end{defn}

We prove a Freidlin-Wentzell Large Deviation Principle for the class of 
reflected McKean-Vlasov equations studied in Section \ref{section exist.unique}. The inclusion of non-Lipschitz measure dependence and reflections extends the classical Freidlin-Wentzell results for SDEs found in \cites{DZ,deuschel2001large,den2008large}. 

Our approach uses sequences of exponentially good approximations, inspired by the methods of \cite{HIP} and \cite{dos2019freidlin}. As with previous works proving Freidlin-Wentzell LDP results for McKean-Vlasov SDEs, the non-Lipschitz measure dependency is accounted for by establishing an LDP for a diffusion that is an exponentially tight approximation. 

The section is structured as follows, first a deterministic path is identified  which the solution to \eqref{eq:MVSS-LDP} approaches as $\varepsilon \to 0$. Definition \eqref{equation Y classical reflected SDE} then introduces an approximation  of \eqref{eq:MVSS-LDP} where the law is replaced by this deterministic path. An LDP is established for this approximation by first obtaining an LDP for its Euler scheme in Lemma \ref{lemma the LDP for Y classical euler reflected sde}, and then transferring it via the method of exponential approximations in Lemmas \ref{lemma : euler scheme is an expo good approximation} and \ref{lemma the LDP for Y classical reflected sde}. Finally the LDP for the object of interest \eqref{eq:MVSS-LDP} is acquired by establishing exponential equivalence between it and the approximation of Definition \ref{definition Y classical reflected SDE}. 

\subsection{Convergence of the law}

Recall that the key point of an LDP is to characterise the rate at which the probability of rare events decreases as we change a parameter in our experiment. In the case of path space LDP for a stochastic processes this relies on identifying a path which the diffusion increasingly concentrates around as the noise decays. The dynamics of the process can then be seen as small perturbations from this fixed path, often referred to as the skeleton path. Consider the reflected McKean-Vlasov SDE
\begin{equation}
\label{eq:MVSS-LDP}
\begin{split}
X^{\varepsilon}_t =& x_0 + \int_0^t b(s,X^{\varepsilon}_s, \mu^{\varepsilon}_s) ds + \int_0^t f\ast \mu_s^\varepsilon ( X^{\varepsilon}_s) ds + \sqrt{\varepsilon} \int_0^t \sigma(s,X^{\varepsilon}_s,\mu^\varepsilon_s) dW_s - k_t^\varepsilon, 
\\
|k^\varepsilon|_t =& \int_0^t \1_{\partial \cD}(X_s^\varepsilon) d|k^\varepsilon|_s, \quad
k^\varepsilon_t = \int_0^t \1_{\partial \cD}(X_s^\varepsilon) \n(X_s^\varepsilon) d|k^\varepsilon|_s.
\end{split}
\end{equation}

Heuristically, as $\varepsilon \to 0$ the noise term  in \eqref{eq:MVSS-LDP} vanishes, the law of $X^\varepsilon$ tends to a Dirac measure of its own deterministic trajectory and hence the interaction term vanishes. Therefore in the small noise limit the dynamics is governed by $b$ and the diffusion behaves like the solution to the following deterministic Skorokhod problem.

\begin{defn}\label{definition skeleton process}
Define $\psi^{x_0}$ to be the solution to the reflected ODE
\begin{equation}
\label{eq:SkeletonProcess-0}
\begin{split}
\psi^{x_0}(t) =& x_0 + \int_0^t b(s,\psi^{x_0}(s),\delta_{\psi^{x_0}(s)} )ds - k_t^\psi, 
\\
|k^\psi|_t =& \int_0^t \1_{\partial \cD}(\psi(s) ) d|k^\psi|_s, \quad
k^\psi_t = \int_0^t \1_{\partial \cD}(\psi(s) ) \n(\psi(s) ) d|k^\psi|_s, 
\end{split}
\end{equation}
on the interval $[0,T]$. We define the Skeleton operator $H: \cH_1^0 \to C_{x_0}([0,T]; \cD)$ by $h \mapsto H[h]$ where
\begin{equation}
\label{eq:SkeletonProcess-h}
\begin{split}
H[h]_t =& x_0 + \int_0^t b(s, H[h]_s, \delta_{\psi^{x_0}(s)}) ds + \int_0^t f(H[h]_s - \psi^{x_0}(s)) ds + \int_0^t \sigma( s, H[h]_s, \delta_{\psi^{x_0}(s)}) dh_s - k_t^h,
\\
|k^h|_t =& \int_0^t \1_{\partial \cD}(H[h]_s) d|k^h|_s, \quad
k^h_t = \int_0^t \1_{\partial \cD}(H[h]_s) \n(H[h]_s) d|k^h|_s. 
\end{split}
\end{equation}
\end{defn}

The existence of a unique solution to the Skorokhod problem for a continuous path into a convex domain \cite{tanaka2002stochastic}*{Theorem 2.1} ensures the existence and uniqueness of a solution to Equation \eqref{eq:SkeletonProcess-h}, this can we proved in a similar and fashion to \cite{tanaka2002stochastic}*{Theorem 4.1}. Hence the operator $H[h]$ is well defined. \\

The following lemma proves that, for small $\epsilon$, the solution $X^\epsilon$ to \eqref{eq:MVSS-LDP} will remain close to the trajectory $\psi^{x_0}$ of the skeleton ODE \eqref{eq:SkeletonProcess-0}. Moreover the law $\mu^\varepsilon$ can be shown to tend to the Dirac measure of  $\psi^{x_0}$.

\begin{lemma}
Let $X^\varepsilon$ be the solution to \eqref{eq:MVSS-LDP} and $\mu^\varepsilon$ its law. Let $\psi^{x_0}$ be the solution of \eqref{eq:SkeletonProcess-0}. Then we have for any $T> 0$, 
\begin{equation}\label{equation X goes to skeleton}
\sup_{t\in[0,T]} \bE\Big[ \|X_t^\varepsilon-\psi^{x_0}(t) \|^2 \Big] \leq \varepsilon T e^{cT}, 
\end{equation}
for a constant $c$ independent of $\varepsilon$ and $x_0$. Moreover for any $x\in \bR^d$ we have that 
\begin{equation}\label{equation convergence of law to the Diract path}
\lim_{\varepsilon \to 0} \| f\ast \mu_t^\varepsilon (x) - f(x - \psi^{x_0}(t)) \|_{\infty, [0,T]}=0.
\end{equation}
\end{lemma}
\begin{proof} Let $t\in [0,T]$. We have
\begin{align*}
\| X_t^\varepsilon - \psi^{x_0}(t)\|^2 =& 2\int_0^t \Big\langle X_s^\varepsilon - \psi^{x_0}(s), b(s, X_s^\varepsilon, \mu_s^\varepsilon) - b(s, \psi^{x_0}(s), \delta_{\psi^{x_0}(s)} ) \Big\rangle ds
\\
&+ \sqrt{\varepsilon} \int_0^t \Big\langle X_s^\varepsilon - \psi^{x_0}(s), \sigma(s, X_s^\varepsilon, \mu_s) dW_s \Big\rangle  
+\varepsilon \int_0^t \| \sigma(s, X_s^\varepsilon, \mu_s^\varepsilon) \|^2 ds
\\
&+ \int_0^t \Big\langle X_s^\varepsilon - \psi^{x_0}(s), f(X_s^\varepsilon) \ast \mu_s^\varepsilon \Big\rangle ds
- \int_0^t \Big\langle X_s^\varepsilon - \psi^{x_0}(s), dk^\varepsilon_s - dk^{\psi}_s\Big\rangle.
\end{align*}
Thus
\begin{align*}
\sup_{t\in[0,T]} \bE\Big[ \| X_t^\varepsilon - \psi^{x_0}(t)\|^2 \Big] 
\leq& 
6L \int_0^T \sup_{s\in[0,t]} \bE\Big[ \| X_s^\varepsilon - \psi^{x_0}(s)\|^2 \Big] ds
\\
& + C \cdot \sup_{t\in[0,T]} \bE\Big[ \Big(1 + \| X_t^\varepsilon - \psi(t)\|^{r}\Big)^2 \Big]^{1/2} \cdot \int_0^T \sup_{s\in[0,t]} \bE\Big[ \| X_s^\varepsilon - \psi^{x_0}(s)\|^2 \Big] dt
\\
&+ \varepsilon \Big( 6TL^2 \sup_{t\in[0,T]} \bE\Big[ \| X_t^\varepsilon - x_0\|^2\Big] + 3 \int_0^T \| \sigma(t, x_0, \delta_{x_0}) \|^2 dt \Big).
\end{align*}
Therefore we can conclude \eqref{equation X goes to skeleton} from the finite moment estimates proved in Proposition \ref{prop:SSMVE-Moments} and Gr\"onwall's inequality. Next, \eqref{equation convergence of law to the Diract path} follows from 
 \eqref{equation X goes to skeleton}
\begin{align*}
\sup_{t\in[0,T]} \| & f \ast \mu_t^\varepsilon(x) - f(x-\psi^{x_0}(t) \| 
\\
&
\leq
C \sup_{t\in[0,T]} \bE\Big[ \|X_t^\varepsilon - \psi^{x_0}(t) \|^2 \Big]^{1/2} \cdot \bE\Big[ \Big( 1+ \| X_t^\varepsilon \|^{r-1} + \| \psi^{x_0}(t)\|^{r-1} \Big)^2 \Big]^{1/2} 
\underset{\varepsilon \to 0}{\longrightarrow}0.
\end{align*}
\end{proof}

\subsection{A classical Freidlin-Wentzell result}

Since the law $\mu^{\varepsilon}$ tends to the Dirac mass of the path $ \psi^{x_0}$, we will first study SDEs where the law in the coefficients of the McKean-Vlasov equation has been replaced by $\delta_{\psi^{x_0}}$. 

\begin{defn}\label{definition Y classical reflected SDE}
Let $Y^{\varepsilon}$ be the solution of 
\begin{equation}
\label{equation Y classical reflected SDE}
\begin{split}
Y^{\varepsilon}_t =& x_0 + \int_0^t b(s,Y^{\varepsilon}_{s},\delta_{\psi^{x_0}(s)})ds +  \int_0^t f\Big( Y^{\varepsilon}_{s} - \psi^{x_0}(s) \Big)ds + \sqrt{\varepsilon} \int_0^t \sigma(s,Y^\varepsilon_s,\delta_{\psi^{x_0}(s)}) dW_s - k^{Y}_{t}, 
\\
|k^{Y}|_t =& \int_0^t \1_{\partial \cD}(Y^\varepsilon_s) d|k^{Y}|_s, 
\qquad
k^{Y}_t = \int_0^t \1_{\partial \cD}(Y^\varepsilon_s) \n(Y^\varepsilon_s) d|k^{Y}|_s.
\end{split}
\end{equation}
\end{defn}

The dynamics of \eqref{equation Y classical reflected SDE} satisfy those of Theorem \ref{thm:ExistUnique-LocLip-Ref}, so the existence and uniqueness of a solution is established. Further, we introduce the follow approximation of \eqref{equation Y classical reflected SDE}.

\begin{defn}
\label{equation Y classical reflected SDE - euler}
Let $n\in \bN$. Let $Y^{n,\varepsilon}$ be the solution of
\begin{align}
\nonumber
Y_{t}^{n,\varepsilon}=& x_0 + \int_0^t b(s,Y^{n,\varepsilon}_{s},\delta_{\psi^{x_0}(s)})
+  f\Big( Y^{n,\varepsilon}_{s} - \psi^{x_0}(s) \Big)ds 
\\
\nonumber
&\sqrt{\varepsilon}\sum_{i=0}^{\lfloor \frac{tn}{T} \rfloor - 1} \sigma\Big( \tfrac{iT}{n} ,Y^{n,\varepsilon}_{\tfrac{iT}{n} },\delta_{\psi^{x_0}\big(\tfrac{iT}{n} \big)} \Big)\cdot \Big( W_{\tfrac{(i+1)T}{n}} - W_{\tfrac{iT}{n}} \Big) 
\\
\label{equation Y classical reflected SDE Euler scheme}
&+ \sqrt{\varepsilon} \sigma\Big( \tfrac{T\lfloor \frac{tn}{T}\rfloor}{n} ,Y^{n,\varepsilon}_{\tfrac{T\lfloor \frac{tn}{T}\rfloor}{n} },\delta_{\psi^{x_0}\big(\tfrac{T\lfloor \frac{tn}{T}\rfloor}{n} \big)} \Big) \Big( W_{\tfrac{T\lceil \frac{tn}{T}\rceil}{n}} - W_{\tfrac{T\lfloor \frac{tn}{T}\rfloor}{n}} \Big) n\Big( t - \tfrac{T\lfloor \frac{tn}{T}\rfloor}{n}\Big) - k^{Y^{n,\varepsilon}}_{t}
\\
\nonumber
|k^{Y^{n,\varepsilon}}|_t =& \int_0^t \1_{\partial \cD}(Y_s^{n, \varepsilon}) d|k^{Y^{n,\varepsilon}}|_s, 
\qquad
k^{Y^{n,\varepsilon}}_t = \int_0^t \1_{\partial \cD}(Y_{s}^{n,\varepsilon}) \n(Y_{s}^{n,\varepsilon}) d|k^{Y^{n,\varepsilon}}|_s.
\end{align}
\end{defn}

On a subset of measure 1, Equation \eqref{equation Y classical reflected SDE Euler scheme} determines the dynamics of a random ODE for which the Skorokhod problem has already been solved, so existence and uniqueness are already assured. 

\begin{defn}
Let $I': C_{0}([0,T]; \bR^d) \to \mathbb{R}$ be the rate function of Schilder's Theorem \cite{DZ}*{Theorem 5.2.3},
\begin{equation*}
I'(g) =  \begin{cases}\frac{1}{2}\int_0^T\| \dot{g}(t) \|^2dt~& \text{if}~ g\in \cH^0_1, \\ \infty~& \text{otherwise}, \end{cases}
\end{equation*}
where $\cH^0_1$ is the Cameron Martin space for Brownian motion defined in Section \ref{sec:Preliminaries}.
\end{defn}

Define the functional $H^n: C_{0}([0,T]; \bR^d) \to C_{x_{0}}([0,T]; \bR^d)$, which maps the Brownian path to the reflected path of \eqref{equation Y classical reflected SDE Euler scheme}, that is
\begin{align} 
\nonumber
H^n[h](t) =& x_0  + \int_0^t b\big(s,H^n[h](s),\delta_{\psi^{x_0}(s)}\big) + f\Big(H^n[h](s) - \psi^{x_0}(s)\Big) ds - k^{h,n}_t 
\\
\nonumber
&+ \sum_{i=0}^{\lfloor \frac{tn}{T} \rfloor -1} \sigma\Big(\frac{iT}{n},H^n[h]\Big(\frac{iT}{n}\Big),\delta_{\psi^{x_0}(\frac{iT}{n})} \Big) \Big(h\Big(\frac{(i+1)T}{n}\Big)-h\Big(\frac{iT}{n}\Big)\Big) 
\\
\label{e30}
&+ \sigma\Big(\frac{T\lfloor \frac{tn}{T} \rfloor }{n},H^n[h]\Big(\frac{T\lfloor \frac{tn}{T} \rfloor }{n}\Big),\delta_{\psi^{x_0}(\frac{T\lfloor \frac{tn}{T} \rfloor }{n})} \Big)\Big( h\Big( \tfrac{T\lceil \frac{tn}{T}\rceil}{n} \Big) - h\Big(\tfrac{T\lfloor \frac{tn}{T} \rfloor  }{n}\Big)\Big)\frac{n}{T} \Big(t-\frac{T\lfloor \frac{tn}{T} \rfloor}{n}\Big), 
\\
|k^{h,n}|_t =& \int_0^t \1_{\partial \cD}(H^n[h](s)) d|k^{h,n}|_s, \quad
k^{h,n}_t = \int_0^t \1_{\partial \cD}(H^n[h](s)) \n(H^n[h](s)) d|k^{h,n}|_s. 
\nonumber
\end{align}

When restricted to $\cH_1^0$, the operator $H^n$ represents a Skeleton operator for the random ODE \eqref{equation Y classical reflected SDE Euler scheme}. Equation \eqref{equation Y classical reflected SDE} is a classical reflected SDE and \cite{dupuis1987large}*{Theorem 3.1} proves a Freidlin-Wentzell type LDP for such reflected SDEs when the coefficients are bounded and Lipschitz. The following lemma extends this result to unbounded domains and allows for unbounded locally Lipschitz coefficients, this is done via the contraction principle \cite{DZ}*{Theorem 4.2.1}. For convenience of notation let

\begin{align*}
   \hat{t}:=\frac{T\lceil \frac{tn}{T}\rceil}{n},~ \check{t}:=\frac{T\lfloor \frac{tn}{T}\rfloor}{n},~\text{and}~ \hat{s}:=\frac{T\lceil \frac{sn}{T}\rceil}{n},~ \check{s}:=\frac{T\lfloor \frac{sn}{T}\rfloor}{n}.
\end{align*}
\begin{lemma}
\label{lemma : continuity of the map Hn}
For each $n\in \bN$, the mapping $H^n :C_0([0,T];\bR^d) \to C_{x_{0}}([0,T];\bR^d)$ defined by \eqref{e30} is continuous.
\end{lemma}

\begin{proof}

Let $\{h_m: m\in \bN\} \subset C_0([0,T]; \bR^d)$ and suppose $\lim_{m\to \infty} \| h_m - h\|_{\infty, [0,T]} =0$. We denote  $\phi=H^n[h]$ and $\phi_m=H^n[h_m]$. Then
\begin{align*}
\| \phi(t) - \phi_m(t)\|^2 =& 2\int_0^t \Big\langle \phi(s) - \phi_m(s), b(s, \phi(s), \delta_{\psi(s)} ) - b(s, \phi_m(s), \delta_{\psi(s)} ) \Big\rangle ds
\\
&+ 2\int_0^t \Big\langle  \phi(s) - \phi_k(s), f( \phi(s) - \psi(s) ) - f( \phi_m(s) - \psi(s) ) \Big\rangle ds
\\
& -2\int_0^t \Big\langle \phi(s) - \phi_m(s), dk_s^{h, n} - dk_s^{h_m, n} \Big\rangle
\\
&+ 2n\int_0^t \Big\langle \phi(s) - \phi_m(s), \sigma( \check{s}, \phi(\check{s}), \delta_{\psi(\check{s})} ) \Big( h( \hat{s}) - h( \check{s}) \Big) 
\\
&\qquad - \sigma(\check{s}, \phi_m(\check{s}), \delta_{\psi(\check{s})} ) \Big( h_m( \hat{s}) - h_m( \check{s}) \Big)\Big\rangle ds. 
\end{align*}
Hence
\begin{align*}
\Big\| \phi(t) - \phi_m(t) \Big\|^2 \leq& 4L\int_0^t \Big\| \phi(s) - \phi_m(s) \Big\|^2 ds
\\
+2n\int_0^t \Big\langle \phi(s)& - \phi_m(s), \Big( \sigma( \check{s}, \phi(\check{s}), \delta_{\psi(\check{s})} ) - \sigma( \check{s}, \phi_m(\check{s}), \delta_{\psi(\check{s})} ) \Big) \cdot \Big( h_m(\hat{s}) - h_m(\check{s}) \Big) ds
\\
+2n \int_0^t \Big\langle \phi(s)& - \phi_m(s), \sigma( \check{s}, \phi(\check{s}), \delta_{\psi( \check{s})} ) 
\cdot 
\Big( (h-h_m)( \hat{s}) - (h-h_m)( \check{s}) \Big\rangle ds. 
\end{align*}

Using the Lipschitz properties of $\sigma$ combined with $n$ being fixed, we get
\begin{align*}
\| \phi - \phi_m \|_{\infty, [0,T]}^2 
\leq&
\Big( 8L + 8n\| h\|_{\infty, [0,T]}\Big) \int_0^t \Big\| \phi(s) - \phi_m(s) \Big\|^2 ds 
\\
&+ 16n^2 \| h-h_m\|_{\infty, [0,T]}^2 \Big( \int_0^T \sigma( \check{s}, \phi(\check{s}), \delta_{\psi(\check{s})} ) ds \Big)^2. 
\end{align*}

As the integral $ \int_0^T \sigma(\check{s}, \phi(\check{s}), \delta_{\psi(\check{s})} ) ds$ will be finite for any choice of $n$ and $h$, we apply Gr\"onwall inequality to conclude
$$
\| \phi - \phi_m \|_{\infty, [0,T]}^2  \lesssim \| h-h_m\|_{\infty, [0,T]}^2. 
$$

\end{proof}

\begin{lemma}
\label{lemma the LDP for Y classical euler reflected sde}
Let $Y^{n,\varepsilon}$ be the solution to \eqref{equation Y classical reflected SDE Euler scheme}. 
Then $Y^{n,\varepsilon}$ satisfies an LDP on the space $C_{x_{0}}([0,T]; \bR^d)$, with a good rate function given by 
\begin{equation}
\label{equation Rate function LDP.}
 I^{n,T}_{x_{0}}( \phi ) 
 \coloneqq \underset{\{h\in \cH_1^0 ~:~H^n(h) = \phi\}}{\inf} I'(h).
\end{equation} 
\end{lemma} 
\begin{proof}
The result is a straightforward application of the contraction principle \cite{DZ}*{Theorem 4.2.1} using the continuous map $H^n$ as established in Lemma \ref{lemma : continuity of the map Hn}). 
\end{proof}
Next we use that $Y^{n,\varepsilon}$ is an approximation of $Y^\varepsilon$ in the appropriate sense to obtain an LDP for $Y^\varepsilon$ via \cite{DZ}*{Theorem 4.2.23}. 

\begin{lemma}
\label{lemma : euler scheme is an expo good approximation}
Let $Y^\varepsilon$ be the solution to \eqref{equation Y classical reflected SDE}, and $Y^{n,\varepsilon}$ be the solution to \eqref{equation Y classical reflected SDE Euler scheme}. Then for every $\delta>0$
\begin{align}
\limsup_{n\to\infty} \limsup_{\epsilon \to 0} \epsilon \log \bP \Big[ \sup_{t\in[0,T]} \| Y^{n,\varepsilon}_t-Y^{\varepsilon}_t\| > \delta  \Big]   = -\infty. \label{eq : expo goo approx definition}
\end{align}
That is $Y^{n,\varepsilon}$ is an exponentially good approximation of $Y^\varepsilon$, in the sense of \cite{DZ}*{Definition 4.2.14}.
\end{lemma}

\begin{proof}
The proof makes use of the LDP for $Y^{n,\varepsilon}$ established in Lemma \ref{lemma the LDP for Y classical euler reflected sde}. We follow a similar strategy as \cite{dos2019freidlin}*{Lemma 4.6}, requiring  an adapted version of \cite{DZ}*{Lemma 5.6.18} stated here in Lemma \ref{lemma : adapted D+Z lemma 5.6.18}.

Define the process $Z^\varepsilon \coloneqq Y^\varepsilon-Y^{n,\varepsilon}$, so that 

\begin{equation*}
Z^\varepsilon_t= \int_0^t b_s ds+\int_0^t \sigma_s ds+k^{Y^{n}}_{t}-k^{Y}_{t},
\end{equation*}
where
\begin{align*}
b_t \coloneqq& b\Big(t,Y^\varepsilon_t,\delta_{\psi(t)}\Big)-b\Big(t,Y^{n,\varepsilon}_{t},\delta_{\psi(t)}\Big)+f\Big(Y^\varepsilon_t -\psi(t) \Big)-f\Big(Y^{n,\varepsilon}_{t}-\psi(t)\Big),   
\\
\sigma_t \coloneqq& \sigma\Big(t,Y^\varepsilon_t,\delta_{\psi(t)}\Big)-
\sigma\Big(\check{t},Y_{\check{t}}^{n,\varepsilon},\delta_{\psi(\check{t})}\Big).
\end{align*}
Next we define the stopping time
\begin{equation*}
\tau_{R+1} \coloneqq \min \Big\{ T, \inf\{t\geq 0 : \|Y^{\varepsilon}_t \| \geq R+1\}, \inf\{t\geq 0:\|Y^{n,\varepsilon}_{t} \| \geq R+1 \} \Big\}.
\end{equation*}
Note that for $t\in[0,\tau_{R+1}]$ by the local Lipschitz property of $b$ and $f$, we have 
\begin{align*}
\|b_t\| \leq&  L_{R}\| Z^\varepsilon_t \|,
\end{align*}
for a constant $L_R$ only depending on $R$. Also note that 
\begin{align*}
\|\sigma_t\|\leq& \Big\| \sigma\Big( t, Y_t^\varepsilon ,\delta_{\psi(t)}\Big)
-
\sigma\Big(\check{t}, Y_t^\varepsilon  , \delta_{\psi(t)} \Big) \Big\| 
+ 
\Big\| \sigma\Big(\check{t}, Y_{\check{t}}^{n,\varepsilon}, \delta_{\psi(t)}\Big )
-
\sigma\Big(\check{t}, Y_t^\varepsilon, \delta_{\psi(t)} \Big) \Big\| 
\\
&+
\Big\| \sigma\Big(\check{t}, Y_{\check{t}}^{n,\varepsilon}, \delta_{\psi(\check{t})} \Big)
- 
\sigma\Big(\check{t}, Y_{\check{t}}^{n,\varepsilon}, \delta_{\psi(t)} \Big)\Big\|
\\
\leq&
L\Big( \| t-\check{t}\|^\beta + \| Z_t^\varepsilon  \| + \| \psi(t) - \psi(\check{t}) \| \Big) 
\\
\leq&
M(\rho(n)+\|Z_t\|),
\end{align*}
for some $M$ large enough, and $\rho(n)  \underset{n \to \infty }{\to} 0 $. Thus the conditions of Lemma $\ref{lemma : adapted D+Z lemma 5.6.18}$ are satisfied. Now fix any $\delta>0$ and notice that
\begin{align*}
\Big\{ \sup_{t\in[0,T]} \| Y^\varepsilon_t-Y^{n,\varepsilon}_{t}  \| \geq \delta \Big\} \subseteq& \Big\{ \sup_{t\in[0,\tau_{R+1}]} \| Y^\varepsilon_t-Y^{n,\varepsilon}_{t}  \| \geq \delta, \tau_{R+1} = T \Big\} 
\cup \Big\{ \sup_{t\in[0,T]} \| Y^\varepsilon_t-Y^{n,\varepsilon}_{t}  \| \geq \delta, \tau_{R+1}< T \Big\}
\\
\subseteq& \Big\{ \sup_{t\in[0,\tau_{R+1}]} \| Y^\varepsilon_t-Y^{n,\varepsilon}_{t}  \| \geq \delta \Big\}
\cup \Big\{ \tau_{R+1}< T \Big\}.
\end{align*}
By Lemma \ref{lemma : adapted D+Z lemma 5.6.18} we know that
\begin{equation*}
\lim_{n\to\infty} \limsup_{\varepsilon\to 0} \varepsilon \log \Big( \bP\Big[  \sup_{t\in[0,\tau_{R+1}]} \| Y^\varepsilon_t-Y^{n,\varepsilon}_{t}  \| \geq \delta \Big] \Big)=-\infty.
\end{equation*}
Furthermore define $\tau^{Y_{n}}_{R}=\inf\{t\geq0: \|Y^{n,\varepsilon}_{t}\| \geq R \} $, and notice
\begin{align*}
\Big\{ \tau_{R+1}< T \Big\} \subseteq& \Big\{  \tau_{R+1}< T, \tau^{Y^{n}}_{R}\leq T  \Big\} 
\cup \Big\{ \tau_{R+1}< T, \tau^{Y^{n}}_{R}> T  \Big\}
\\
\subseteq& \Big\{\tau^{Y^{n}}_{R}\leq T  \Big\} 
\cup \Big\{ \| Y_{\tau_{R+1}}^\varepsilon-Y_{\tau_{R+1}}^{n,\varepsilon}  \|\geq 1 \Big\}. 
\end{align*}
Again, by Lemma \ref{lemma : adapted D+Z lemma 5.6.18} and setting $ \delta=1$ we have that 
\begin{equation*}
\lim_{n\to\infty} \limsup_{\varepsilon\to 0} \varepsilon \log \Big( \bP\Big[ \sup_{t\in[0,\tau_{R+1}]} \| Y^\varepsilon_t-Y^{n,\varepsilon}_{t} \| \geq 1 \Big] \Big)=-\infty.
\end{equation*}
Recalling the identity, for positive $\alpha_\varepsilon,\beta_{\varepsilon}$
\begin{equation*}
\limsup_{\varepsilon\to 0}\varepsilon \log \Big( \alpha_{\varepsilon}+\beta_{\varepsilon} \Big)= \limsup_{\varepsilon \to 0}\varepsilon \log \Big( \max\Big\{ \alpha_{\varepsilon},\beta_{\varepsilon} \Big\} \Big), 
\end{equation*}
and appealing to the LDP satisfied by $Y^{n,\varepsilon}$, we are left with
\begin{align*}
\lim_{n\to\infty} \limsup_{\varepsilon \to 0} \varepsilon \log \Big( \bP\Big[ \sup_{t\in[0,T]} \|Y_t^\varepsilon-Y^{n,\varepsilon}_{t} \|\geq \delta \Big] \Big)
\leq& 
\lim_{n\to\infty} \limsup_{\varepsilon \to 0} \varepsilon \log \Big( \bP\Big[ \sup_{t\in[0,T]}\|Y^{n,\varepsilon}_{t} \|\geq R \Big] \Big)
\\
\leq& \lim_{n\to\infty} - 
    \underset{\phi \in C_{x_0}([0,T];\bR^d) : \sup_{t\in[0,T]}\|\phi(t)\|\geq R }{\inf} ~~ I^{n,T}_{x_{0}}(\phi).
\end{align*}
Hence to conclude \eqref{eq : expo goo approx definition} we show that
\begin{align} 
\lim_{R\to \infty} \lim_{n\to \infty} \underset{\phi \in C_{x_0}([0,T];\bR^d) : \sup_{t\in[0,T]}\|\phi(t)\|\geq R }{\inf} ~~ I^{n,T}_{x_{0}}(\phi) = \infty. \label{eq : what needs t be shown limit}
\end{align}

Indeed, let $\phi\in C_{x_{0}}([0,T]; \bR^d)$ be such that $\sup_{s\in[0,T]}\|\phi(s)\| \geq R$. Let $h \in \cH^{0}_1$ be a function such that $H^n[h] = \phi$, recall that if $h\notin \cH^0_1$ we immediately have that $I'(h)=\infty$. Via a concatenation argument it is simple to show that we can assume the path $\phi$ is increasing on $[0,T]$. Assuming $\phi$ is increasing we have $\forall s_1\leq s_2$ the bound
\begin{align}
\|\phi(s_1)-x_0\|\leq& 3\|\phi(s_2)-x_0\|+2\|x_0\|.\label{eq : increasing in inf.}
\end{align}   

Note that

 \begin{align*}
\|\phi(t) - x_{0}\|^2 
=&  2\int_0^t \Big\langle \phi(s)-x_0 , b(s,\phi(s),\delta_{\psi(s)})+f(\phi(s)-\delta_{\psi(s)})\Big\rangle ds
\\
&+  \int_{0}^{t} \Big \langle \phi(s)-x_0, \sigma(\check{s},\phi(\check{s}),\delta_{\psi(\check{s})})\frac{n}{T} \Big(h(\hat{s})-h(\check{s})\Big)\Big\rangle ds
\\
&- 2\int_0^t \Big\langle \phi(s)-x_{0},\mathbf{n}(\phi(s)) \Big\rangle |k^{h,n}|_s.
\end{align*}
By Cauchy–Schwarz and the one-sided Lipschitz properties of $b$ and $f$ we can bound the drift term by

\begin{align*}
\Big\langle& \phi(s)-x_0 , b(s,\phi(s),\delta_{\psi(s)})+f(h(s)-\delta_{\psi(s)})\Big\rangle
\\
&\leq 2(L+2)\|\phi(s)-x_0\|^2+2\|f(x_0-\delta_{\psi(s)}) \|^2+2\| b(s,x_0,\delta_{\psi(s)}) \|^2.
\end{align*}
Using this bound, the integrability conditions of $f$ and $b$, and Lemma \ref{lem:NormalToDomain} we have for a constant $c_1=c_1(L,x_0)$ independent of $t$
\begin{align}
\nonumber
\|\phi(t)& - x_{0}\|^2 
=  c_1\Big(1+\int_0^t \| \phi(s)-x_0 \|^2ds\Big)
\\
\label{e110}
&+ \int_{0}^{t} \Big \langle \phi(s)-x_0, \sigma (\check{s},\phi(\check{s}),\delta_{\check{s})})\frac{n}{T} \Big(h(\hat{s})-h(\check{s})\Big)\Big\rangle ds. 
\end{align}
We can further bound the above term by noting that for any vector $a\in \mathbb{R}^d$,

\begin{align*}
\Big\langle \phi(s)-x_0,\sigma(\check{s},\phi(\check{s}),  \delta_{\check{s})}) a \Big\rangle 
\leq&
L \| \phi(s)-x_0 \| \| \phi(\check{s})-x_0 \| \|a\|
\\
&+ 
\| \phi(s)-x_0 \| \| \sigma(\check{s},x_0,\delta_{\psi(\check{s})}) \| \|a\|.
\end{align*}

Since $\check{s}\leq s $ employing \eqref{eq : increasing in inf.}, and $c<c^2+1$ for $c\in \bR$,we have for a constant $c_2=c_2(L,x_0)$ independent of $t$, $n$
\begin{align*}
\Big\langle \phi(s)-x_0,\sigma(\check{s},\phi(\check{s}),  \delta_{\psi(\check{s})}) a \Big\rangle 
\leq& c_2 \Bigg( \| \phi(s)-x_0 \|^2 \Big( \|a\| +  \|\sigma(\check{s},x_0,\delta_{\psi(\check{s})})\|\|a\| \Big) 
\\
&+ \|a\|+ \| \sigma(\check{s},x_0,\delta_{\psi(\check{s})}) \| \|a\| \Bigg).
\end{align*}
Setting 
$$
a= \frac{n}{T} \Big(h(\hat{s})-h(\check{s})\Big)=\frac{n}{T}\int_{\check{s}}^{\hat{s}}\dot{h}(u)du,
$$
and substituting this bound into \eqref{e110}, we get that for a constant $c=c(L,x_0)$ independent of $t$ or $n$
\begin{align}
\|\phi(t) - x_{0}\|^2 
\leq & c\Bigg(\int_0^t \Big\| \frac{n}{T}\int_{\check{s}}^{\hat{s}}\dot{h}(u)du \Big\| + \Big\| \sigma(\check{s},x_0,\delta_{\psi(\check{s})}) \Big\|  \Big\| \frac{n}{T} \int_{\check{s}}^{\hat{s}}\dot{h}(u)du \Big\| ds \label{e120}
\\
&+
\int_0^t \|\phi(s)-x_0\|^2 \Big( 1+ \Big\| \frac{n}{T} \int_{\check{s}}^{\hat{s}}\dot{h}(u)du \Big\| + \Big\| \sigma(\check{s},x_0,\delta_{\psi(\check{s})}) \Big\| \Big\| \frac{n}{T} \int_{\check{s}}^{\hat{s}}\dot{h}(u)du \Big\| \Big)ds \Bigg). \nonumber
\end{align}
Also note that we have
$$
\frac{n}{T}\int_0^t\int_{\check{s}}^{\hat{s}}\|\dot{h}(u)\| du ds \leq \int_0^{T}\| \dot{h}(s)\| ds, 
$$
and similarly
\begin{align*}
\frac{n}{T}\int_0^t\| \sigma(\check{s},x_0,\delta_{\psi(\check{s})}) \| \int_{\check{s}}^{\hat{s}}\|\dot{h}(u)\| du ds=& \frac{n}{T}\int_0^t \int_{\check{s}}^{\hat{s}}\| \sigma(\check{s},x_0,\delta_{\psi(\check{s})}) \|\|\dot{h}(u)\| du ds  \\
\leq& \int_0^T \| \sigma(\check{s},x_0,\delta_{\psi(\check{s})}) \|\|\dot{h}(s)\| ds.
\end{align*}

By applying to Gr\"onwall's Inequality in \eqref{e120}, and using the previous two observations, we have

\begin{align*}
\|\phi(t) - x_{0}\|^2  \leq c&\Bigg( \int_0^{T}\| \dot{h}(s)\|  + \| \sigma(\check{s},x_0,\delta_{\psi(\check{s})}) \|\|\dot{h}(s)\| ds 
\\
& \cdot \exp\Big(c \int_0^{T}1+\| \dot{h}(s)\|  + \| \sigma(\check{s},x_0,\delta_{\psi(\check{s})}) \|\|\dot{h}(s)\| ds \Big) \Bigg).
\end{align*}

Now adding and subtracting the terms $\| \sigma(s,x_0,\delta_{\psi(\check{s})} ) \|,\| \sigma(\check{s},x_0,\delta_{\psi(s)} ) \|$, using the Triangle Inequality, Cauchy-Schwarz's inequality, the continuity of $\psi$, and recalling the Assumption \ref{assumption : Holder regularity of sigma} we obtain \eqref{eq : what needs t be shown limit}.
\end{proof}
\begin{lemma}
\label{lemma the LDP for Y classical reflected sde}
Let $Y^\varepsilon$ be the solution to \eqref{equation Y classical reflected SDE}. Then $Y^\varepsilon$ satisfies an LDP on the space $C_{x_0}([0,T];\bR^d)$ with the good rate function 
\begin{equation}\label{eq : the rate function for our self-stab. reflec. sde}
I^T_{x_{0}}(\phi)=\inf_{\{h\in \cH_1^0 ~:~ H[h]=\phi\}} I'(h),
\end{equation}
where the skeleton operator $H$ was defined in \eqref{eq:SkeletonProcess-h}.
\end{lemma}
\begin{proof} The proof will follow by appealing to \cite{DZ}*{Theorem 4.2.23}. That is we need to show that for every $\alpha>0$  
\begin{equation}\label{e1}
\lim_{n\to\infty}\sup_{\{h\in \cH^0_1  ~:~ \|h\|_{\cH^0_1}<\alpha\}}\|H^n[h]-H[h]\|=0.
\end{equation}
Fix $\alpha<\infty$, $h\in \cH^0_1 $ with $\|h\|_{\cH^0_1}<\alpha$. Denote $\phi^n=H^n(h)$, $\phi=H(h)$. Now by the one-sided Lipschitz property of the drift and Lemma \ref{lem:NormalToDomain}, 
\begin{align}
\nonumber
\|\phi^n(t)-\phi(t)\|^2 \leq&  2
\int_0^t  \Big\langle \phi^n(s)-\phi(s), \sigma(\check{s},\phi^n(\check{s}),\delta_{\psi(\check{s})}h_n(s)
\\
\label{e111}
&-\sigma\Big(s,\phi(s),\delta_{\psi(s)}\Big)\dot{h}(s) \Big\rangle ds  + \int_0^t 4L\|\phi^n(s)-\phi(s)\|^2ds,
\end{align}
where we have denoted  $h_n(s)\coloneqq \frac{n}{T} \Big( h( \hat{s} ) - h( \check{s})\Big) $. Next notice that 
\begin{align*}
\Big\|  \sigma(\check{s},\phi^n(\check{s}),\delta_{\psi(\check{s})})-\sigma(s,\phi(s),\delta_{\psi(s)}) \Big\|
\leq& 
\Big\| \sigma(\check{s},\phi^n(\check{s}),\delta_{\psi(\check{s})})-\sigma(s,\phi^n(\check{s}),\delta_{\psi(\check{s})}) \Big\|
\\
&+ \Big\| \sigma(s,\phi^n(\check{s}),\delta_{\psi(\check{s})})-\sigma(s,\phi^n(\check{s}),\delta_{\psi(s)}) \Big\|
\\
&+\Big\| \sigma(s,\phi^n(\check{s}),\delta_{\psi(s)})-\sigma(s,\phi(s),\delta_{\psi(s)})  \Big\| 
\\
\leq& \rho^n(s)+L\|\phi^n(s)-\phi(s)\|,
\end{align*}
where $\sup_{s\in[0,T]}\rho^n(s)\underset{n\to\infty}{\to} 0$, by continuity of $\psi$ and the Assumption \ref{assumption : Holder regularity of sigma}. Hence
\begin{align*}
\Big\|  \sigma&(\check{s},\phi^n(\check{s}),\delta_{\psi(\check{s})})h_n(s)-\sigma\Big(s,\phi(s),\delta_{\psi(s)}\Big)\dot{h}(s) \Big\|
\\
\leq& (\rho^n(s)+L\|\phi^n(s)-\phi(s)\|)\| h_n(s) \|+\| \sigma(s,\phi(s),\delta_{\psi(s)}) \| \| \dot{h}(s)-h_n(s) \|.
\end{align*}
Substituting this bound into \eqref{e111} and applying Gr\"onwall 
we get that for a constant $c$ independent of $n$ or $t$,
\begin{align*}
&\|\phi^n(t)-\phi(t)\|^2\leq c\exp\Bigg( c\int_0^t 1 + (\rho^n(s)+1)\|h_n(s)\| + \|\sigma(s,\phi(s),\delta_{\psi(s)} ) \| \cdot \| \dot{h}(s)-h_n(s) \|ds \Bigg) 
\\
&\qquad \cdot \int_0^t (\rho^n(s)+1)\|h_n(s)\| + \|\sigma(s,\phi(s),\delta_{\psi(s)} ) \| \cdot \|\dot{h}(s)-h_n(s)\|ds 
\\
&\leq c\exp\Bigg( c\int_0^t 1+ (\rho^n(s)+1) \cdot (\|\dot{h}(s)\|+\|h_n(s)-\dot{h}(s)\|) + \| \sigma(s,\phi(s),\delta_{\psi(s)})\| \cdot \|\dot{h}(s)-h_n(s)\|ds \Bigg) 
\\
&\qquad \cdot \int_0^t (\rho^n(s)+1)\|\dot{h}(s)\|
+(\rho^n(s)+1)\|h_n(s)-\dot{h}(s)\|
+\| \sigma(s,\phi(s),\delta_{\psi(s)})\|\cdot \|\dot{h}(s)-h_n(s)\|ds.
\end{align*}
Applying Cauchy–Schwarz on the $\|\sigma(s,\phi(s),\delta_{\psi(s)})\| \cdot \|\dot{h}(s)-h_n(s) \|$ terms and sending $n\to \infty$ gives \eqref{e1}. The LDP for $Y^{\epsilon}$ with rate function \eqref{eq : the rate function for our self-stab. reflec. sde} now follows by appealing to \cite{DZ}*{Theorem 4.2.23} and the fact that $Y^{n,\varepsilon}$ are exponentially good approximations of  $Y^\varepsilon$ Lemma \ref{lemma : euler scheme is an expo good approximation}.
\end{proof}

\subsection{Freidlin-Wentzell results for reflected McKean-Vlasov equations}
Next we pass the LDP from the process $Y^\varepsilon$ to $X^\varepsilon$ using exponential equivalence.
\begin{theorem}
\label{ldp for xi}
Let $x_0^\varepsilon \in \mathbb{R}^d$, converge to $x_0\in \mathbb{R}^d$ as $\varepsilon \to 0$. Let $Y^\varepsilon$ be the solution to \eqref{equation Y classical reflected SDE}, $\psi^{x_0}$ the solution of \eqref{eq:SkeletonProcess-0}, and $X^{\varepsilon}$ be the solution to Equation \eqref{eq:MVSS-LDP} started at $X^\varepsilon_0= x^\varepsilon_0$. Then the reflected McKean-Vlasov equation $X^{\varepsilon}$ satisfies an LDP on $C_{x_{0}}([0,T]; \bR^d)$ with rate function  \eqref{eq : the rate function for our self-stab. reflec. sde}. 
\end{theorem}
\begin{proof} 
Firstly, one can quickly verify that $\|\psi^{x^\varepsilon_0} (t)-\psi^{x_0}(t) \|\underset{\varepsilon \to 0}{\to} 0$. Let $Z_t^\varepsilon\coloneqq X_t^\varepsilon-Y_t^\varepsilon$. Then $Z^\varepsilon$ satisfies
\begin{equation*}
Z^\varepsilon_t=z_0+ \int_0^t b_s ds+\int_0^t \sigma_s ds+k^{Y,\varepsilon}_{t}-k^{\varepsilon}_{t},
\end{equation*}
where $ z_0 \coloneqq  x_0^\epsilon-x_0$, $\sigma_t \coloneqq \sigma\big(t,X^\varepsilon_t,\mu_t^\varepsilon \big)-\sigma\big(t,Y^\varepsilon_t,\delta_{\psi^{x_{0}}(t)}\big)$ and 
\begin{align*}
b_t \coloneqq
& 
 b\Big(t,X^\varepsilon_t,\mu_t^\varepsilon \Big)
-b\Big(t,Y^\varepsilon_t,\delta_{\psi^{x_{0}}(t)}\Big)
+\int_{\mathbb{R}^d} f(X^\varepsilon_t-x) d\mu_t^\varepsilon -f(Y^\varepsilon_t-\psi^{x_0}(t))
.
\end{align*}
Let $R>0$ be large enough so that $x_0^\varepsilon,y\in B_{R+1}(0)$, and $\psi^{x_{0}}(t)$ does not leave $B_{R+1}(0)$ up to time $T$. We are able to do since $\psi$ is non-explosive. Let $\tau_{R+1}\coloneqq \min \Big\{T , \inf\{ t\geq0 : \|X_t^\varepsilon \| \geq R+1 \},\inf\{  t\geq0 : \|Y_t^\varepsilon \| \geq R+1\}  \Big\}$. Notice that for all $t\in [0,\tau_{R+1}]$ we have
\begin{align*}
\Big\| b\Big(t,X^\varepsilon_t, \mu_t^\varepsilon \Big) &- b\Big(t, Y^\varepsilon_t, \delta_{\psi^{x_{0}}(t)}\Big) \Big\|
\\
&\leq \Big\| b\Big(t,X^\varepsilon_t,\mu_t^\varepsilon \Big) - b\Big(t,X^\varepsilon_t,\delta_{\psi^{x^{\varepsilon}_{0}}(t)}\Big) \Big\| 
+ \Big\| b\Big(t,X^\varepsilon_t,\delta_{\psi^{x^{\varepsilon}_{0}}(t)}\Big) - b\Big(t,X^\varepsilon_t,\delta_{\psi^{x_{0}}(t)}\Big) \Big\| 
\\
&\quad + \Big\| b\Big(t,X^\varepsilon_t,\delta_{\psi^{x_{0}}(t)} \Big)-b\Big(t,Y^\varepsilon_t,\delta_{\psi^{x_{0}}(t)}\Big) \Big\| 
\\
&\leq
L \bE\Big[ \|X^\varepsilon_t - \psi^{x^{\varepsilon}_{0}}(t)\|^2\Big]^{\frac{1}{2}}
+ L\| \psi^{x_0^\varepsilon}(t)-\psi^{x_0}(t) \|
+ L_R \|X^\varepsilon_t-Y^\varepsilon_t \|.
\end{align*}
Hence 
\begin{equation*}
    \Big\| b\Big(t,X^\varepsilon_t, \mu_t^\varepsilon \Big) - b\Big(t, Y^\varepsilon_t, \delta_{\psi^{x_{0}}(t)}\Big) \Big\| \leq B^1_{R} \big(\rho^1(\varepsilon)+\|Z_t^\varepsilon\|^2 \big)^{\frac{1}{2}},
\end{equation*}
for a constant $B^1_R$ large enough, and $\rho^1(\varepsilon)\coloneqq \EE\|X_t^\varepsilon-\psi^{x^{\varepsilon}_0}(t)\|^2+\| \psi^{x_0^\varepsilon}(t)-\psi^{x_0}(t) \| \underset{\varepsilon \to 0}{\to} 0$ by \eqref{equation X goes to skeleton}. Furthermore for $t\in[0,\tau_{R+1}]$ we also have 
\begin{align*}
\Big\| \int_{\mathbb{R}^d} &f(X^\varepsilon_t-x)d\mu_t^\varepsilon-f(Y^\varepsilon_t-\psi^{x_0}(t)) \Big\| 
\\
\leq& \Big\| \int_{\mathbb{R}^d} f(X^\varepsilon_t-x)-f(X^\varepsilon_t-\psi^{x^{\varepsilon}_0}(t)) \Big\|
+ \Big\| f(X^\varepsilon_t-\psi^{x^{\varepsilon}_0}(t))-f(X^\varepsilon_t-\psi^{x_0}(t)) \Big\|
\\
&+ \Big\| f(X^\varepsilon_t-\psi^{x_0}(t))-f(Y^\varepsilon_t-\psi^{x_0}(t)) \Big\| 
\\
\leq& \Big\| \int_{\mathbb{R}^d} f(X_t^\varepsilon-x)d\mu^\varepsilon_t -f(X-\psi^{x_0^\varepsilon}(t)) \Big\| + L_R \Big\| \psi^{x^{\varepsilon}_0}(t)-\psi^{x_0}(t) \Big\| + L_R \|Z_t\|.
\end{align*}
Hence 
\begin{equation*}
    \|b_t\|\leq B^2_R\Big( \rho^2(\varepsilon)+\|Z_t\|^2 \Big)^{\frac{1}{2}},
\end{equation*}
for a constant $B^2_R$ and $\rho^2(\varepsilon)\coloneqq \| \int_{\mathbb{R}^d} f(X_t^\varepsilon-x)d\mu^\varepsilon_t -f(X-\psi^{x_0^\varepsilon}(t)) \| +\|\psi^{x^{\varepsilon}_0}(t)-\psi^{x_0}(t) \|\underset{\varepsilon \to 0}{\to} 0$, thanks to \eqref{equation convergence of law to the Diract path}. Now for the diffusion term,
\begin{align*}
\|\sigma_t \|\leq& \Big\| \sigma\Big(t,X^\varepsilon_t,\mu_t^\varepsilon \Big)-\sigma\Big(t,X^\varepsilon_t,\delta_{\psi^{x^\varepsilon_{0}}(t)}\Big)  \Big\|  
+ \Big\| \sigma\Big(t,X^\varepsilon_t,\delta_{\psi^{x^\varepsilon_{0}}(t)}\Big)-\sigma\Big(t,X^\varepsilon_t,\delta_{\psi^{x_{0}}(t)}\Big) \Big\|  
\\
&+\Big\| \sigma\Big(t,X^\varepsilon_t,\delta_{\psi^{x_{0}}(t)}\Big)-\sigma\Big(t,Y^\varepsilon_t,\delta_{\psi^{x_{0}}(t)}\Big)  \Big\|
\\
\leq& L \Big( \bE\Big[ \|X^\varepsilon_t-\psi^{x^{\varepsilon}_{0}}(t)\|^2\Big]^{\frac{1}{2}}
+ \| \psi^{x_0^\varepsilon}(t)-\psi^{x_0}(t) \|
+ \|X^\varepsilon_t-Y^\varepsilon_t \| \Big). 
\end{align*}
Hence
\begin{equation}
\|\sigma_t \|\leq M\big( \rho(\varepsilon)+\|Z^\varepsilon_t\|^2 \big)^{\frac{1}{2}} ,
\end{equation}
for a constant $M$ and $\rho(\varepsilon) \underset{\varepsilon \to 0}{\to}0$. 

Now fix $\delta>0$ and notice that 
\begin{align*}
\Big\{ \sup_{t\in[0,T]} \| X^\varepsilon_t-Y^\varepsilon_t  \| \geq \delta \Big\} \subseteq& \Big\{ \sup_{t\in[0,\tau_{R+1}]} \| X^\varepsilon_t-Y^\varepsilon_t  \| \geq \delta, \tau_{R+1} = T \Big\} 
\cup \Big\{ \sup_{t\in[0,T]} \| X^\varepsilon_t-Y^\varepsilon_t  \| \geq \delta, \tau_{R+1}< T \Big\}
\\
\subseteq& \Big\{ \sup_{t\in[0,\tau_{R+1}]} \| X^\varepsilon_t-Y^\varepsilon_t  \| \geq \delta \Big\}
\cup \Big\{ \tau_{R+1}< T \Big\}. 
\end{align*}
By Lemma \ref{lemma : adapted D+Z lemma 5.6.18} we know that 
\begin{equation*}
\limsup_{\varepsilon\to 0}\varepsilon \log \Big( \bP\Big[ \sup_{t\in[0,\tau_{R+1}]} \| X^\varepsilon_t - Y^\varepsilon_t  \| \geq \delta \Big] \Big) =-\infty.
\end{equation*}
Furthermore, define $\tau^Y_{R} \coloneqq \inf\{t\geq0: \|Y^\varepsilon_t\| \geq R \} $, and notice that
\begin{align*}
\Big\{ \tau_{R+1}< T \Big\} \subseteq& \Big\{  \tau_{R+1}< T, \tau^Y_{R}\leq T  \Big\} 
 \cup \Big\{ \tau_{R+1}< T, \tau^Y_{R}> T  \Big\}
 \\
 \subseteq& \Big\{\tau_{R+1}< T  \Big\} 
 \cup \Big\{ \| X_{\tau^Y_{R}}^\varepsilon-Y_{\tau_{R+1}}^\varepsilon  \|\geq 1 \Big\}. 
\end{align*}
Again, setting $ \delta=1$ and using Lemma \ref{lemma : adapted D+Z lemma 5.6.18}, we have that 
\begin{equation*}
\limsup_{\varepsilon\to 0}\varepsilon \log \Big( \bP\Big[ \sup_{t\in[0,\tau_{R+1}]} \| X^\varepsilon_t - Y^\varepsilon_t \| \geq 1 \Big] \Big) = -\infty,
\end{equation*}
hence are left with 
\begin{align*}
\limsup_{\varepsilon \to 0}\varepsilon \log \Big( \bP \Big[ \sup_{t\in[0,T]} \|X_t^\varepsilon - Y^\varepsilon_t \|\geq \delta \Big] \Big) \leq& \limsup_{\varepsilon \to 0} \varepsilon \log \Big( \bP \Big[ \sup_{t\in[0,T]} \|Y^\varepsilon_t \|\geq R \Big] \Big).
\end{align*}
Applying the LDP proved for $Y^\varepsilon$ in Lemma \ref{lemma the LDP for Y classical reflected sde} we conclude, 
\begin{align*}
\limsup_{\varepsilon \to 0} \varepsilon \log \Big( \bP \Big[ &\sup_{t\in[0,T]} \|X_t^\varepsilon - Y^\varepsilon_t \|\geq \delta  \Big] \Big)
\\
&
\leq
  - 
\underset{\{\phi\in C_{x_0}([0,T];\bR^d, ~:~\sup_{t\in[0,T]}\|\phi(t)\|\geq R \}}{\inf} ~~ I^{T}_{x_{0}}(\phi) \underset{R\to \infty}{\longrightarrow}-\infty,
\end{align*}
by the same arguments as the end of the proof of Lemma \ref{lemma : euler scheme is an expo good approximation}. 
\end{proof}
An immediate consequence (choosing $x_0^\varepsilon=x_0$) we have an LDP for our reflected McKean-Vlasov equation's solution $X^\varepsilon$ of \eqref{eq:MVSS-LDP} with $X^{\varepsilon}_0=x_0$. The point of allowing $\varepsilon$-dependent initial conditions for  $X^\varepsilon$ enables us to claim the LDP uniformly on compacts, similarly to \cite{HIP}*{Corollary 3.5}, or \cite{Herrmann2013StochasticR}*{Propositions 4.6 and 4.8}. We provide a statement and a brief proof, the full justification is identical to those found in \cites{HIP,Herrmann2013StochasticR}.

\begin{corollary}
Let $\mathbb{P}_{x_{0}}[X^\varepsilon\in\cdot]$ be the law on $C_{x_{0}}([0,T]; \bR^d)$ of the solution $X^\varepsilon$ to \eqref{eq:MVSS-LDP} with $X_0^\varepsilon=x_0$. Let $M\subset \mathbb{R}^d$ be a compact subset. Then, for any Borel set $A\subset C([0,T]; \bR^d)$, we have 
\begin{align}
\label{equation uniform ldp upper}
\liminf_{\varepsilon \to 0}\varepsilon \log \sup_{x_0\in M}\PP_{x_0}[ X^{\varepsilon}\in A] \leq &-\inf_{x_0 \in M}\inf_{\phi\in \overline{A} } I_{x_0}^T(\phi),
\end{align}
and
\begin{align}
\label{equation uniform ldp lower}
\liminf_{\varepsilon \to 0}\varepsilon \log \inf_{x_0\in M} \PP_{x_0}[ X^{\varepsilon}\in A] \geq&-\sup_{x_0 \in M}\inf_{\phi \in A^{\circ} } I_{x_0}^T(\phi).
\end{align}
\end{corollary}
\begin{proof}
Allowing $\varepsilon$-dependent initial conditions, implies that (otherwise we would contradict the LDP)
\begin{align*}
\limsup_{\underset{ x_{\varepsilon}\to x_0}{\varepsilon \to 0}}\varepsilon\log \bP_{x_{\varepsilon}}[ X^{\varepsilon}\in A ] \leq & - \inf_{\phi\in \overline{A}}I^T_{x_{0}}(\phi),
\end{align*}
then arguing as in \cite{DZ}*{Corollary 5.6.15} yields \eqref{equation uniform ldp upper}. The lower bound \eqref{equation uniform ldp lower} is done similarly. 
\end{proof}

Furthermore, proceeding like in \cite{HIP} we could obtain uniform on compacts LDP for the process $X^\varepsilon$ started at some later time $s>0$, and initial condition $x_s^\varepsilon$. Such uniform LDP can be useful when obtaining exit-time results in the manner of \cite{HIP}. However we will not need them, and instead obtain exit-time results by the method of \cite{tugaut2016simple}.

\section{Exit-time}
\label{sec:ExitTimes}

In this section we obtain a characterisation of the exit-time of $X^\varepsilon$ from an open subdomain $\fD\subset \cD$ under several  additional assumptions: strict convexity of potentials, the diffusion matrix is the identity matrix and time-homogeneity of the coefficients. These are motivated by applications (like \cites{di2017jump,di2019sharp}) where the exit-cost of the diffusion from a domain needs to be computed explicitly, here we refer to $\Delta$ in Theorem \ref{thm:ExitTime}. The results obtained in this section are, from a methodological point of view, inspired by \cite{tugaut2016simple}. 

Let us start by introducing the process of interest $(X_t^\varepsilon)_{t\geq 0}$ over $\bR^d$ with dynamics
\begin{align}
\label{eq:ExitTime-Process}
X_t^\varepsilon =& x_0 + \int_0^t b( X_s^\varepsilon) ds + \int_0^t f\ast \mu^\varepsilon_s(X_s^\varepsilon) ds + \sqrt{\varepsilon} W_t - k_t^\varepsilon, \quad \bP\big[ X_t^\varepsilon \in dx \big] = \mu_t^\varepsilon (dx),
\\
\nonumber
|k^\varepsilon|_t =& \int_0^t \1_{\partial \cD}(X_s^\varepsilon) d|k^\varepsilon|_s, 
\qquad 
k^\varepsilon_t = \int_0^t \1_{\partial \cD}(X_s^\varepsilon)\n(X_s^\varepsilon) d|k^\varepsilon|_s. 
\end{align}
\begin{assumption}
\label{ass:ExitTime-Coefficients}
Let $\cD$ satisfy Assumption \ref{assumption:domain}. Let $r>1$ and let $b: \cD \to \bR^d$, $f:\bR^d \to \bR^d$ satisfy
\begin{itemize}
\item There exist functions $B: \cD \to \bR$ and $F:\bR^d \to \bR$ such that
$$
b(x) = \nabla B(x), 
\quad
f(x) = \nabla F(x),
$$
\item $B$ is uniformly strictly concave, $\exists L>0$ such that $\forall x,y\in \cD$, 
$$
\big\langle x-y, b( x) - b( y) \big\rangle \leq -L \| x - y\|^2, 
$$
\item $\exists G:\bR \to \bR$ a convex even polynomial such that $F(x) = G(\| x\|)$ of order $r$ where
$$
G(\| x\|) < C ( 1+ \| x\|^r),
$$
and $\forall x, y\in \bR^d$ we have $\big\langle x - y, f(x) - f(y) \big\rangle \leq 0$,
\item $\exists \tilde{x} \in \cD^\circ $ such that 
$\inf_{x\in \cD} \| b( x)\| = \| b( \tilde{x})\| = 0$.
\end{itemize}
\end{assumption}
We study the metastability of the system around $\tilde{x}$ within the domain $\fD$. Intuitively, the dynamics of the process are similar to those of the non-reflected case, so that in the small noise limit the process spends most of its time around the stable point $\tilde{x}$ and with a high probability excursions from the stable point promptly return to it. Therefore, the only way to leave the domain $\fD$ is to receive a large shock from the driving noise, which is expected to take a long time to happen. 
\begin{defn}
\label{defnPstasta}
Let $\cG$ be a subset of $\cD$ and let $U:\cD \to \bR^d$. For all $x\in \cD$, let $\varphi$ be the dynamical system 

$$\bR^{+} \ni t\mapsto \varphi_t(x)= x + \int_0^t U ( \varphi_s(x)) ds.$$ 

We say that the domain $\cG$ is \emph{stable by} $U$ if $\forall x\in \cG$, 
$$
\Big\{ \varphi_t(x): \ t\in\bR^{+} \Big\} \subset \cG. 
$$
\end{defn}
This is also referred to as ``positively invariant'' in other works. We now introduce supplementary assumptions on the domain $\fD$ in order to obtain the exit-time. The first one is slightly different from the one in \cite{HIP} as we do not assume that $\fD$ is stable by $b$ but instead we work with the following.

\begin{assumption}
\label{Ass:DomainDofExitTime}
Let $\fD \subset \cD$ be an open, connected set containing $\tilde{x}$ such that $ \overline{\fD} \subset \cD$ and $\partial \cD\cap \fD = \emptyset$.

Let $x_0\in \fD$. Let $\psi_t=x_0 + \int_0^t b( \psi_s )ds$. The orbit
$$
\Big\{ \psi_{t}: t\in \bR^{+} \Big\} \subset \fD. 
$$

Further domain $\fD$ is stable by $b( \cdot) +f(\cdot-\tilde{x})$.
\end{assumption} 

Roughly speaking, when the time is small, the reflected self-stabilizing diffusion behaves like the dynamical system $\{\psi_t\}_{t\in[0,T]}$. As a consequence, and in order to have a non-trivial exit-time, we assume that the orbit of the dynamical system without noise stays in the domain $\fD$.

After a long time, the reflected self-stabilizing diffusion stays close to a linear reflected diffusion with potential $B( \cdot) + F\ast\delta_{\tilde{x}}$. It is then natural to assume that the domain is stable by $b(\cdot) +f(\cdot-\tilde{x})$. 

\begin{defn}
\label{defn:balance}
Let $x\in \cD$. Let $r>1$ and let $\kappa>0$. Let
$\bB_x^{\kappa, r} \subset \cP_r(\cD)$ denote the set of all the probability measures such that
$$
\int_{\cD} \| y - x \|^{r}\mu(dy) \leq \kappa^{r}. 
$$
\end{defn}
We study the distribution of the following stopping time.
\begin{defn}
\label{dfn:ExitTime}
Let $\fD \subset \bR^d$, $x_0, \tilde{x}\in \bR^d$ satisfy Assumption \ref{Ass:DomainDofExitTime}. Let $\varepsilon>0$ and let $X^{\varepsilon}$ be the solution to \eqref{eq:ExitTime-Process}. 

Define the exit-time $\tau_\fD(\varepsilon)$ of $X^{\varepsilon}$ from the domain $\fD$ as 
$$
\tau_\fD(\varepsilon):= \inf\Big\{ t\geq 0:  X^{\varepsilon}_t\notin \fD \Big\}.
$$
\end{defn}

Within classical SDE theory, there is no difference between the reflected and the non-reflected process since the exit domain $\fD$ is necessarily contained in the domain of constraint $\cD$. This is not the case for McKean-Vlasov equations where the reflective term acts on the law to ensure it remains on the domain $\cD$ and is thus different from the law of the non-reflected McKean-Vlasov. In the language of particle systems, see \eqref{eq:ParticleSystem}, each particle  $i$ is additionally affected by the reflections of all other particles $j \neq i$.

One of our contributions here is to rigorously argue that although the law of the reflected process and the law of the non-reflected process are different, the difference \textit{does not} affect the distribution of the exit-time $\tau_\fD(\varepsilon)$. Further, we remark that the results of Sections \ref{subsec;ExitTime_CoM}, \ref{subsec;ExitTime_PoEbC} and \ref{subsec;ExitTime_CR} typically hold under much broader conditions than those of Assumption \ref{ass:ExitTime-Coefficients}. This not the case for the proof of Theorem \ref{thm:ExitTime} which relies on classical methods and so determines the scope of our results. 

\subsection{Control of the moments}
\label{subsec;ExitTime_CoM}
In this section, we study the distance between the law of the process at time $t$ and the Dirac measure at $\tilde{x}$. 
\begin{defn}
\label{dfn:ExitTime-Moments}
Let $\cD$ satisfy Assumption \ref{assumption:domain}. Let $W$ be a $d$-dimensional Brownian motion and let $r>1$, $b$, $f$, $x_0$ and $\tilde{x}$ satisfy Assumption \ref{ass:ExitTime-Coefficients}. Let $X^\varepsilon$ be the solution to Equation \eqref{eq:ExitTime-Process}. Define $\xi_\varepsilon^r: \bR^+ \to \bR^+$ to be
$$
\xi_\varepsilon^r(t):= \bE\Big[ \| X_t^\varepsilon - \tilde{x} \|^{r} \Big]. 
$$ 

For $\kappa>0$, define
\begin{equation*}
T^{\kappa,r} (\varepsilon):= \min \Big\{t\geq 0: \xi_\varepsilon^r(t)\leq\kappa^{r} \Big\}.
\end{equation*}
\end{defn}
\begin{prop} 
\label{prop:ExitTime-moments:1} 
We have
$$
\sup_{t\in \bR^{+} } \xi_\varepsilon^r (t)\leq\max \Big\{ \|x_0-\tilde{x} \|^{r}, \Big( \tfrac{d\varepsilon(r-1)}{2L}\Big)^{r/2} \Big\}.
$$
For $\varepsilon < \tfrac{\kappa^2 L}{d(r-1)}$, we have 
$$
T^{\kappa,r} (\varepsilon) \leq \tfrac{1}{rL} \log\Big( \tfrac{2\| x_0 - \tilde{x}\| }{\kappa^2} - 1\Big).  
$$

Finally, for all $t\geq T^{\kappa,r}(\varepsilon)$ with $\varepsilon < \tfrac{\kappa^2 L}{2r-1}$ we have $\xi_\varepsilon(t)\leq\kappa^{2r}$.
\end{prop}

\begin{proof} 
Let $t\in\bR^+$. We apply the It\^o formula, integrate, take expectations and then the derivative in time. We obtain
\begin{align*}
\xi_\varepsilon^r(t)
=& \bE\Big[ \| x_0 - \tilde{x}\|^r \Big]
\\
&+ \int_0^t r\bE\Big[ \| X_s^\varepsilon - \tilde{x} \|^{r-2}\Big\langle  X_s^\varepsilon - \tilde{x}, b( X_s^\varepsilon )\Big \rangle\Big] 
+ r \bE\Big[ \| X_s^\varepsilon - \tilde{x} \|^{r-2} \Big\langle X_s^\varepsilon - \tilde{x}, f\ast\mu_s^\varepsilon (X_s^\varepsilon) \Big\rangle\Big] ds
\\
&+\frac{dr(r-1)}{2} \varepsilon \int_0^t \bE\Big[ \| X_s^\varepsilon - \tilde{x} \|^{r-2} \Big] ds
- r \bE\Big[ \int_0^t \| X_s^\varepsilon - \tilde{x}\|^{r-1} \Big\langle X_s^\varepsilon - \tilde{x}, dk_s^\varepsilon \Big\rangle\Big]. 
\end{align*}
Using the uniform strict concavity of $B$, we get
$$
r \int_0^t \bE\Big[ \| X_s^\varepsilon - \tilde{x} \|^{r-2}\Big\langle  X_s^\varepsilon - \tilde{x}, b( X_s^\varepsilon )\Big \rangle\Big] ds
\leq -rL \int_0^t \xi_\varepsilon^r (s) ds.
$$
Next, denoting by $\overline{X_t^\varepsilon}$ an independent version of $X_t^\varepsilon$ and $G$ the concave even polynomial such that $F(x) = G(\| x\|)$, we get
\begin{align*}
r&\int_0^t \bE\Bigg[ \| X_s^\varepsilon - \tilde{x}\|^{r-2} \frac{ G'\big(\| X_s^\varepsilon - \overline{X_s^\varepsilon}\| \big) }{\| X_s^\varepsilon - \overline{X_s^\varepsilon}\| } \Big\langle  X_s^\varepsilon - \overline{X_s^\varepsilon}, X_s^\varepsilon - \tilde{x} \Big\rangle \Bigg]
\\
&= r\int_0^t \bE\Bigg[ \frac{ G'\big(\| X_s^\varepsilon - \overline{X_s^\varepsilon}\| \big) }{\| X_s^\varepsilon - \overline{X_s^\varepsilon}\| } \Big\langle \big( X_s^\varepsilon - \tilde{x}\big) - \big( \overline{X_s^\varepsilon} - \tilde{x}\big) , \big( X_s^\varepsilon - \tilde{x}\big) \| X_s^\varepsilon - \tilde{x}\|^{r-2} \Big\rangle \Bigg] ds
\\
&= \frac{r}{2} \int_0^t \bE\Bigg[ \frac{ G'\big(\| X_s^\varepsilon - \overline{X_s^\varepsilon}\| \big) }{\| X_s^\varepsilon - \overline{X_s^\varepsilon}\| } \Big\langle \big( X_s^\varepsilon - \tilde{x}\big) - \big( \overline{X_s^\varepsilon} - \tilde{x}\big) , \big( X_s^\varepsilon - \tilde{x}\big) \| X_s^\varepsilon - \tilde{x}\|^{r-2} - \big( \overline{X_s^\varepsilon} - \tilde{x}\big) \| \overline{X_s^\varepsilon} - \tilde{x}\|^{r-2} \Big\rangle \Bigg] ds
\\
&\leq 0,
\end{align*}
since by Cauchy–Schwarz inequality, $\forall x,y\in\bR^d$  (see alternatively \cite{HIP}*{Lemma 2.3 (d)})
\begin{align*}
\big\langle x \| x\|^{r-2} - y\| y\|^{r-2}, x - y \big\rangle
\geq
\big( \| x\|^{r-1}-\| y\|^{r-1}\big)\big( \| x\| - \| y\| \big) \geq 0.
\end{align*}

We obtain
$$
\frac{d}{dt} \xi_\varepsilon^r(t) \leq -rL \cdot \xi_\varepsilon^r (t)^{1-\frac{2}{r}} \Big( \xi_\varepsilon^r (t)^{\frac{2}{r}} - \frac{d(r-1)\varepsilon}{2L} \Big) .
$$

Thus we get the bound
$$
|\xi_\varepsilon^r(t)|^{\frac{2}{r}} \leq \max\Big\{ \tfrac{d(r-1)\varepsilon}{2L}, \| x_0 - \tilde{x}\|^2\Big\}. 
$$
Choosing $\varepsilon<\tfrac{\kappa^2 L}{d(r-1)}$, we see $\sup_{t\in \bR^+} |\xi_\varepsilon^r(t)|^{\frac{2}{r}} \leq \max\Big\{ \tfrac{\kappa^2}{2}, \| x_0 - \tilde{x}\|^2\Big\}$. 

Now additionally suppose that $\| x_0 - \tilde{x}\|^2> \tfrac{\kappa^2}{2}$ then we get the upper bound
$$
| \xi_\varepsilon^r(t)|^{\frac{2}{r}} \leq \tfrac{\kappa^2}{2} + \Big( \| x_0 - \tilde{x}\|^2 - \tfrac{\kappa^2}{2}\Big) \exp\Big( -rLt\Big). 
$$
In this case
$$
T^{\kappa, r}(\varepsilon) \leq \tfrac{1}{rL} \log\Big( \tfrac{2\| x_0 - \tilde{x}\| }{\kappa^2} - 1\Big).
$$
Conversely, if $\| x_0 - \tilde{x}\|^2\leq \tfrac{\kappa^2}{2}$ then $T^{\kappa, r}(\varepsilon) = 0$. 
\end{proof}

\subsection{Probability of exiting before converging}
\label{subsec;ExitTime_PoEbC}

Recall that after time $T^{\kappa,r}(\varepsilon)$, the process $X_t^\varepsilon$ is expected to remain close to $\tilde{x}$. Additionally, it also happens that before time $T^{\kappa, r}(\varepsilon)$ and in  the small noise limit the process $X_t^\varepsilon$ does not leave $\fD$. This can be argued from the fact that the dynamical system $\psi_t$ introduced in Assumption \ref{Ass:DomainDofExitTime} stays in the domain $\fD$.

\begin{prop} 
Let $\tau_\fD(\varepsilon)$ be the stopping time as defined in Definition \ref{dfn:ExitTime}. Let $\xi_\varepsilon^r$ and $T^{\kappa, r}(\varepsilon)$ be as defined in Definition \ref{dfn:ExitTime-Moments}. Then for any $\kappa>0$ we have that 
$$
\lim_{\varepsilon\to 0}\bP\Big[ \tau_{\fD}(\varepsilon)<T^{\kappa,r}(\varepsilon) \Big]=0.
$$
\end{prop}

\begin{proof}
Let $t\in\bR^+$. Then, 
\begin{align*}
\bE\Big[ \| X_t^\varepsilon - \psi_t \|^2 \Big]
=&
\varepsilon d t + 2\int_0^t \bE\Big[ \Big\langle X_s^\varepsilon - \psi_s, b( X_s^\varepsilon) - b( \psi_s) \Big\rangle \Big] ds
\\
&+2\int_0^t \bE\Big[ \Big\langle X_s^\varepsilon - \psi_s, f\ast \mu_s^\varepsilon(X_s^\varepsilon) \Big\rangle \Big] ds
-2\int_0^t \bE\Big[ \Big\langle X_s^\varepsilon - \psi_s, dk^\varepsilon_s \Big\rangle \Big] .
\end{align*}
Using standard methods, we get
$$
\bE\Big[ \| X_t^\varepsilon - \psi_t\|^2 \Big] \leq \tfrac{\varepsilon d}{2L}\Big( 1 - \exp\Big(-2Lt\Big) \Big). 
$$
Then, for any $\delta>0$ define
\begin{equation*}
\tau_\delta(\varepsilon):= \inf\Big\{ t>0: \|X_t^\varepsilon - \psi_t\| > \delta \Big\}.
\end{equation*}
Thus for any $T>0$, 
$$
\lim_{\varepsilon\to 0} \bP\Big[ \tau_\delta( \varepsilon )<T \Big]=0.
$$
We are interested in the interval $[0, T^{\kappa,r} (\varepsilon)]$, which depends on $\varepsilon$ but has a uniform bound. Thus by Proposition \ref{prop:ExitTime-moments:1}, 
$$
\bP\Big[ \tau_\delta(\varepsilon) < T^{\kappa,r}(\varepsilon) \Big]
\leq
\bP\Big[\tau_\delta(\varepsilon) < \tfrac{1}{rL} \log\Big( \tfrac{2\| x_0 - \tilde{x}\| }{\kappa^2} - 1\Big) \Big],
$$
which we just argued, goes to $0$ as $\varepsilon \to 0$. 

Finally, from Assumption \ref{Ass:DomainDofExitTime}, we have $\big\{\psi_t\,\,:\,\,t>0\big\}\subset \fD$ and consequently for any $\kappa>0$ we obtain the limit 
$$
\lim_{\varepsilon\to 0}\bP \Big[ \tau_{\fD}(\varepsilon) < T^{\kappa,r}(\varepsilon) \Big]=0. 
$$
\end{proof}

\subsection{The coupling result}
\label{subsec;ExitTime_CR}
Now, we study the exit of the diffusion from the domain after the time $T^{\kappa,r}(\varepsilon)$. To do so, we use the inequality
$$
\sup_{t\geq T^{\kappa,r} (\varepsilon)} \xi_\varepsilon(t) \leq \kappa^{r}, 
$$
which holds for any $\kappa>0$ provided $\varepsilon < \frac{\kappa^2L}{d(r-1)}$. 

From this we deduce that the drift $b( \cdot) + f\ast\mu_t^\varepsilon(\cdot)$ is close to the vector field $b(\cdot) + f(\cdot - \tilde{x})$. Let $\cK\subset \fD$ be a compact set with non-zero Lebesgue measure interior such that $\tilde{x} \in \fD$. We consider the following diffusion defined for $t\geq T^{\kappa,r}(\varepsilon)$ as 
\begin{align}
\label{eq:RestartedSSMVE-Z}
Z_t^\varepsilon =& X_{T^{\kappa,r}(\varepsilon)} + \sqrt{\varepsilon} \big(W_t - W_{T^{\kappa,r}(\varepsilon)} \big) + \int_{T^{\kappa,r}(\varepsilon)}^t b( Z_s^\varepsilon ) ds + \int_{T^{\kappa,r}(\varepsilon)}^t f\big( Z_s^\varepsilon - \tilde{x} \big) ds - k^{Z,\varepsilon}_t,
\\
\nonumber
|k^{Z, \varepsilon}|_t =& \int_{T^{\kappa,r}(\varepsilon)}^t \1_{\partial \cD} (Z_s^\varepsilon) d|k^{Z, \varepsilon}|_s, 
\quad
k^{Z, \varepsilon}_t = \int_{T^{\kappa,r}(\varepsilon)}^t \1_{\partial \cD} (Z_s^\varepsilon) \n(Z_s^\varepsilon) d|k^{Z, \varepsilon}|_s 
\qquad 
\mbox{when $X_{T^{\kappa,r}(\varepsilon)}^\varepsilon \in \cK$}
\\
\nonumber
Z_t^\varepsilon =& X_t^\varepsilon 
\qquad
\mbox{ if $X_{T^{\kappa, r}(\varepsilon)}^\varepsilon \notin \cK$.}
\end{align}
\begin{defn}
\label{dfn:ExitTime-StoppingTime}
Let $\cD$ satisfy Assumption \ref{assumption:domain}. Let $W$ be a $d$-dimensional Brownian motion and let $r>1$, $b$, $f$ $x_0$ and $\tilde{x}$ satisfy Assumption \ref{ass:ExitTime-Coefficients}. Let $\cK$ be a compact set with non-zero Lebesgue measure interior that $\tilde{x}\in \cK$ and $\cK\subset \fD$. Let $X^\varepsilon$ be the solution to Equation \eqref{eq:ExitTime-Process} and let $Z^\varepsilon$ be the solution to \eqref{eq:RestartedSSMVE-Z}.

Define the stopping times
\begin{align*}
\tau_{\cK,\kappa}(\varepsilon):= \inf \Big\{ t > T^{\kappa,r}(\varepsilon) : X_t^\varepsilon \notin \cK \Big\},
\qquad
\tau_{\cK,\kappa}'(\varepsilon):= \inf \Big\{ t > T^{\kappa,r}(\varepsilon) : Z_t^\varepsilon \notin \cK \Big\},
\end{align*}
and $\cT_{\cK,\kappa} (\varepsilon) := \min \Big\{ \tau_{\cK,\kappa} (\varepsilon) ,\tau_{\cK,\kappa}'(\varepsilon) \Big\}$. 
\end{defn}

The following Proposition establishes a coupling between $X^\varepsilon$ the reflected McKean-Vlasov  SDE and $Z^\varepsilon$ the reflected SDE. That is, in the time interval $[T^{\kappa,r}(\varepsilon),\cT_{\cK,\kappa}(\varepsilon)]$ the processes remain close to each other with high probability when the noise is small enough.  

\begin{prop}
\label{prop:ExitTime-poc:uniforme} 
Let $\cT_{\cK, \kappa}$ be as in Definition \ref{dfn:ExitTime-StoppingTime}. Then $\exists \kappa_0>0$ such that $\forall \kappa < \kappa_0$ $\exists \varepsilon_0>0$ such that $\forall \varepsilon< \varepsilon_0$ we have
\begin{equation*}
\bP \left[ \sup_{T^{\kappa,r}(\varepsilon) \leq t \leq \cT_{\cK,\kappa}(\varepsilon)} 
\| Z_t^\varepsilon - X_t^\varepsilon \| \geq \eta(\kappa) \right]
\leq \eta(\kappa), 
\end{equation*}
where $\eta$ is some positive, continuous and increasing function such that $\eta(0)=0$.
\end{prop}
\begin{proof}
Let $t\in \bR^+$. If $X_{T^{\kappa,r}(\varepsilon)} \in \cK$ then, for all $T^{\kappa,r}(\varepsilon) \leq t \leq \cT_{\kappa}(\varepsilon)$, we have
\begin{align*}
\| Z_t^\varepsilon - X_t^\varepsilon \|^2
=
& + 2\int_{T^{\kappa,r}(\varepsilon)}^t \Big\langle Z_s^\varepsilon - X_s^\varepsilon, b( Z_s^\varepsilon) - b( X_s^\varepsilon) \Big\rangle ds
\\
&+ 2\int_{T^{\kappa,r}(\varepsilon)}^t \Big\langle Z_s^\varepsilon - X_s^\varepsilon, f(Z_s^\varepsilon - \tilde{x}) - f\ast \mu_s^\varepsilon (X_s^\varepsilon) \Big\rangle ds
- 2\int_{T^{\kappa,r}(\varepsilon)}^t \Big\langle Z_s^\varepsilon - X_s^\varepsilon, dk_s^{Z, \varepsilon} - dk_s^\varepsilon \Big\rangle. 
\end{align*}
Set  
$$
\eta(\kappa):= \sup_{ \nu \in \bB_{\tilde{x}}^{\kappa,r}} \sup_{x\in\cK} \Big( \frac{\| f \ast \nu(x) - f(x-\tilde{x}) \| }{L} \Big)^{\frac{2}{3}}, 
$$
where $\bB_{\tilde{x}}^{\kappa,r}$ was introduced in Definition \ref{defn:balance}. Using Assumption \ref{assumption:domain} and Gr\"onwall Inequality, we get
$$
\sup_{T^{\kappa, r}(\varepsilon) \leq t \leq \cT_{\cK,\kappa}(\varepsilon)} \| Z_t^\varepsilon - X_t^\varepsilon \|^2 \leq \eta(\kappa)^3
\quad \Rightarrow \quad
\bE\Big[ \sup_{T^{\kappa, r}(\varepsilon) \leq t \leq \cT_{\cK,\kappa}(\varepsilon)} \| Z_t^\varepsilon - X_t^\varepsilon \|^2 \Big]\leq \eta(\kappa)^3.
$$
Appealing to Markov's inequality yields the claim.
\end{proof}

\subsection{The Exit-time result}
\label{subsec;ExitTime_Result}

Let $\tilde{Z}^\varepsilon$ evolve as $Z^\varepsilon$ without reflection, that is for $t\in[T^{\kappa, r}(\varepsilon),\infty)$,
\begin{align*}
\tilde{Z}_t^\varepsilon = X_{T^{\kappa, r}(\varepsilon)} + \sqrt{\varepsilon} \big( W_t - W_{T^{\kappa,r}(\varepsilon)} \big) + \int_{T^{\kappa,r}(\varepsilon)}^t b( \tilde{Z}_s^\varepsilon ) ds + \int_{T^{\kappa,r}(\varepsilon)}^t f (\tilde{Z}_s^\varepsilon - \tilde{x})ds. 
\end{align*}
As the closure of the domain $\fD$ from which the process exits is included into the domain $\cD$ where there is reflection, we remark that $Z_t^\varepsilon = \tilde{Z}_t^\varepsilon$ whilst $t\leq\tau_{\fD}'(\varepsilon)$, where 
$$
\tau_{\fD}'(\varepsilon):= \inf \Big\{t \geq T^{\kappa,r}(\varepsilon): \tilde{Z}_t^\varepsilon \notin \fD \Big\}.
$$
As a consequence, the first exit-time from $\fD$ of the diffusion $\tilde{Z}^\varepsilon$ is the same as the first exit-time from $\fD$ of the diffusion $Z^\varepsilon$. However, the latter exit-time is well understood thanks to the classical Freidlin-Wentzell theory.

The familiar reader will recognise $\Delta$ given as 
\begin{equation*}
\Delta:= \inf_{z\in \partial \fD} \Big\{ B(z) + F(z-\tilde{x}) - B(\tilde{x}) \Big\},
\end{equation*}
to be the exit cost from the domain $\fD$, see \cite{tugaut:tel-01748560}*{Proposition B.4, Item 3}. 
\begin{theorem} 
\label{thm:ExitTime}
Let $\cD$ satisfy Assumption \ref{assumption:domain}. Let $W$ be a $d$-dimensional Brownian motion and let $r>1$, $b$, $f$, $x_0$ and $\tilde{x}$ satisfy Assumption \ref{ass:ExitTime-Coefficients}. Let $X^\varepsilon$ be the solution to Equation \eqref{eq:ExitTime-Process}. Then for all $\delta>0$ the following limit holds 
\begin{equation*}
\lim_{\varepsilon\to 0} 
     \bP\left[ \tfrac{2}{\varepsilon}(\Delta-\delta) 
< \log\Big( \tau_{\fD}(\varepsilon)\Big) 
< \tfrac{2}{\varepsilon}(\Delta+\delta) \right] = 1.
\end{equation*}
\end{theorem}
\begin{proof}
The proof is inspired by \cite{T2011f}, we proceed in a stepwise fashion. 

{\bf Step 1.} Let $\kappa>0$ and we introduce the usual least distance of $x\in\bR^d$ to a (non-empty) set $A\subset \bR^d$ as $d(x;A) := \inf\{ \|x-a\| : a \in  A\}$.  We can prove (by proceeding like in \cite{T2011f}*{Proposition 2.2}) that there exist two families of domains $\left(\fD_{i,\kappa}\right)_{\kappa>0}$ and $\left(\fD_{e,\kappa}\right)_{\kappa>0}$ such that 
\begin{itemize}
\item $\fD_{i,\kappa}\subset \fD\subset \fD_{e,\kappa}$,

\item $\fD_{i,\kappa}$ and $\fD_{e,\kappa}$ are stable by $b(s, \cdot) + f(\cdot -\tilde{x})$, 

\item $\sup_{z\in\partial \fD_{i,\kappa}} {\rm d}\left(z ; \fD^c \right) + \sup_{z\in\partial \fD_{e,\kappa} }{\rm d} \left( z; \fD \right)$ tends to $0$ when $\kappa$ goes to $0$,

\item $\inf_{z\in\partial \fD_{i,\kappa}}{\rm d}\left(z\,;\,\fD^c\right)=\inf_{z\in\partial \fD_{e,\kappa}}{\rm d}\left(z; \fD \right) = r(\kappa)$.
\end{itemize}

\noindent{}{\bf Step 2.} By $\tau_{i,\kappa}'(\varepsilon)$ (resp. $\tau_{e,\kappa}'(\varepsilon)$), we denote the first exit-time of $Z^\varepsilon$ from $\fD_{i,\kappa}$ (resp. $\fD_{e,\kappa}$).
\\
\noindent{}{\bf Step 3.}  We prove here the upper bound:
\begin{align*}
\bP \left[ \tau_{\fD}(\varepsilon) \geq e^{\frac{2(\Delta+\delta)}{\varepsilon}} \right]
&
=
\bP\left[ \tau_{\fD}(\varepsilon) \geq e^{\frac{2(\Delta+\delta)}{\varepsilon}},  \tau_{e,\kappa}'(\varepsilon)\geq e^{\frac{2(\Delta+\delta)}{\varepsilon}}\right]
+\bP\left[ \tau_{\fD}(\varepsilon) \geq e^{\frac{2(\Delta+\delta)}{\varepsilon}}, \tau_{e,\kappa}'(\varepsilon) < e^{\frac{2(\Delta+\delta)}{\varepsilon}} \right]
\\
&
\leq
\bP \left[ \tau'_{e,\kappa}(\varepsilon) \geq e^{\frac{2(\Delta+\delta)}{\varepsilon}} \right]
+\bP \left[ \tau_{\fD}(\varepsilon) \geq e^{\frac{2(\Delta+\delta)}{\varepsilon}} , \tau_{e,\kappa}'(\varepsilon) < e^{\frac{2(\Delta+\delta)}{\varepsilon}} \right]
\\
&=:a_\kappa(\varepsilon) + b_\kappa(\varepsilon).
\end{align*}
{\bf Step 3.1.} By classical results in Freidlin-Wentzell theory, \cite{Herrmann2013StochasticR}*{Theorem 2.42 }, there exists $\kappa_1>0$ such that for all $0<\kappa<\kappa_1$, we have
$$
\lim_{\varepsilon \to 0}\bP \left[ \tau_{e,\kappa}'(\varepsilon) < \exp\left( \frac{2}{\varepsilon} \left(\Delta+\delta\right) \right)\right]=1. 
$$
Therefore, the first term $a_\kappa(\varepsilon)$ tends to $0$ as $\varepsilon$ goes to $0$.

\noindent{}{\bf Step 3.2.} For $\kappa$ sufficiently small, we have $\fD_{e,\kappa} \subset \cK$ and consequently we have 
\begin{align*}
&\bP \Big[ \tau_{\fD}(\varepsilon) \geq e^{\frac{2(\Delta+\delta)}{\varepsilon}} , \tau_{e,\kappa}'(\varepsilon) \leq e^{\frac{2(\Delta+\delta)}{\varepsilon}} \Big]
\\
& \qquad 
\leq \bP \Big[ \| X_{\tau_{e,\kappa}'(\varepsilon)} - Z_{\tau_{e,\kappa}'(\varepsilon)} \|\geq \eta(\kappa) \Big]
\leq
\bP \Big[ \sup_{T^{\kappa,r} (\varepsilon) \leq t\leq \cT_{\cK,\kappa}(\varepsilon)} \| X_{t}^\varepsilon - Z_{t}^\varepsilon \|\geq \eta(\kappa) \Big].
\end{align*}

According to Proposition \ref{prop:ExitTime-poc:uniforme}, there exists $\varepsilon_0>0$ such that the previous term is less than $\eta(\kappa)$ for all $\varepsilon<\varepsilon_0$.
\\
{\bf Step 3.3.} Let $\delta>0$. By taking $\kappa$ arbitrarily small, we obtain the upper bound 
$$
\lim_{\varepsilon \to 0}\bP \left[ \tau_{\fD} (\varepsilon) \geq \exp \left( \frac{2(\Delta+\delta)}{\varepsilon} \right) \right] =0.
$$

\noindent{\bf Step 4.} Analogous arguments show that $\lim_{\varepsilon \to 0} \bP \left[ T^{\kappa, r} (\varepsilon) \leq \tau_{\fD}(\varepsilon) \leq e^{\frac{2(\Delta-\delta)}{\varepsilon}} \right]=0$. However, by Proposition \ref{subsec;ExitTime_PoEbC} we have $\lim_{\varepsilon\to 0}\bP\left[ \tau_{\fD}(\varepsilon) \leq T^{\kappa,r}(\varepsilon) \right]=0$. 

This concludes the proof.
\end{proof}

\appendix
\section{Appendix}

\begin{lemma}\label{lemma : adapted D+Z lemma 5.6.18}
Let $z_0\in \mathbb{R}^d$ be deterministic. For $t\geq 0$, let $b_t\in \mathbb{R}^d$, $\sigma_t\in \mathbb{R}^{d\times d'}$,$k_{t}\in \mathbb{R}^d$ be progressively measurable processes, with $k$ having bounded variation. Let $Z_t$ be the solution of
\begin{equation*}
Z_t=z_0+\int_0^t b_s ds + \sqrt{\varepsilon}\int_0^t \sigma_s dW_s+ k_t,
\end{equation*}
where $k$ is such that 
\begin{equation}\label{e.appen. d+z lem reflection assumption}
\int_0^t \langle Z_s , dk_{s}\rangle \leq 0
\quad \textrm{a.s.~ for all $t \geq 0$.}
\end{equation}
Further assume that $\tau_1 \in [0,T]$ is a stopping time with respect the filtration generated by $\{W_t ~:t\in[0,T]\,\}$, and that 
\begin{align}
\|b_t\| \leq & B(\rho^2+\|Z_t\|^2)^{\frac{1}{2}}
\label{e.appen.d+z assumpbound 1}
\qquad\textrm{and}\qquad
\|\sigma_t\|  \leq 
M(\rho^2+\|Z_t\|^2)^{\frac{1}{2}}, 
\end{align}
for some constants $M,B,\rho$. Then for any $\delta>0$, $\varepsilon <1$
\begin{equation}
\varepsilon \log\Big( \mathbb{P}(\sup_{t\in [0,\tau_1]}\|Z_t\|)\geq \delta) \Big)\leq 2B+M^2\Big(2+d\Big)+\log \Big( \frac{\rho^2+\|z_0\|^2}{\rho^2+\delta^2} \Big).
\end{equation}
\end{lemma}
\begin{proof} The proof is a
slight adaptation of \cite{DZ}*{Lemma 5.6.18}. Let $\varepsilon <1$. Define 
$U_{t} =\phi(Z_t)=(\rho^2+\|Z_t\|^2)^{\frac{1}{\varepsilon}}$, and note $\nabla \phi (Z_t)= \frac{2\phi(Z_t)}{\varepsilon(\rho^2+\|Z_t\|^2)}Z_t$. By It\^o we have 
\begin{equation}
U_t= \phi(z_0)+\int_0^t \tilde{b}_s ds + \int_0^t \tilde{\sigma}_s dW_s + \int_0^t \langle \nabla \phi(Z_s), \alpha_{s} \rangle d|k|_{s},
\end{equation}
where 
\begin{align*}
    \tilde{\sigma}_t:=&\sqrt{\varepsilon}\nabla \phi(Z_t)' \sigma_t
    \quad \textrm{and}\quad
    \tilde{b}_t
    :=
    \sqrt{\varepsilon}\nabla \phi(Z_t)' b_t+\frac{\varepsilon}{2}\text{Trace}\big[\sigma_t \nabla^2 \phi(Z_t)\sigma'_t \big].
\end{align*}
Note that for $t\in[0,\tau_1]$ we have, 
\begin{align*}
\|\nabla \phi(Z_t)'b_t\|
\leq&
\frac{2B\phi(Z_t)}{\varepsilon(\|Z_t\|^2)^{\frac{1}{2}}}\|Z_t\|
=
\frac{2B U_t}{\varepsilon},
\end{align*}
and
\begin{align}
\nonumber
\frac{\varepsilon}{2}\text{Trace}\big[ \sigma_t \nabla^2 \phi(Z_t) \sigma'_t  \big]
&
\leq
\frac{\varepsilon}{2}\|\sigma\|^2\|\nabla^2 \phi(Z_t)\|
\\
\label{verylocalAUXILIARYlabel}
& \leq \frac{\varepsilon}{2}M^2(\rho^2+\|Z_t\|^2) \|\nabla^2 \phi(Z_t)\| 
\leq \frac{M^2(d+2)U_t}{\varepsilon},
\end{align}
indeed we can directly compute and decompose
\begin{align*}
\nabla^2\phi(Z_t)
=
\frac2\varepsilon \frac{\phi(Z_t)}{(\rho^2+\|Z_t\|^2)}I_d
+
2\Big(\frac{1}{\varepsilon}-1\Big) \frac{2}{\varepsilon}   \frac{\phi(Z_t)}{(\rho^2+\|Z_t\|^2)^2}Z_t Z_t'
=
A I_d +B (I_d Z_t) (I_d Z_t)',
\end{align*}
with $A$ and $B$ two auxiliary variables representing the coefficients of $I_d$ and $(I_d Z_t) (I_d Z_t)'$, for $Z_t \in \bR^{d}$, $Z_t Z_t'\in \bR^{d\times d}$ and $I_d$ the $d$-dimensional identity matrix. Hence 
\begin{align*}
\|\nabla^2 \phi(Z_t)\|
\leq A\cdot d  +B \|Z_t\|^2 
&
=
\frac{2}{\varepsilon} \frac{\phi(Z_t)}{\rho^2+\|Z_t\|^2} \Big( d  \frac{\phi(Z_t)}{\rho^2+\|Z_t\|^2} \Big)
+
\frac{4}{\varepsilon} \Big(\frac1\varepsilon-1\Big) \frac{\phi(Z_t)}{\rho^2+\|Z_t\|^2}
\frac{\|Z_t\|^2}{\rho^2+\|Z_t\|^2} 
\\
&
\leq 
\Big[\ \frac{2d}\varepsilon 
+
\frac{4}{\varepsilon^2}  
\Big]
 \frac{U_t}{\rho^2+\|Z_t\|^2},
\end{align*}
using this result on the 1st term in  \eqref{verylocalAUXILIARYlabel}, yields the result.  

Hence for any $t\in[0,\tau_1]$ we have  
\begin{equation}\label{e,appen,d+z bound}
\tilde{b}_t\leq \frac{K U_t}{\varepsilon}
\quad \textrm{with $K=2B+M^2(d+2)< \infty $.}
\end{equation}
Fix $\delta>0$, define the stopping time $\tau_2=\inf\{t\geq0: \|Z_t\|\geq \delta \}\wedge \tau_1$. Let $t\in[0,\tau_2]$, note that 
\begin{align*}
\|\tilde{\sigma}_t\|
\leq 
\|\nabla \phi(Z_t)\|\|\sigma_t\|  
& 
\leq \frac{2M}{\varepsilon}\frac{(\rho^2+\|Z_t\|^2)^{\frac{1}{\varepsilon}}}{(\rho^2+\|Z_t\|^2)^\frac{1}{2}}\|Z_t\|
\leq 
\frac{\sqrt{2}M}{\sqrt{\rho}\varepsilon}\frac{(\rho^2+\|Z_t\|^2)^{\frac{1}{\varepsilon}}}{\|Z_t\|^{\frac{1}{2}}}\|Z_t\|
\leq 
\frac{\sqrt{2}M}{\sqrt{\rho}\varepsilon}(\rho^2+\delta^2)^{\frac{1}{\varepsilon}}\delta^{\frac{1}{2}},
\end{align*}
in other words $\|\tilde{\sigma}\|$ is uniformly bounded on $[0,\tau_2]$. Hence for $t\in [0,\tau_2]$
\begin{equation*}
\int_0^t \tilde{\sigma}_s dW_s= U_t-\int_0^t \tilde{b}_s ds-\int_0^t \langle \nabla \phi(Z_s),  dk_{s} \rangle, 
\end{equation*}
is a Martingale. Therefore Doob's theorem implies
\begin{align*}
\EE[ U_{t\wedge \tau_2}] =& \phi(z_0)+  \EE\Big[ \int_0^{t\wedge \tau_2 } \tilde{b}_s ds \Big]+\EE\Big[ \int_0^{t\wedge \tau_2} \langle \nabla \phi(Z_s),  dk_{s} \rangle \Big].
\end{align*}
Non-negativity of $U$ and \eqref{e.appen.d+z assumpbound 1}, and \eqref{e.appen. d+z lem reflection assumption} imply that
\begin{align*}
\EE[U_{t\wedge \tau_2}]\leq& \phi(z_0)+\frac{K}{\varepsilon}\EE\Big[\int_0^{t\wedge \tau_2}U_s ds\Big].
\end{align*}
From here one can conclude by proceeding identically to \cite{DZ}*{Lemma 5.6.18}.
\end{proof}

\section{Additional Existence \& Uniqueness results}
\label{appendixB}

\begin{theorem}
\label{thm:ExistUnique-Lip-Ref}
Let $\cD$ satisfy Assumption \ref{assumption:domain}. Let $p\geq 2$. Let $W$ be a $d'$ dimensional Brownian motion. Let $\theta:\Omega \to \cD$, $b:[0,T] \times \Omega \times \cD \to \bR^d$ and $\sigma:[0,T] \times \Omega \times \cD \to \bR^{d\times d'}$ be progressively measurable maps. Suppose that 
\begin{itemize}
\item $\theta \in L^p( \cF_0, \bP; \cD)$.

\item $\exists x_0 \in \cD$ such that $b$ and $\sigma$ satisfy the integrability conditions
$$
\bE\Big[ \Big( \int_0^T \|b(s, x_0)\| ds \Big)^p \Big]
\vee
\bE\Big[ \Big( \int_0^T \|\sigma(s, x_0)\|^2 ds \Big)^{p/2} \Big] < \infty.  
$$

\item $b$ and $\sigma$ satisfy a Lipschitz condition over $\cD$, $\exists L>0$ such that for almost all $(s, \omega)\in [0,T] \times \Omega$ and $\forall x, y\in \cD$, 
$$
\| b(s, x) - b(s, y)\| \vee \| \sigma(s, x) - \sigma(s, y) \| \leq L \| x - y \|. 
$$
\end{itemize}
Then there exists a unique solution to the reflected Stochastic Differential Equation \eqref{eq:reflectedSDE} in $\cS^p([0,T])$ and
$$
\bE\Big[ \| X - x_0\|_{\infty, [0,T]}^p \Big] \lesssim \bE\Big[ \|\theta - x_0\|^p\Big] + \bE\Big[ \Big( \int_0^T \|b(s, x_0)\| ds \Big)^p \Big] + \bE\Big[ \Big( \int_0^T \| \sigma(s, x_0)\|^2 ds \Big)^{p/2} \Big]. 
$$ 
\end{theorem}

\begin{proof}
Let $n\in \bN$, and for clarity we emphasise this is distinct from $\n$ as defined in Definition \ref{dfn:Skorokhodproblem}. We consider the following sequence of random processes defined recursively over the interval $[0,T]$:
\begin{itemize}
\item $X^{(0)} = \theta$,
\item $Y^{(n+1)}_t:= \theta + \int_0^t b(s, X_s^{(n)} ) ds + \int_0^t \sigma(s, X_s^{(n)}) dW_s$,
\item $(X^{(n)}, k^n)$ is the solution to the Skorokhod problem $(Y^{(n)}, \cD, \n)$.
\end{itemize}
The solution to the Skorokhod problem $(X^{(n+1)}, k^n)$ exists $\bP$-almost surely by Theorem \ref{thm:SkorokhodProblem} since the process $Y^{(n)}$ is a semi-martingale. By taking an intersection of the sequence of $\bP$-measure-$1$ sets, we obtain a $\bP$-measure-$1$ set on which all such Skorokhod problems are solvable. 

Thus $X^{(n+1)}$ is the recursively defined It\^o process
\begin{align*}
X^{(n+1)}_t =& \theta + \int_0^t b(s, X^{(n)}_s) ds + \int_0^t \sigma(s, X^{(n)}_s) dW_s - k_t^n,
\\
|k^n|_t =& \int_0^t \1_{\partial \cD} (X^{(n+1)}_s) d|k^n|_s
\quad
k_t^n = \int_0^t \1_{\partial \cD} (X^{(n+1)}_s) \n(X^{(n+1)}) d|k^n|_s.
\end{align*}
It is immediate that $X^{(0)} \in \cS^p([0,T])$. Now suppose that $X^{(n)} \in \cS^p([0,T])$. 

Next, we show that this sequence of Picard iterations converges. Firstly, 
\begin{align*}
X_t^{(1)} - X_t^{(0)} = \int_0^t b(s, \theta) ds + \int_0^t \sigma(s, \theta) dW_s - k^0_t,
\end{align*}
and hence 
$\bE\Big[ \| X_t^{(1)} - \theta\|_{\infty, [0,T]}^p \Big] 
\leq
\bE\Big[ \Big( \int_0^T |b(s, \theta)| ds \Big)^p \Big] + \bE\Big[ \Big( \int_0^T |\sigma(s, \theta)|^2 ds\Big)^{p/2} \Big] . 
$

Next consider
\begin{align*}
\| X_t^{(n+1)}& - X_t^{(n)} \|^p 
\\
=& p\int_0^t \| X_s^{(n+1)} - X_s^{(n)} \|^{p-2} \Big\langle X_s^{(n+1)} - X_s^{(n)} , b(s, X_s^{(n)}) - b(s, X_s^{(n-1)}) \Big\rangle ds
\\
&+ p\int_0^t \| X_s^{(n+1)} - X_s^{(n)} \|^{p-2} \Big\langle X_s^{(n+1)} - X_s^{(n)}, \Big(\sigma(s, X_s^{(n)}) - \sigma(s, X_s^{(n-1)}) \Big) dW_s \Big\rangle
\\
&+ \tfrac{p}{2} \int_0^t \| X_s^{(n+1)} - X_s^{(n)} \|^{p-2} \Big\| \sigma(s, X_s^{(n)}) - \sigma(s, X_s^{(n-1)} ) \Big\|^2 ds
\\
&+ \tfrac{p(p-2)}{2} \int_0^t \| X_s^{(n+1)} - X_s^{(n)} \|^{p-4} \Big\| (X_s^{(n+1)} - X_s^{(n)})' \Big( \sigma(s, X_s^{(n)}) - \sigma(s, X_s^{(n-1)})\Big) \Big\|^2 ds
\\
&-p\int_0^t  \| X_s^{(n+1)} - X_s^{(n)} \|^{p-2}  \Big\langle X^{(n+1)}_s-X^{(n)}_s,\n(X^{(n)}_s)d|k^{n}|_s-\n(X^{(n-1)}_s)d|k^{n-1}|_s \Big\rangle .
\end{align*}
Taking a supremum over the time interval $[0,T]$ and taking expectations yields
\begin{align*}
\bE\Big[ \|X^{(n+1)} - X^{(n)} \|_{\infty, [0,T]}^p \Big] 
\leq& pL \bE\Big[ \| X^{(n+1)} - X^{(n)}\|_{\infty, [0,T]}^{p-1} \int_0^T \| X^{(n)} - X^{(n-1)}\|_{\infty, [0,s]} ds \Big]
\\
& + pC_1L \bE\Big[ \| X^{(n+1)} - X^{(n)}\|_{\infty, [0,T]}^{p-1}  \Big(\int_0^T \| X^{(n)} - X^{(n-1)}\|_{\infty, [0,s]}^2  ds\Big)^{1/2} \Big]
\\
& + \tfrac{p(p-1)L^2}{2} \bE\Big[ \| X^{(n+1)} - X^{(n)}\|_{\infty, [0,T]}^{p-2} \int_0^T \| X^{(n)} - X^{(n-1)}\|_{\infty, [0,s]}^2 ds \Big],
\end{align*}
where the final term was dominated by $0$ using Lemma \ref{lem:NormalToDomain}. An application of Young's Inequality yields 
\begin{align}
\nonumber
\bE\Big[ \|X^{(n+1)} - X^{(n)} \|_{\infty, [0,T]}^p \Big] 
\leq& (p-1)^{p-1} \big(4L\big)^{p} T^{p-1} \int_0^T \bE\Big[ \|X^{(n)} - X^{(n-1)} \|_{\infty, [0,s]}^p \Big] ds
\\
\nonumber
&+(p-1)^{p-1} \big(4LC_1\big)^p T^{(p-2)/2} \int_0^T \bE\Big[ \|X^{(n)} - X^{(n-1)} \|_{\infty, [0,s]}^p \Big] ds
\\
\nonumber
&+ 2(p-1)^{p/2} (p-2)^{(p-2)/2} 4^{p/2} T^{(p-2)/2} \int_0^T \bE\Big[ \|X^{(n)} - X^{(n-1)} \|_{\infty, [0,s]}^p \Big] ds
\\
\label{eq:aux-estimation}
\leq&K\int_0^T \bE\Big[ \|X^{(n)} - X^{(n-1)} \|_{\infty, [0,s]}^p \Big] ds. 
\end{align}
Therefore, by inductively substituting in for preceding terms of the sequence and integrating, we get

$$
\bE\Big[ \|X^{(n+1)} - X^{(n)} \|_{\infty, [0,T]}^p \Big] \leq \frac{K^n}{n!} T^n \bE\Big[ \| X^{(1)} - \theta\|_{\infty, [0,T]}^p\Big].
$$
Thus
$$
\bE\Big[ \| X^{(n)} - \theta\|_{\infty, [0,T]}^p \Big] 
\leq 
\bE\Big[ \|\theta\|^p \Big] + \sum_{i=1}^n \bE\Big[ \|X^{(i)} - X^{(i-1)} \|_{\infty, [0,T]}^p \Big] < \bE\Big[ \|\theta\|^p \Big] + \bE\Big[ \| X^{(1)} - \theta\|_{\infty, [0,T]}^p \Big] e^{KT}. 
$$
Therefore, there exists a limit of the sequence of random variables $X^{(n)}$ in the Banach space $\cS^p([0,T])$. 

Further, by Chebyshev's inequality we have 
$$ 
\bP\Big[ \Big\{ \| X^{(n+1)} - X^{(n)}\|_{\infty, [0,T]} >2^{-n} \Big\} \Big] 
\leq
\frac{(2K)^n}{n!},
$$
so that by the Borel-Cantelli lemma
$$
\bP\Big[ \limsup_{n\to \infty} \Big\{ \| X^{(n+1)} - X^{(n)}\|_{\infty, [0,T]} >2^{-n} \Big\} \Big]=0,
$$
so that the limit of the $X^{(n)}$ exists $\bP$-almost surely. Denote this limit by the stochastic process $X$. 

Finally, let
$$
Y_t = \theta + \int_0^t b(s, X_s) ds + \int_0^t \sigma(s, X_s) dW_s,
$$
and let $(Z, k)$ be the solution to the Skorokhod problem $(Y, \cD, \n)$. Thus $Z$ satisfies the SDE
\begin{align}
\label{eq:thm:ExistUnique-Lip-Ref1.1}
Z_t =& \theta + \int_0^t b(s, X_s) ds + \int_0^t \sigma(s, X_s) dW_s - k_t,
\\
\nonumber
|k|_t =& \int_0^t \1_{\partial \cD} (Z_s) d|k|_s, 
\quad 
k_t = \int_0^t \1_{\partial \cD} (Z_s) \n(Z_s) d|k|_s. 
\end{align}
By similar estimates and Lemma \ref{lem:NormalToDomain} we show, as $n\to \infty$, that 
$\bE[\, \| X^{(n)} - Z \|_\infty^p ] \to 0$. We know that $X$ is the unique limit of the random processes $X^{(n)}$, so $X$ must satisfy the stochastic differential equation \eqref{eq:thm:ExistUnique-Lip-Ref1.1}.

In light of the estimates above, uniqueness follows trivially and we sketch only the core argument. Assume $X,Y$ are two solution to \eqref{eq:reflectedSDE}, then estimating $\bE[\, \|X - Y \|_{\infty, [0,T]}^p ] $ as in \eqref{eq:aux-estimation} leads to an inequality where Gr\"onwall's inequality can be directly applied to yield $\bE[\, \|X - Y \|_{\infty, [0,T]}^p ]=0$ and hence delivering uniqueness.
\end{proof}

\begin{proof}[Proof of Theorem \ref{thm:ExistUnique-LocLip-MVE}]
Let $n\in \bN$. Define the drift term
$$
b_n(s, x) \coloneqq \begin{cases}
b(s, x), & \text{if $x\in \cD_n$},
\\
b\Big( s, \mbox{arg min}_{y\in \cD_n} \| x - y\| \Big), & \text{if $x \notin \cD_n$}. 
\end{cases}
$$

By the local Lipschitz condition of $b$, we have that $b_n$ is a uniformly Lipschitz function. By Theorem \ref{thm:ExistUnique-Lip-Ref}, we know that for each $n\in \bN$, there exists a unique solution to the SDE 
\begin{align*}
X_t^n =& \theta + \int_0^t b_n(s, X_s^n) ds + \int_0^t \sigma(s, X_s^n) dW_s - k_t^n,
\end{align*}
with $|k^n|_t= \int_{0}^{t} \1_{ \partial\cD }(X_s^n) ds$ and $k^n_t= \int_{0}^{t} \1_{ \partial\cD}(X_s^n)\n (X_s^n) d|k^n|_s$ over the interval $[0,T]$. Next, define the sequence of stopping times 
$
\tau_n:= \inf\{ t\in [0,T]: X_t \notin \cD_n\},
$ 
and $\tau_\infty:= \lim_{n\to \infty} \tau_n$. Observe that on the interval $[0,\tau_n]$, we have $b_n( s, X_s^n) = b(s, X_s^n)$. Thus we can equivalently write that on the interval $[0,\tau_n]$ that
$$
X_t^n = \theta + \int_0^t b(s, X_s^n) ds + \int_0^t \sigma(s, X_s^n) dW_s - k_t^n,
$$
and so $X_t = X_t^n$. Applying the one-sided Lipschitz condition, we have
\begin{align*}
\bE\Big[  \|X - x_0 \|_{\infty, [0,T\wedge \tau_n]}^p \Big] 
\lesssim&
\bE\Big[ \| \theta-x_0 \|^p\Big] + \bE\Bigg[ \Big( \int_0^{T\wedge \tau_n} \| b(s, x_0) \| ds \Big)^p \Bigg] + \bE\Bigg[ \Big( \int_0^{T \wedge \tau_n} \| \sigma(s, x_0) \|^2 ds \Big)^{p/2} \Bigg]
\\
\lesssim&
\bE\Big[ \| \theta-x_0 \|^p\Big] + \bE\Bigg[ \Big( \int_0^T \| b(s, x_0) \| ds \Big)^p \Bigg] + \bE\Bigg[ \Big( \int_0^{T} \| \sigma(s, x_0) \|^2 ds \Big)^{p/2} \Bigg]. 
\end{align*}
As each $\tau_n<\tau_{n+1}$, we have that the sequence of random variables satisfies
$ \| X - x_0 \|_{\infty, [0,T\wedge \tau_n]} \leq \| X - x_0 \|_{\infty, [0,T\wedge \tau_{n+1}]}$, so we apply Beppo Levi to conclude that
$$
\bE\Big[ \| X - x_0\|_{\infty, [0,T\wedge \tau_\infty]}^p \Big] 
\lesssim 
\bE\Big[ \| \theta - x_0 \|^p \Big] + \bE\Bigg[ \Big( \int_0^T \| b(s, x_0) \| ds \Big)^p \Bigg] + \bE\Bigg[ \Big( \int_0^{T} \| \sigma(s, x_0) \|^2 ds \Big)^{p/2} \Bigg]. 
$$
Note that the probability 
\begin{align*}
\bP\Big[ \tau_n <T\Big] 
= \bP\Big[ \|X^n - x_0\|_{\infty, [0,T]} \geq n \Big] 
\leq \bP\Big[ \| X - x_0\|_{\infty, [0,T\wedge \tau_\infty]} \geq n \Big] 
\leq \frac{1 }{n^p}\bE\big[\, \| X - x_0 \|_{\infty, [0,T\wedge \tau_\infty]}^p \big]. 
\end{align*}
Thus by the Borel Cantelli lemma, 
$$
\bP\big[ \limsup_{n\to \infty} \big\{ \tau_n < T\big\} \big]= 0. 
$$
\end{proof}

By the Cauchy-Schwarz inequality and the polynomial growth of $f$, we obtain
\begin{align*}
\tfrac{1}{N}\sum_{j=1}^N\bE\Big[& \Big\langle X_s^{i, N} - X_s^i,  f(X_s^i - X_s^j) - f\ast \mu_s(X_s^i) \Big\rangle \Big]
\\
&\leq  C \bE\Big[ \|X_s^{i,N} - X_s^i\|^2\Big]^{1/2}\Big(1+\bE\Big[\|X^i_s\|^{2r}\Big]\Big)^{1/2} 
\end{align*}

\bibliographystyle{abbrv}

\end{document}